\documentclass[12pt]{amsart}

\usepackage{calligra,mathrsfs}
\usepackage[all]{xy}
\usepackage{float, comment}
\usepackage{mathtools}
\usepackage{amsmath}
\usepackage{amsthm}
\usepackage{amssymb}
\usepackage{amsbsy}
\usepackage{amstext}
\usepackage{amsopn}
\usepackage[mathscr]{eucal}
\usepackage{enumerate}
\usepackage{xcolor}
\usepackage{graphicx} 
\usepackage{scalerel}
\usepackage{microtype} 
\usepackage[margin=1in,marginparwidth=0.8in, marginparsep=0.1in]{geometry}
\usepackage[bookmarks=true, bookmarksopen=true, bookmarksdepth=3,bookmarksopenlevel=2, colorlinks=true, linkcolor=blue, citecolor=blue, filecolor=blue, menucolor=blue, urlcolor=blue]{hyperref}
\usepackage{tikz}
\usepackage{bbm}
\usepackage[all]{xy}
\usepackage{xspace}
\usetikzlibrary{3d,arrows,calc,positioning,decorations.pathreplacing,matrix} 

\tikzset{dbl/.style={double,
		double distance=6.0,
		-implies,
		shorten >=10pt,
		shorten <=10pt}}

\DeclareMathOperator{\cHom}{\mathscr{H}\text{\kern -3pt {\calligra\large om}}\,}

\numberwithin{equation}{section}
\newtheorem{Theorem}[equation]{Theorem}
\newtheorem{Proposition}[equation]{Proposition} 
\newtheorem{Lemma}[equation]{Lemma}

\newtheorem{Corollary}[equation]{Corollary}
\newtheorem{Conjecture}[equation]{Conjecture}

\theoremstyle{definition}
\newtheorem{Remark}[equation]{Remark}

\newtheorem{Definition}[equation]{Definition}
\newtheorem{Construction}[equation]{Construction} 

\numberwithin{figure}{section}

\def\hz{\widehat{z}}

\def\H{{\widehat{G}}}
\def\cl{{\mathrm{cl}}}

\def\td{{d}}
\def\rd{{\widetilde{d}}}

\def\tm{{\widetilde{m}}}

\def\Spec{{\rm{Spec}}\,}

\def\ux{{\underline{x}}}

\def\la{\langle}
\def\ra{\rangle}

\def\l{\lambda}

\newcommand{\ttimes}{\mathbin{\widetilde{\times}}}
\newcommand{\tbox}{\mathbin{\widetilde{\boxtimes}}}

\def\O{\mathcal{O}}
\def\K{\mathcal{K}}

\newcommand{\fg}{\mathfrak{g}}

\newcommand{\A}{\mathbb{A}}
\newcommand{\C}{\mathbb{C}}
\newcommand{\D}{\mathbb{D}}

\renewcommand{\P}{\mathbb{P}}

\newcommand{\Z}{\mathbb{Z}}

\newcommand{\cE}{\mathcal{E}}
\newcommand{\cF}{\mathcal{F}}
\newcommand{\cG}{\mathcal{G}}
\newcommand{\cH}{\mathcal{H}}
\newcommand{\cI}{\mathcal{I}}
\newcommand{\cK}{\mathcal{K}}

\newcommand{\cM}{\mathcal{M}}
\newcommand{\cN}{\mathcal{N}}
\newcommand{\cO}{\mathcal{O}}
\newcommand{\cP}{\mathcal{P}}
\newcommand{\cQ}{\mathcal{Q}}

\newcommand{\cT}{\mathcal{T}}

\newcommand{\cV}{\mathcal{V}}
\newcommand{\cW}{\mathcal{W}}

\newcommand{\stable}{stable\xspace}

\newcommand{\tcF}{\widetilde{\mathcal{F}}}

\newcommand{\Catinfty}{\mathrm{Cat}_{\infty}}

\newcommand{\catC}{{\mathscr{C}}}
\newcommand{\catD}{{\mathscr{D}}}

\newcommand{\Kzl}{K}
\newcommand{\hs}{\heartsuit}

\newcommand{\al}{\alpha}
\newcommand{\be}{\beta}

\newcommand{\La}{\Lambda}

\newcommand{\ot}{\otimes}
\newcommand{\wt}{\widetilde}

\newcommand{\convspace}{\mathcal{S}}
\newcommand{\conv}{{\mathbin{\scalebox{1.1}{$\mspace{1.5mu}*\mspace{1.5mu}$}}}}

\newcommand{\into}{\hookrightarrow}
\newcommand{\onto}{\twoheadrightarrow}

\newcommand{\id}{{id}}

\newcommand{\N}{{N}}

\newcommand{\T}{\mathcal{T}}

\newcommand{\Gr}{\mathrm{Gr}}

\newcommand{\hGO}{{\widehat{G}_{\cO}}}
\newcommand{\hGOX}{{\widehat{G}_{\cO,X^I}}}

\newcommand{\hGK}{\widehat{G}_{\cK}}
\newcommand{\NO}{N_{\cO}}
\newcommand{\NK}{N_{\cK}}
\newcommand{\tpr}{u}

\newcommand{\hR}{{\mathcal{R}}}

\newcommand{\QCoh}{\mathrm{QCoh}}
\newcommand{\Coh}{\mathrm{Coh}}
\newcommand{\IndCoh}{\mathrm{IndCoh}}
\newcommand{\IC}{\mathrm{IC}}

\newcommand{\Rep}{\mathrm{Rep}}

\newcommand{\Sym}{\mathrm{Sym}}

\newcommand{\Maps}{\mathrm{Maps}}
\newcommand{\Map}{\mathrm{Map}}

\newcommand{\indGStkk}{\mathrm{indGStk}}

\newcommand{\indGStkkadtm}{\mathrm{indGStk}^{tm,ad}}

\newcommand{\Ind}{\mathrm{Ind}}
\newcommand{\CAlg}{\mathrm{CAlg}}
\newcommand{\Alg}{\mathrm{Alg}}

\newcommand{\PrL}{\mathcal{P}\mathrm{r}^{\mathrm{L}}}

\newcommand{\Corr}{\mathrm{Corr}}

\newcommand{\pt}{\mathrm{pt}}

\newcommand{\li}{1}
\newcommand{\mi}{2}
\newcommand{\ri}{3}

\newcommand{\curv}{X}

\newcommand{\loops}[2]{{#1}_{\cK,\curv^{#2}}}
\newcommand{\jets}[2]{{#1}_{\cO,\curv^{#2}}}

\newcommand{\KPcohGN}{{\mathcal{KP}}_{G,N}}
\newcommand{\KPcohGNeta}{{\mathcal{KP}}^\eta_{G,N}}

\newcommand{\KPGLn}{{\mathcal{KP}_n}}
\newcommand{\KPGLtwo}{{\mathcal{KP}_2}}
\newcommand{\TGN}{{{T}}_{G,N}}
\newcommand{\LGN}{L_{G,N}}

\newcommand{\KP}[2]{{\mathcal{KP}}_{#1,#2}}

\newcommand{\Theory}[1]{{{T}}_{#1}}
\newcommand{\CB}[1]{\cM_{#1}}

\newcommand{\StkC}{\mathrm{Stk}_\C}

\newcommand{\CAlgC}{\mathrm{CAlg}_\C}

\newcommand{\tmdefreasalg}{Def. 3.2}
\newcommand{\tmsubsecgeomtame}{Sec. 4.2}
\newcommand{\tmdefreasaffmorph}{Def. 4.6}
\newcommand{\tmproptamequotients}{Prop. 4.11}
\newcommand{\tmdefgeomcohpull}{Def. 4.12}
\newcommand{\tmdefweaklysmooth}{Def. 4.16}
\newcommand{\tmpropwprosmoothcohdiagb}{Prop. 4.18}
\newcommand{\tmsecindtamemorphisms}{Sec. 5.2}
\newcommand{\tmdefindP}{Def. 5.8}
\newcommand{\tmproptamefiltcolims}{Prop. 5.10}
\newcommand{\tmdefindgeomcohpull}{Def. 5.12}
\newcommand{\tmpropindPbaseprops}{Prop. 5.15}
\newcommand{\tmdefcohonindgstks}{Def. 5.19}
\newcommand{\tmpropupshriekupstargenICcaseIG}{Prop. 6.17}

\newcommand{\tmsecindextcohcase}{Sec. 7.1}
\newcommand{\tmpropeFXRupperstarcohpullindcase}{Prop. 7.16}

\newcommand{\igdefindgeomstack}{Def. 4.1}

\newcommand{\igpropindPtwoofthreeprops}{Prop. 4.14}

\newcommand{\igpropupshrieklowstargeom}{Prop. 6.8}

\newcommand{\igpropeFXfunctoriality}{Prop. 7.26}
\newcommand{\igpropeFXuppershriekcompatindgeom}{Prop. 7.35}

\newcommand{\igpropeFXRlowerstar}{Prop. 7.48}

\newcommand{\igpropeFXReFXindgeomcoh}{Prop. 7.54}

\DeclareFontFamily{U}{mathx}{\hyphenchar\font45}
\DeclareFontShape{U}{mathx}{m}{n}{
	<5> <6> <7> <8> <9> <10>
	<10.95> <12> <14.4> <17.28> <20.74> <24.88>
	mathx10
}{}
\DeclareSymbolFont{mathx}{U}{mathx}{m}{n}
\DeclareFontSubstitution{U}{mathx}{m}{n}
\DeclareMathAccent{\widecheck}{0}{mathx}{"71}

\DeclareMathSymbol{\shortminus}{\mathbin}{AMSa}{"39}
\DeclareMathSymbol{-}{\mathbin}{AMSa}{"39}

\DeclareRobustCommand{\SkipTocEntry}[5]{}

\DeclareMathOperator*{\colim}{colim}
\DeclareMathOperator{\Hom}{Hom}

\newcommand{\rmat}[1]{\mathbf{r}_{#1}}

\newcommand{\congto}{\xrightarrow{\sim}}

\newcommand{\arrtip}{latex'}

\begin{document}
\title{Canonical bases for Coulomb branches of 4d $\cN=2$ gauge theories}

\author[Sabin Cautis]{Sabin Cautis}
\address[Sabin Cautis]{University of British Columbia \\ Vancouver BC, Canada}
\email{cautis@math.ubc.ca}

\author[Harold Williams]{Harold Williams}
\address[Harold Williams]{University of California \\ Davis CA, USA}
\email{hwilliams@math.ucdavis.edu}

\begin{abstract}
We construct and study a nonstandard t-structure on the derived category of equivariant coherent sheaves on the Braverman-Finkelberg-Nakajima space of triples $\hR_{G,N}$, where $\N$ is a representation of a reductive group $G$. Its heart $\KPcohGN$ is a finite-length, rigid, monoidal abelian category with renormalized $r$-matrices. We refer to objects of $\KPcohGN$ as Koszul-perverse coherent sheaves. Simple objects of $\KPcohGN$ define a canonical basis in the quantized $K$-theoretic Coulomb branch of the associated gauge theory. These simples possess various characteristic properties of Wilson-'t Hooft lines, and we interpret our construction as an algebro-geometric definition of the category of half-BPS line defects in a 4d $\mathcal{N}=2$ gauge theory of cotangent type. 


\end{abstract}

\maketitle

\setcounter{tocdepth}{1}

\tableofcontents

\section{Introduction}\label{sec:intro}

\thispagestyle{empty}

Let $G$ be a complex reductive group and $N$ a finite-dimensional representation of $G$. The purpose of this paper is to associate to this data a tensor category $\KPcohGN$. We refer to the objects of $\KPcohGN$ as Koszul-perverse coherent sheaves. Our main motivations for introducing and studying these categories are as follows. 
\begin{enumerate}
	\item They define canonical bases for quantized $K$-theoretic Coulomb branches, objects which unify and generalize a range examples of independent interest in geometry and Lie theory. 	
	\item They provide a new class of tensor categories with various features, such as renormalized $r$-matrices and monoidal cluster structures, which are of broader interest in representation theory.
	\item They provide a mathematically precise definition of the category of half-BPS line defects in a 4d $\mathcal{N}=2$ gauge theory of cotangent type. 
\end{enumerate}

The definition of $\KPcohGN$ involves an infinite-dimensional space $\hR_{G,N}$ introduced by Braverman-Finkelberg-Nakajima \cite{BFN}. This space is an intersection of infinite-rank vector bundles over the affine Grassmannian $\Gr_G$. It is essential for our purposes that this intersection be taken in the sense of derived algebraic geometry. The space $\hR_{G,N}$ and its classical locus $\hR_{G,N}^{\cl}$ have the same homology and $K$-theory, hence they are equivalent for the purposes of \cite{BFN}, but they do not have the same coherent sheaf theory. 

By construction, $\hR_{G,N}$ is equipped with a closed embedding
$$ i: \hR_{G,N} \to \Gr_G \times N_\cO, $$
where $\cO := \C[[t]]$ and $N_\O := N \otimes_\C \O$. We write $\sigma: \Gr_G \to \Gr_G \times N_\O$ for the inclusion of the zero section. These maps are equivariant for the natural action of $\hGO := (G_\cO \rtimes \C^\times) \times \C^\times$ on the spaces involved: here the inner $\C^\times$ acts by multiplication on the loop variable $t$, while the outer $\C^\times$ acts by scalar multiplication on $N$. We let $\Coh^{\hGO}(\hR_{G,N})$ denote the derived category of $\hGO$-equivariant coherent sheaves on $\hR_{G,N}$.

\begin{Theorem}[cf. Theorem \ref{thm:koszul-perverse}]\label{thm:introKPdef}
	There exists a unique t-structure on $\Coh^{\hGO}(\hR_{G,N})$ such that the composition
	$$ \Coh^{\hGO}(\hR_{G,N}) \xrightarrow{i_*} \Coh^{\hGO}(\Gr_G \times N_\cO) \xrightarrow{\sigma^*} \Coh^{\hGO}(\Gr_G) $$
	is t-exact with respect to the Koszul-perverse t-structure on $\Coh^{\hGO}(\Gr_G)$. We refer to this as the Koszul-perverse t-structure on $\Coh^{\hGO}(\hR_{G,N})$ and write $\KPcohGN$ for its heart. 
\end{Theorem}

The perverse t-structure on $\Coh^{\hGO}(\Gr_G)$ was introduced in \cite{AB10} and studied in \cite{BFM,CW1}. Its heart is referred to as the coherent Satake category. The Koszul-perverse t-structure on $\Coh^{\hGO}(\Gr_G)$ is its image under a simple regrading autoequivalence of $\Coh^{\hGO}(\Gr_G)$. 

By restriction, the Koszul-perverse t-structure on $\Coh^{\hGO}(\hR_{G,N})$ defines a t-structure on $\Coh^{\C^\times}\!(N_\O)$, the derived category of finitely generated graded modules over the symmetric algebra $\Sym^\bullet (N_\cO^\vee)$. This  restricted t-structure is familiar from the theory of Koszul duality \cite{BGS96}. 
It is distinguished by the fact that tensor products are t-exact. This is a special case of the t-exactness of convolution products implicit in Theorem \ref{thm:introstructuralsummary} below. Our terminology reflects the fact that the Koszul-perverse t-structure combines this ``Koszul'' t-structure with the perverse t-structure on $\Coh^{\hGO}(\Gr_G)$. In \cite{CW3} we make this more explicit, showing that a variant of linear Koszul duality \cite{MR10} relates $\hR_{G,\N}$ with a sheaf of infinite-dimensional Clifford algebras on $\Gr_G$, and that the objects of $\KPcohGN$ are the coherent sheaves which become perverse under the associated duality. 

Our main structural results about $\KPcohGN$ can be summarized as follows.

\begin{Theorem}[cf. Theorems \ref{thm:koszul-perverse}, \ref{thm:convsecmainthm}, \ref{thm:adjoints}, \ref{thm:rmatrix}]\label{thm:introstructuralsummary}
The category $\KPcohGN$ is a finite-length, rigid, monoidal abelian category with renormalized $r$-matrices. Its simple objects are, up to shifts, in bijection with pairs $(\l^\vee, \mu) \in (P^\vee \times P)/W$. 
\end{Theorem}
 
Here $P$ and $P^\vee$ are the weight and coweight lattices of $G$ and the Weyl group $W$ acts diagonally. Like the coherent Satake category, $\KPcohGN$ is not semisimple nor symmetric monoidal in general. In the remainder of the introduction we elaborate on the main motivations listed above. We then discuss some further directions and illustrate our construction in the simple (but still illuminating) example of $(G,\N) = (GL_1,\C)$.

\subsection{Canonical bases for Coulomb branches}

Coulomb branches are geometric objects attached to certain supersymmetric quantum field theories. The case relevant to our construction is that of a 4d $\mathcal{N}=2$ theory on ${\mathbb{R}}^3 \times S^1$. In the case of a 4d $\cN=2$ gauge theory $\TGN$ of cotangent type --- a theory defined by the data $G$ and $N$ we have been discussing --- the Coulomb branch $\CB{G,N}$ is the spectrum of $K^{G_\O}(\hR_{G,N})$, the $G_\O$-equivariant algebraic $K$-theory of $\hR_{G,N}$, equipped with its convolution product \cite{Nak16, BFN19}. One refers to this incarnation of $\CB{G,N}$ as the $K$-theoretic Coulomb branch of $\TGN$, indicating that the construction does not specify the hyperk\"ahler metric but does specify an integral form. It has a canonical quantization defined by working equivariantly with respect to loop rotation. The following is an immediate consequence of Theorem \ref{thm:introstructuralsummary}. 

\begin{Corollary}\label{cor:introbasis}
	The classes of simple Koszul-perverse coherent sheaves define a canonical positive integral basis in the quantized $K$-theoretic Coulomb branch of $\Theory{G,N}$. 
\end{Corollary}

The significance of \cite{BFN} stems in part from the fact that many objects of independent interest in geometry and Lie theory can be recognized as (quantized) K-theoretic Coulomb branches. Examples include the spherical DAHA,  certain open Richardson varieties in affine flag manifolds, multiplicative Nakajima quiver varieties of type $A^{(1)}_n$, and certain wild character varieties of marked surfaces. Corollary \ref{cor:introbasis} thus provides a uniform construction of categorified canonical bases for these objects. In some cases, such as the spherical DAHA, this is the first such construction we are aware of. In general, one expects these bases to be close relatives of the dual canonical/upper global basis of Lusztig and Kashiwara (cf. \cite{FF21}).  

\subsection{Renormalized $r$-matrices and cluster categorification}

Structurally, $\KPcohGN$ is closely related to two other well-studied classes of tensor categories: finite-dimensional representations of affine quantum groups and of quiver Hecke (or KLR) algebras. The root of their similarities is that they are all equipped with renormalized $r$-matrices, a weak relative of a braided monoidal structure. Renormalized $r$-matrices play a fundamental role in the representation theory of quantum affine algebras \cite{Kas02}, and were constructed in the setting of quiver Hecke algebras in \cite{KKK13}. 

In the context of quantum affine algebras, renormalized $r$-matrices are constructed by renormalizing the dependence of the universal $r$-matrix on a spectral parameter. Their construction in the setting of $\KPcohGN$ replaces this spectral parameter with the global coordinate appearing in the extension of $\hR_{G,N}$ to a factorization space (see Section \ref{sec:factorization}). This global version of $\hR_{G,N}$ specializes to the Beilinson-Drinfeld Grassmannian \cite{BD} in the case $N = 0$, and was first considered in general in \cite{BFN19}. 

Renormalized $r$-matrices have many implications for the structure of a tensor category. For example, by results of \cite{KKKO15a, KKKO18} they imply the following. 

\begin{Corollary}[{cf. \cite[Thm. 5.15]{CW1}}]\label{cor:introLeclerc}
Let $\cF$, $\cG \in \KPcohGN$ be simple and suppose either $\cF \conv \cF$ or $\cG \conv \cG$ is simple. Then the head and socle of $\cF \conv \cG$ are simple and each occur only once as a factor in any composition series of $\cF \conv \cG$. Moreover,  if $[\cF]$ and $[\cG]$ $q$-commute in $K^{\hGO}(\hR_{G,N})$ then $\cF \conv \cG$ is simple and isomorphic to $\cG \conv \cF$ up to a loop shift. 
\end{Corollary}

Informally, this says that the bases of Corollary \ref{cor:introbasis} satisfy a version of Leclerc's conjecture on the dual canonical bases of quantum unipotent coordinate rings \cite{Lec03}. 

A key insight of \cite{KKKO18} is that renormalized $r$-matrices provide a structural explanation for monoidal cluster categorifications. In \cite{HL10} and its successors \cite{HL13, Nak11, KQ14, Qin17,KKOP20,KKOP21} it was shown that certain representation categories of quantum affine algebras are categorified cluster algebras. In \cite{KKKO18} it was shown that representation categories of quiver Hecke algebras provide further examples of such. 

The existence of renormalized $r$-matrices suggests that the categories $\KPcohGN$ provide a new class of monoidal cluster categorifications.  This was established in \cite{CW1} for $\KP{GL_n}{0}$. Moreover, in Section \ref{sec:GLn} we compute enough of the canonical basis induced by $\KP{GL_2}{\C^2}^\eta$ (where $\KP{GL_2}{\C^2}^\eta$ is a variant of $\KP{GL_2}{\C^2}$ involving the center of $GL_2$) to show the following. 

\begin{Theorem}[cf. Theorem \ref{thm:qcatthminthesec}]\label{thm:introGL2clusters}
The category $\KP{GL_2}{\C^2}^\eta$ is a quantum monoidal cluster categorification of type~$A_2^{(1)}$ with one frozen variable. 
\end{Theorem}

We also confirm in this example that the duality functors in $\KP{GL_2}{\C^2}^\eta$ categorify the twist (or DT) automorphism of the relevant cluster algebra (Proposition \ref{prop:dualsforGLn}). Related to this is the observation that, since sheaves supported on the identity are manifestly invariant under double duals, the duality functors define a discrete integrable system on the K-theoretic Coulomb branch for general $G$ and $N$. 

\subsection{Line defects in 4d $\mathcal{N}=2$ theories} A prominent role in 4d $\cN=2$ quantum field theory is played by line defects which preserve a half-BPS subalgebra \cite{Kap06b,Kap06a,KS09,GMN13,CN14,CiDZ17}. These defects are organized into a monoidal category, which we denote by $\LGN$ for the gauge theory $\TGN$. Many of our results about $\KPcohGN$ correspond to expected properties of $\LGN$. Based on these parallels, we interpret $\KPcohGN$ as providing a mathematical definition of $\LGN$. 

A half-BPS subalgebra contains a supercharge $Q_{HT}$ which defines a holomorphic-topological twist of a 4d $\cN=2$ theory (with some caveats) \cite{Kap06b}. The category of line defects in the twist of $\TGN$ can be interpreted as boundary conditions for an equivariant $B$-model with target $\hR_{G,N}$, hence as the derived category $\Coh^{G_\cO}(\hR_{G,N})$. Preserving $Q_{HT}$ is a weaker condition than preserving the full half-BPS subalgebra, hence $\LGN$ should be a proper subcategory of $\Coh^{G_\cO}(\hR_{G,N})$. 

The basic examples of half-BPS line defects are Wilson-'t Hooft lines \cite{Kap06a}. In \cite{Kap06b,KS09} it is argued that these induce a basis of the Grothendieck ring of the category of line defects. Wilson-'t Hooft lines are labeled by their electric and magnetic charges, which correspond to pairs of a weight and coweight up to the Weyl group action. This matches our finding that $\KPcohGN$ is finite-length and its simple objects are classified by the same data. 

In \cite{GMN13} it is argued that the basis of simple half-BPS line defects has the structure of a cluster algebra (or a close variant thereof). The type of this cluster structure is given by the BPS quiver of the theory, which in the case of $\TGN$ has a known recipe \cite{ACCERV14}. We confirmed that $\KP{GL_n}{0}$ categorifies the expected cluster algebra in \cite{CW1}, and we confirm that $\KP{GL_2}{\C^2}$ does in Theorem \ref{thm:introGL2clusters}. Despite this being the simplest example with $G$~nonabelian and $\N$ nonzero, Theorem \ref{thm:introGL2clusters} is still a very nontrivial consistency check, requiring multiple independent computations to align in a precise way. 

\subsection{Further directions}

We mention some expectations regarding $\KPcohGN$ which are not pursued in this paper. Writing $\fg$ for the adjoint representation, the space $\hR_{G,\fg}$ is a variant of the affine Grassmannian Steinberg variety of $G$. In this case one expects the following, where $G^\vee$ is the Langlands dual of $G$. 

\begin{Conjecture}\label{con:introSdual}
	There is a monoidal equivalence $\KP{G^\vee}{\fg^\vee} \congto \KP{G}{\fg}$. The induced bijection between the simple objects on each side is compatible with their labelings by $(P \times P^\vee)/W$. 
\end{Conjecture}

This should be read as an extension of the conjecture of \cite[Section 7.9]{BFM} that $\Coh^{G^\vee_\cO}(\hR_{G^\vee,\fg^\vee})$ and $\Coh^{G_\cO}(\hR_{G,\fg})$ are equivalent, which in turn is a classical, local manifestation of the geometric Langlands correspondence. Conjecture \ref{con:introSdual} implies in particular the following coherent version of the geometric Satake equivalence. It should be interpreted geometrically as reflecting the relationship between ${\mathcal D}$-modules on $\Gr_G$ (in the guise of $\Rep\, G^\vee$) and coherent sheaves on the cotangent bundle of $\Gr_G$ (in the guise of $\KP{G}{\fg}$). 

\begin{Conjecture}
There is a monoidal functor $\Rep\, G^\vee \to \KP{G}{\fg}$ which takes the irreducible $G^\vee$-representation with highest weight $\omega^\vee$ to the simple object labeled by $(0,\omega^\vee) \in P \times P^\vee$. 
\end{Conjecture}

Another case of interest is that of quiver gauge theories. Let $\Gamma$ be an ADE root system, $Q$ an orientation of its Dyknin diagram, and $\alpha = \sum_{i \in Q_0} n_i \alpha_i$ a positive root. Consider the group $G = \prod_{i \in Q_0} GL_{n_i}$ and the $G$-representation $N = \bigoplus_{i \to j \in Q_1} \Hom(\C^{n_i}, \C^{n_j})$. The associated quantized K-theoretic Coulomb branch is a quantization of the coordinate ring of an open Richardson variety $R_{w,v}$ in the flag variety of type $\Gamma^{(1)}$, the affinization of $\Gamma$ (here $w, v$ are affine Weyl group elements determined by~$\alpha$) \cite{FKRD18}. On the other hand, it is known that this quantized coordinate ring is the localization of the Grothendieck ring of a certain monoidal subcategory $C_{w,v}$ of finite-dimensional graded modules over a quiver Hecke algebra of type $\Gamma^{(1)}$ \cite{KKOP17}. 

\begin{Conjecture}
There is a monoidal functor $C_{w,v} \to \KPcohGNeta$ which takes distinct simples to distinct simples. In particular, the bases defined by the classes of simples in each category coincide (up to inverting certain classes of simples in $C_{w,v}$). 
\end{Conjecture}

\subsection{An abelian example}\label{sec:abex}

\newcommand{\CO}{\C_\O}
\newcommand{\CK}{\C_\K}
\newcommand{\hRCstar}{\hR_{\C^\times\hspace{-.5mm},\hspace{0mm}\C}}
\newcommand{\KPetaCstar}{{\mathcal{KP}}^\eta_{\C^\times\hspace{-.5mm},\hspace{0mm}\C}}

For concreteness, let us explicitly describe the category of Koszul-perverse coherent sheaves on $\hR := \hR_{\C^\times,\C}$ where $\C^\times$ acts on $\C$ by multiplication. Though the perverse and standard t-structures on $\Gr_{\C^\times}$ coincide, this example still usefully illustrates the interaction between derived geometry, Koszul duality, and cluster theory. 

To simplify the exposition we replace all spaces with their reduced locus (the sheaves we care about always have reduced support), we do not explicitly distinguish between ordinary algebras and pro-algebras (we speak as if $\CK$ is a scheme rather than an ind-scheme), and we ignore the action of loop rotation. 

The space $\hR$ is the disjoint union of infinitely many components $\{\hR_m\}_{m \in \Z}$ where $\hR_m$ is the derived intersection of $t^m \CO$ and $\CO$ inside $\CK$. 
Note that $\CK$ and $\CO$ are really just~$\K$ and $\O$, but we use the former notation when treating them as geometric objects rather than vector spaces or rings. 
The naive classical intersection $\hR^\cl_m$ is just the affine space $t^{m_+}\CO$, where $m_+ := \max\{0,m\}$. Thus at the level of coordinate rings we have
$$ \C[\hR^\cl_m] \cong \Sym^\bullet((t^{m_+}\cO)^\vee). $$
The coordinate ring of the derived intersection, meanwhile, is a graded commutative ring with additional generators in cohomological degree $-1$:
$$ \C[\hR_m] \cong \Sym^\bullet((t^{m_+}\cO)^\vee) \ot \Sym^\bullet ((\cK/t^{m_-}\cO)^\vee[1]), $$ 
where $m_- := \min\{0,m\}$. 

We let the subgroup $\eta: \C^\times \into \C^\times_\O$ of constant loops assume the role of the scaling $\C^\times$ in $\hGO$. This determines a Koszul-perverse t-structure on $\Coh^{\C^\times_\O}(\hR)$ whose heart is denoted by $\KPetaCstar$ (see the discussion in Section \ref{sec:eta2}). In this case $\eta$ has an evident splitting $\C^\times_\O \to \C^\times$, defining a $\C_\O^\times$-equivariant structure on the trivial line bundle on $\hR$. We use the notation $\cF \mapsto \cF\la 1 \ra$ to mean the operation of twisting by this line bundle, and the term Koszul shift to mean $\cF \mapsto \cF[1]\la -1 \ra$. 

The classical intersection $\hR_m^\cl$ is naturally a closed subscheme of $\hR_m$, hence of $\hR$ itself. 
Via pushforward, the structure sheaf of $\hR_m^\cl$ then gives rise to a $\C^\times_\O$-equivariant coherent sheaf~$\O_{\hR_m^\cl}$ on $\hR$. Anticipating the notation of Section \ref{sec:GLn}, we write $\O_{\hR_m^\cl}$ as $\cP_{m,0}$. One can check that $\cP_{m,0}$ is an irreducible Koszul-perverse sheaf, and that any irreducible is a Koszul shift of some $\cP_{m,0}$. 

The convolution products of these irreducibles are Koszul shifts of the structure sheaves of more general subschemes of $\hR$. For example, one finds that $\cP_{1,0} \conv \cP_{-1,0}$ and $\cP_{-1,0} \conv \cP_{1,0}$ are respectively the structure sheaves of the closed subschemes 
$$Z_{1,-1} := t\CO \ \ \text{ and } \ \ Z_{-1,1} := \CO \times_{t^{-1}\CO} \CO.$$
The former is a codimension-one subscheme of $\hR_0^\cl = \CO$, while the latter is a derived scheme lying between $\hR_0^\cl$ and $\hR_0$. At the level of coordinate rings we have 
\begin{equation}\label{eq:introabconvvars} \C[Z_{1,-1}] \cong \Sym^\bullet((t\cO)^\vee), \quad \C[Z_{-1,1}] \cong \Sym^\bullet(\cO^\vee) \ot \Sym^\bullet ((t^{-1}\cO/\cO)^\vee[1]), \end{equation}
both viewed as quotients of 
\begin{equation}\label{eq:introhrzero} \C[\hR_0] \cong \Sym^\bullet(\cO^\vee) \ot \Sym^\bullet ((\cK/\cO)^\vee[1]) \end{equation}
in the obvious way. 

These convolution products fit into short exact sequences
\begin{align}\label{eq:abexseq}
\begin{split}
 &0 \to \cP_{0,0} \to \cP_{1,0} \conv \cP_{-1,0} \to \cP_{0,0}[1] \la -1 \ra \to 0 \\
 &0 \to \cP_{0,0}[1] \la -1 \ra \to \cP_{-1,0} \conv \cP_{1,0} \to \cP_{0,0} \to 0
\end{split}
\end{align}
in $\KPetaCstar$. To see the first, note that the Koszul resolution associated to the codimension one embedding $Z_{1,-1} = t\CO \subset \CO = \hR_0^\cl$ gives rise to an exact sequence
\begin{equation*}
0 \to \O_{\hR^\cl_0} \la -1 \ra \to \O_{\hR^{\cl}_0} \to \O_{Z_{1,-1}} \to 0
\end{equation*}
in the heart of the {\it standard} t-structure on $\Coh^{\C^\times_\cO}(\hR_0^\cl)$. Rotating the associated exact triangle yields the desired exact sequence in $\KPetaCstar$. 

To see the second, note that the factor $\Sym^\bullet ((t^{-1}\cO/\cO)^\vee[1])$ in the right-hand side of (\ref{eq:introabconvvars}) is two-dimensional as a complex vector space. 
Its subspace $(t^{-1}\cO/\cO)^\vee[1]$ generates a $\C[\hR_0]$-submodule of $\C[Z_{-1,1}]$ which is isomorphic to the Koszul shift of $\C[\hR_0^\cl]$. Since the quotient by this submodule is just $\C[\hR_0^\cl]$ itself, we obtain the desired exact sequence. 

The exact sequences (\ref{eq:abexseq}) do not split, and consequently $\cP_{1,0} \conv \cP_{-1,0}$ and $\cP_{-1,0} \conv \cP_{1,0}$ are not isomorphic. On the other hand, we can see they are connected by nonzero maps in either direction, namely the compositions
\begin{gather*} 
	r_{\cP_{1,0}, \cP_{-1,0}}: \cP_{1,0} \conv \cP_{-1,0} \onto \cP_{0,0}[1]\la -1 \ra  \into \cP_{-1,0} \conv \cP_{1,0}, \\
	r_{\cP_{-1,0}, \cP_{1,0}}: \cP_{-1,0} \conv \cP_{1,0} \onto \cP_{0,0} \into \cP_{1,0} \conv \cP_{-1,0}
\end{gather*}
of the maps appearing in the exact sequences (\ref{eq:abexseq}). These maps are the renormalized $r$-matrices of Theorem~\ref{thm:introstructuralsummary}.

Noting that $\cP_{0,0}$ is the monoidal unit, the relation in $K^{\C^\times_\cO}(\hR)$ implied by (\ref{eq:abexseq}) can be recognized as the unique exchange relation of the cluster algebra of the following quiver with one frozen and one unfrozen vertex.
\begin{equation}\label{eq:introquiver}
\begin{tikzpicture}[baseline=(current  bounding  box.center),thick,>=\arrtip]
	
	\coordinate (a) at (0,0);
	\coordinate (b) at (2,0);
	\node (c) at (-1.2,0) {$\cP_{0,0}[1]\la -1 \ra$};
	\node (c) at (2.6,0) {$\cP_{1,0}$};
	
	\newlength\squareEdgeLength
	\setlength\squareEdgeLength{0.12cm} 
	
	\draw (a) ++(-0.5\squareEdgeLength,-0.5\squareEdgeLength) rectangle ++(\squareEdgeLength,\squareEdgeLength);
	
	\fill (b) circle (.06);
	\draw [->,shorten <=1.7mm,shorten >=1.7mm] ($(a)$) to ($(b)$);
\end{tikzpicture}
\end{equation}
Here $[\cP_{1,0}]$ and $[\cP_{-1,0}]$ correspond to the two unfrozen cluster variables in this cluster algebra and $[\cP_{0,0}[1]\la -1 \ra]$ to the frozen cluster variable. 
After checking some further commutativity relations one confirms by inspection that $\KPetaCstar$ is indeed a monoidal cluster categorification. 

\subsection{Acknowledgements}
We are deeply grateful to Sam Raskin, Hiro Lee Tanaka, Aaron Mazel-Gee, and Chang-Yeon Chough for taking the time to discuss numerous technical issues that arose in the preparation of this paper and its companions \cite{CWig, CWtm}. S.C. was supported by NSERC Discovery Grant 2019-03961 and H. W. was supported by NSF grants DMS-1801969 and DMS-2143922. 

\section{Conventions}\label{sec:convnot}

This section summarizes our notation and terminology. 
It also serves as a brief digest of \cite{CWig,CWtm}, sufficient to follow our use of the basic notions from these companion papers.  
For general categorical and geometric background our default references are \cite{LurHTT,LurHA,LurSAG}, and we follow their conventions up to a few exceptions below. 
\subsection*{Categorical conventions}

\begin{itemize}
	\item We use cohomological indexing for t-structures. If $\catC$ has a t-structure $(\catC^{\leq 0}, \catC^{\geq 0})$ with heart $\catC^\heartsuit$ we write $\tau^{\leq n} : \catC \to \catC^{\leq n}$, $H^n: \catC \to \catC^\heartsuit$, etc., for the associated functors. 
	\item We say category and $\infty$-category interchangeably, and say ordinary category to specify a category in the traditional sense. We write $\Map_\catC(X,Y)$ for the mapping space between $X, Y \in \catC$, and regard ordinary categories as $\infty$-categories with discrete mapping spaces. We write $\Catinfty$ for the $\infty$-category of (possibly large) $\infty$-categories. 
	\item All limit or colimit diagrams are implicitly small unless otherwise stated. Thus in ``let $X \cong \colim X_\al$ be a filtered colimit'' the indexing diagram is assumed to be small. 
	\end{itemize}

\subsection*{Algebraic conventions}
\begin{itemize}
	\item We write $\CAlgC$ for the $\infty$-category of nonpositively graded commutative dg $\C$-algebras. Given $A \in \CAlgC$, we write $\CAlg_A := (\CAlgC)_{A/}$ for the category of (commutative dg) $A$-algebras and $\tau_{\leq n} \CAlg_A$ for the category of $n$-truncated $A$-algebras (i.e. $B \in \CAlg_A$ such that $H^k(B) \cong 0$ for $k < -n$). Note the implicit use of homological indexing when referring to truncations, to avoid excessive signs (notationally this is  reflected in the use of subscripts rather than superscripts). An algebra is truncated if it is $n$-truncated for some $n$. 
	\item An $A$-algebra $B$ is finitely $n$-presented if it is a compact object of $\tau_{\leq n} \CAlg_A$. It is \textbf{strictly tamely $n$-presented} if, for every finitely $n$-presented $A$-algebra $C$, every morphism $C \to B$ in $\CAlg_A$ factors through a flat morphism $C' \to B$ such that $C'$ is also finitely $n$-presented over~$A$ \cite[\tmdefreasalg]{CWtm}. Equivalently, $B$ can be written as a filtered colimit of finitely $n$-presented $A$-algebras over which it is flat. 
	\item An $A$-algebra $B$ is almost finitely presented if its truncation $\tau_{\leq n} B$ is finitely $n$-presented for all $n$. It is (almost) \textbf{strictly tamely presented} if $\tau_{\leq n} B$ is strictly tamely $n$-presented for all $n$. We omit ``almost'' by default for the latter but not the former, ``finitely presented'' having a different, stronger meaning in derived contexts. 
	\item If $A$ is an ordinary Noetherian ring, an ordinary $A$-algebra $B$ is almost finitely presented if and only if it is finitely presented in the usual sense. It is strictly tamely presented if and only if it is the union of the finitely presented $A$-subalgebras over which it is flat. For example,  $\C[x_1,x_2,\dotsc]$ is strictly tamely presented over $\C$.
	\item Strictly tamely presented $\C$-algebras are coherent (i.e. $H^0(A)$ is a coherent ordinary ring and $H^n(A)$ is finitely presented over  $H^0(A)$ for all $n$), but are more well-behaved than arbitrary coherent rings (e.g. they are closed under tensor products). 
\end{itemize}

\subsection*{Geometric conventions}

\begin{itemize}
	\item We write $\StkC$ for the category of stacks, i.e. functors from $\CAlgC$ to the category of (possibly large) spaces which are sheaves for the fpqc topology.  If $G$ is a group scheme acting on an ind-scheme $X$, we write $X/G$ for the fpqc quotient. We write $X^\cl$ for the underlying classical stack of $X$ (i.e. its restriction to ordinary $\C$-algebras). 
	\item A stack $X$ is geometric if its diagonal $X \to X \times X$ is affine and there exists faithfully flat morphism $\Spec A \to X$. This follows \cite[Ch. 9]{LurSAG}, but we caution that terminology varies in the literature. A geometric stack $X$ is truncated (resp. classical) if $A$ is truncated (resp. an ordinary ring) for any flat $\Spec A \to X$. A morphism $X \to Y$ in $\StkC$ is geometric if  $X \times_Y \Spec A$ is geometric for any $\Spec A \to Y$. 
	\item An affine morphism $X \to Y$ is strictly tamely presented if the coordinate ring of $X \times_Y \Spec A$ is a strictly tamely presented $A$-algebra for any $\Spec A \to Y$. A geometric morphism $X \to Y$ is \textbf{tamely presented} if for any $\Spec A \to Y$, there exists a strictly tamely presented flat cover $\Spec B \to \Spec A \times_Y X$ such that $B$ is a strictly tamely presented $A$-algebra \cite[\tmdefreasaffmorph]{CWtm}. 
	\item A geometric stack $X$ is tamely presented if it is so over $\Spec \C$. A  (quasi-compact, semi-separated) scheme $X$ is tamely presented if it is so as a geometric stack. If $G$ is an affine group scheme acting on a tamely presented scheme $X$, then $X/G$ is a tamely presented geometric stack.  
	\item A stack $X$ is convergent if $X(A) \cong \lim X(\tau_{\leq n})$ for any $A \in \CAlgC$. An \textbf{ind-geometric stack} is a convergent stack $X$ which admits an expression $X \cong \colim_{\al} X_\al$ as a filtered colimit of truncated geometric stacks along closed immersions (in the category of convergent stacks) \cite[\igdefindgeomstack]{CWig}. Any classical ind-scheme $X \cong \colim X_\al$ is an ind-geometric stack (assuming the $X_\al$ are quasi-compact and semi-separated), as is the quotient $X/G$ of $X$ by a classical affine group scheme $G$. An ind-geometric stack or ind-scheme is \textbf{ind-tamely presented} if we can choose the $X_\al$ to be tamely presented and the maps among them almost finitely presented \cite[\tmdefindP]{CWtm}. 
	\item  A morphism $f: X \to Y$ of geometric stacks is proper if for any $\Spec A \to Y$, the fiber product $X \times_Y \Spec A$ is proper over $\Spec A$ in the sense of \cite[Def. 5.1.2.1]{LurSAG} (in particular, $f$ is a relative algebraic space). 
	\item A geometric stack $X$ is locally Noetherian if it admits a flat cover $\Spec A \to X$ such that $A$ is Noetherian. It is \textbf{admissible} if it admits an affine morphism to a locally Noetherian geometric stack. An ind-geometric stack is admissible if it admits a presentation $X \cong \colim X_\al$ such that each $X_\al$ is admissible. 
\end{itemize}
	
\subsection*{Sheaf-theoretic conventions}	
\begin{itemize}
	\item We write $\QCoh(X)$ for the $\infty$-category of quasi-coherent sheaves on $X$. When $X$ is a truncated geometric stack $\Coh(X) \subset \QCoh(X)$ denotes the subcategory of coherent sheaves, i.e. the bounded, almost perfect objects. When $X$ is tamely presented these are the bounded objects with coherent cohomology sheaves in the abelian sense. 
	\item  A morphism $f: X \to Y$ of truncated geometric stacks has \textbf{coherent pullback} if $f^*: \QCoh(Y) \to \QCoh(X)$ takes $\Coh(Y)$ to $\Coh(X)$. It has \textbf{\stable coherent pullback} if for any truncated $Y'$ and any tamely presented $Y' \to Y$, the base change $f': X \times_Y Y' \to Y'$ has coherent pullback \cite[\tmdefgeomcohpull]{CWtm}. If $Y$ is a tamely presented scheme, $f$ having coherent pullback implies it has \stable coherent pullback. 
	\item  If $Y$ is an ind-tamely presented ind-geometric stack and $Y \cong \colim Y_\al$ a presentation as such, we have $\Coh(Y) \cong \colim \Coh(Y_\al)$. If in addition $Y$ is admissible, the category $\IndCoh(Y)$ is the ind-completion of $\Coh(Y)$. A geometric morphism $X \to Y$ has \stable coherent pullback if its base change to each $Y_\al$ does in the sense defined above \cite[\tmdefindgeomcohpull]{CWtm}. 
	\item A morphism $f: X \to Y$ of truncated geometric stacks is \textbf{weakly smooth} if it is flat, tamely presented, and the diagonal $X \to X \times_Y X$ has \stable coherent pullback \cite[\tmdefweaklysmooth]{CWtm}. A tamely presented truncated geometric stack is weakly smooth if it is so over $\Spec \C$. Any smooth scheme is weakly smooth, while e.g. $\A^\infty := \Spec \C[x_1,x_2,\dotsc]$ is weakly smooth but not smooth. 
	\item Categories of equivariant sheaves are understand as being defined via quotient stacks. That is, if $G$ is a group scheme acting on an scheme $X$, we write $\QCoh^G(X)$ for $\QCoh(X/G)$. The expressions $\Coh^G(X)$ and $\IndCoh^G(X)$ are understood similarly, assuming the quotient $X/G$ is such that these categories are defined. 
\end{itemize}

\section{Koszul t-structures}\label{sec:Koszul}

In this section we define and study Koszul t-structures. We work in the setting of a classical Noetherian scheme $X$ acted on by an affine group scheme $\H$, together with a finite-rank $\H$-equivariant vector bundle $W \to X$ and two  $\H$-equivariant sub-bundles $V_1, V_2 \subset W$. We write $Y$ for the derived intersection of $V_1$ and $V_2$ in $W$. In our application $X$ will be a closed $\hGO$-invariant subscheme of $\Gr_G$, the group $\H$ will be $\hGO$, and $Y$ will be a finite-dimensional approximation of the restriction of $\hR_{G,\N}$ to $X$. 

We assume that $\H$ contains a central subgroup $\eta:  \C^\times \to \H$ which acts trivially on $X$ and with weight one on $W$. We say a t-structure on $\Coh^\H(X)$ or $\Coh^\H(Y)$ is Koszul if it satisfies a certain t-exactness condition with respect to $\eta$ (Definition \ref{def:Koszult}). This condition is related to the Koszul resolution of $\sigma_{k*}(\cO_X)$, where $\sigma_k: X \to V_k$ denotes the zero section for $k=1,2$. In our application $\eta$ is given either by the scaling $\C^\times \subset \hGO$, or induced by a central cocharacter of $G$ which acts with weight one on $\N$. 

Our main goal is to show that Koszul t-structures on $\Coh^{\H}(X)$ lift canonically to Koszul t-structures on $\Coh^{\H}(Y)$. This is captured by the following result, where $i_k: Y \to V_k$ denotes the natural map for $k=1,2$.

\begin{Theorem}\label{thm:koszul}
Given a bounded, Noetherian, Koszul t-structure on $\Coh^{\H}(X)$, there exists a unique t-structure on $\Coh^{\H}(Y)$ such that
$$\sigma^*_1 i_{1*}: \Coh^{\H}(Y) \to \Coh^{\H}(X)$$
is t-exact. Symmetrically, it is the unique t-structure such that $\sigma^*_2 i_{2*}$ is t-exact. This t-structure is bounded, Noetherian, and Koszul over $X$, and is Artinian if the t-structure on $\Coh^\H(X)$ is. 
\end{Theorem}

We begin in Section \ref{sec:tlifting} by considering such lifting problems in general. In Section~\ref{sec:koszulity} we introduce and discuss Koszul t-structures. In Section \ref{sec:koszullift} we explain how to lift Koszul t-structures from $\Coh^\H(X)$ to $\Coh^\H(V)$ for any $\H$-equivariant vector bundle $V \to X$ (Proposition~\ref{prop:koszul-t-bundles}), thus proving Theorem \ref{thm:koszul} in the case $W = V_1 = V_2$. In Section \ref{sec:koszul-ints} we further lift this t-structure from $V_1$ or $V_2$ to obtain a Koszul t-structure on $\Coh^\H(Y)$, proving Theorem~\ref{thm:koszul} in general. Finally, in Section \ref{sec:simples} we classify simple objects in the heart of this t-structure, with the main result being Theorem \ref{thm:koszulsimples}. 

\subsection{Lifting t-structures}\label{sec:tlifting} 
In this section we discuss the general problem of extending or lifting a t-structure along a functor $\Phi: \catC \to \catD$ so that the functor becomes t-exact. The ideas involved are an elaboration of those of \cite[Sec. 2]{Pol07}, and in turn of \cite{AJS03,AP06}.

Recall that a t-structure on a stable $\infty$-category $\catC$ consists of a pair of full subcategories $(\catC^{\leq 0}, \catC^{\geq 0})$ which are closed under isomorphism and satisfy the following conditions. 
	\begin{enumerate}
	\item $\Maps_\catC(X,Y[-1])$ is contractible for all $X \in \catC^{\leq 0}$, $Y \in \catC^{\geq 0}$.
	\item $\catC^{\leq 0}[1] \subset \catC^{\leq 0}$ and $\catC^{\geq 0}[-1] \subset \catC^{\geq 0}$.
	\item Every $X \in \catC$ fits into an exact triangle $ X' \to X \to X''$ with $X' \in \catC^{\leq 0}$ and $X'' \in \catC^{\geq 0}[-1]$.
\end{enumerate}
A t-structure on $\catC$ is the same data as a t-structure on its homotopy category $h\catC$ in the classical sense \cite{BBD82}. 

Recall that a category is presentable if it admits small colimits and is accessible, a mild set-theoretic condition. In this setting the adjoint functor theorem makes it easy to construct t-structures, as illustrated by the following result. Here given a subcategory $\catC' \subset \catC$ we write $\catC'^\perp$ for its right orthogonal (the full subcategory of $Y \in \catC$ such that $\Map_\catC(X,Y)$ is contractible for all $X \in \catC'$) and $\la \catC' \ra$ for the smallest full subcategory of~$\catC$ containing $\catC'$ and closed under extensions and colimits (so if $X' \to X \to X''$ is an exact triangle with $X', X'' \in \la \catC' \ra$, then $X \in \la \catC' \ra$). The condition of being compatible with filtered colimits means that $\catC^{\geq 0}$ is closed under filtered colimits in $\catC$. 

\begin{Proposition}\label{prop:generatedtstructure}
Let $\catC$ be a presentable stable $\infty$-category and $\catC' \subset \catC$ a small subcategory. The subcategories 
$ \catC^{\leq 0} := \la \catC' \ra$ and $\catC^{> 0} := \catC'^\perp $ 
are presentable and define a t-structure on $\catC$. If $\catC'$ consists of compact objects this t-structure is compatible with filtered colimits. 
\end{Proposition}
\begin{proof}
Existence of the stated t-structure follows from \cite[Prop. 1.4.4.11]{LurHA} and smallness of $\catC'$, as does presentability of $\catC^{\leq 0}$ and $\catC^{> 0}$. By definition $\catC^{> 0} = \la \catC' \ra^\perp$. Trivially $\la \catC' \ra^\perp \subset \catC'^\perp$, while $ \catC'^\perp \subset \la \catC' \ra^\perp$ follows as in \cite[Lem. 3.1]{AJS03}: if $X \in \catC'^\perp$ then  ${}^\perp X$ contains $\la \catC' \ra$ since it contains~$\catC'$ and is closed under colimits and extensions. The final claim is immediate from $\catC^{> 0} := \catC'^\perp$.
\end{proof}

We will mostly use Proposition \ref{prop:generatedtstructure} as packaged by the following construction. There are other possible t-structures for which $\Phi$ is t-exact, but the choices below are characterized by their minimizing $\catD^{\le 0}$ and maximizing $\catC^{\le 0}$, respectively. 

\begin{Proposition}\label{prop:induce}
Let $\catC$ and $\catD$ be presentable stable $\infty$-categories and $\Phi: \catC \leftrightarrows \catD: \Phi^R$ an adjoint pair. 
\begin{enumerate}
\item Suppose $\catC' \subset \catC$ is a small subcategory such that $\Phi^R \Phi$ is left t-exact for the t-structure given by $\catC^{\leq 0} := \la \catC'\ra$. Then $\Phi$ is t-exact with respect to the t-structure on $\catD$ given by $\catD^{\le 0} := \la \Phi(\catC') \ra$.
\item Suppose $(\catD^{\leq 0}, \catD^{\geq 0})$ is a t-structure such that $\catD^{\leq 0}$ is presentable and $\Phi \Phi^R$ is left t-exact. Then $\Phi$ is t-exact with respect to the t-structure on $\catC$ given by $\catC^{\le 0} := \Phi^{-1}(\catD^{\leq 0})$. 
\end{enumerate}
\end{Proposition}
\begin{proof}
We prove (2) as the other claim is similar. Note that $\Phi^{-1}(\catD^{\leq 0})$ is presentable since $\catD^{\leq 0}$ is and $\Phi$ preserves colimits. If $X \in \catC^{> 0}$ then $\Phi^R\Phi(X) \in \catC^{> 0}$ since $\Phi^R \Phi$ is left t-exact. This implies $\Map_{\catD}(\Phi(X'),\Phi(X)) \cong \Map_{\catC}(X',\Phi^R\Phi(X))$ is contractible for any $X' \in \catC'$, hence $\Phi(X) \in \catD^{> 0}$ and $\Phi$ is left t-exact. On the other hand, if $Y \in \catD^{> 0}$ then for any $X' \in \catC'$ we have that $\Map_{\catC}(X',\Phi^R(Y)) \cong \Map_{\catC}(\Phi(X'),Y)$ is contractible. Thus $\Phi^R$ is also left t-exact, hence its left adjoint $\Phi$ is right t-exact. 
\end{proof}

We say a t-structure on $\catC$ restricts to a subcategory $\catC_0 \subset \catC$ if $(\catC^{\leq 0} \cap \catC_0, \catC^{\geq 0} \cap \catC_0)$ is a t-structure on $\catC_0$. This is equivalent to $\catC_0$ being stable under the truncation functors in~$\catC$. 

\begin{Proposition}\label{prop:restrict}
Let $\Psi: \catC \to \catD$ be a t-exact functor between stable $\infty$-categories equipped with t-structures. Suppose $\catC_0 \subset \catC$ and $\catD_0 \subset \catD$ are stable subcategories such that $\Psi(\catC_0) \subset \catD_0$ and such that the t-structure on $\catD$ restricts to one on $\catD_0$. Then the t-structure on $\catC$ restricts to $\catC_0$ if either of the following conditions are satisfied.
\begin{enumerate}
		\item If $X \in \catC^{\leq 0}$ and $\Psi(X) \in \catD_0$, then $X \in \catC_0$. 
		\item If $X \in \catC^{\geq 0}$ and $\Psi(X) \in \catD_0$, then $X \in \catC_0$. 
\end{enumerate}
\end{Proposition}
\begin{proof}
For $X \in \catC_0$, we have $\Psi(\tau_{\catC}^{\le 0}(X)) \cong \tau_{\catD}^{\le 0}(\Psi(X)) \in \catD_0$ since $\Psi$ is t-exact. Assuming the first condition, we then have $\tau_{\catC}^{\le 0}(X) \in \catC_0$. We further have $\tau_{\catC}^{> 0}(X) \in \catC_0$ since $\catC_0$ is stable, hence the t-structure on $\catC$ restricts to $\catC_0$. The same argument works assuming the second condition.
\end{proof}

Recall that a t-structure is {\it Noetherian} (resp. {\it Artinian}) if every object in its heart satisfies the ascending (resp. descending) chain condition, is \emph{finite-length} if it is Noetherian and Artinian (equivalently, every object in its heart is finite-length), and is {\it bounded} if $\catC = \cup_{n} \catC^{\geq n}$  and $\catC = \cup_{n} \catC^{\leq n}$.

\begin{Proposition}\label{prop:conservative}
Let $\Psi: \catC \to \catD$ be a conservative t-exact functor between stable $\infty$-categories equipped with t-structures. If the t-structure on $\catD$ is Noetherian (resp. Artinian, finite-length, bounded), so is the t-structure on $\catC$.
\end{Proposition}
\begin{proof}
Suppose $\catD^\heartsuit$ is Noetherian. Let $X \in \catC^\heartsuit$ and let $X_1 \subset X_2 \subset \cdots$ be an ascending chain of subobjects of $X$. Since $\Psi$ is t-exact we get an ascending chain $\Psi(X_1) \subset \Psi(X_2) \subset \cdots$ of sub-objects of $\Psi(X) \in \catD^\heartsuit$. Since $\catD^\heartsuit$ is Noetherian this chain stabilizes. It follows that $\Psi(X_{i+1}/X_i) \cong 0$ for $i \gg 0$. Since $\Psi$ is conservative $X_{i+1}/X_i \cong 0$ for $i \gg 0$, hence $\catC^{\heartsuit}$ is Noetherian. The remaining properties are established similarly. 
\end{proof}

Finally, we record the following result for the proof of Proposition \ref{prop:koszul-t-bundles}.  

\begin{Lemma}\label{lem:noetherian-extend}
	Let $\catC$ be a presentable stable $\infty$-category and $\catC_0 \subset \catC$ a small stable subcategory which is closed under isomorphisms and equipped with a Noetherian t-structure. If we extend this to a t-structure on $\catC$ such that $\catC^{\leq 0} := \la \catC_0^{\leq 0} \ra,$  then $\catC_0^\heartsuit$ is closed under subobjects in $\catC^\heartsuit$. 
	That is, if $X \into Y$ is a monomorphism in $\catC^\heartsuit$ and $Y \in  \catC_0^\heartsuit$, then $X \in  \catC_0^\heartsuit$. 
\end{Lemma}
\begin{proof}
The claim is trivial if $X \cong 0$, so assume $X \ncong 0$. We claim there exists a nonzero map $X' \to X$ with $X' \in \catC_0^{\heartsuit}$. It suffices to find a nonzero map $X' \to X$ with $X' \in \catC_0^{\leq 0}$, since $X \in \catC^{\geq 0}$ and thus $\tau^{\geq 0}(X') \to X$ is still nonzero by adjunction. If no such $X' \to X$ exists, then since $\catC^{> 0} = (\catC_0^{\leq 0})^\perp$ we have $X \in \catC^{> 0} \cap \catC^{\heartsuit}$, hence $X = 0$. 
	
	Let $X_1$ denote the image of $X'\to X$. It is nonzero by construction, and it belongs to $\catC_0^\heartsuit$ since it is isomorphic to the image of $X' \to Y$. Now consider the monomorphism $X/X_1 \into Y/X_1$. If $X \notin \catC_0^\heartsuit$ then $X/X_1 \notin \catC_0^\heartsuit$, and the same argument implies that $X/X_1$ contains a nonzero subobject which belongs to $\catC_0^{\heartsuit}$. Letting $X_2 \subset X$ denote its preimage and repeating, we get a nonstabilizing ascending chain $X_1 \subset X_2 \subset \cdots$ of subobjects of $Y$ in $\catC_0^\heartsuit$. This contradicts $Y$ being Noetherian, so we must have $X \in \catC_0^\heartsuit$. 
\end{proof}

\subsection{Local and Koszul t-structures}\label{sec:koszulity}

Given an affine group scheme $\H$ acting on a classical Noetherian scheme $X$, we consider the following class of t-structures. 

\begin{Definition}\label{def:localt}
	Let $\pi: Y \to X$ be a $\H$-equivariant morphism of Noetherian schemes with $X$ classical. A t-structure on $\Coh^\H(Y)$ is {\it local over $X$} if $- \otimes \pi^* (\cV)$ is t-exact for any locally free sheaf $\cV \in \Coh^\H(X)$.
\end{Definition}

Our terminology follows that used in the nonequivariant setting in \cite{Pol07}, the definition there agreeing with ours by Proposition \ref{prop:Utstructure}. Perverse t-structures are local (including the standard t-structure, corresponding to the trivial perversity), and at least in the smooth, nonequivariant case every local t-structure on $X$ itself is of this form \cite[Cor. 2.3.6]{Pol07}. 

Now suppose we have a central cocharacter $\eta: \C^\times \to \H$ whose action on $X$ is trivial. This induces a weight decomposition $\Coh^{\H}(X) = \bigoplus_{n \in \Z} \Coh^\H_n(X)$ as follows. We begin with evident weight decomposition $\Coh^{\H \times \C^\times}(X) = \bigoplus_{n \in \Z} \Coh_n^{\H \times \C^\times}(X)$, where $\Coh_n^{\H \times \C^\times}(X) \cong \Coh^{\H}(X)$ for all $n$. The multiplication homomorphism $\H \times \C^\times \to \H$ induces a map $X/(\H \times \C^\times) \to X/\H$, and pullback provides a conservative functor $\Coh^\H(X) \to \Coh^{\H \times \C^\times}(X)$. We then define $\Coh^\H_n(X) \subset \Coh^\H(X)$ as the preimage of $\Coh^{\H \times \C^\times}_n(X) \subset \Coh^{\H \times \C^\times}(X)$. 

The weight decomposition of $\Coh^{\H}(X) $ leads to the following variant of Definition \ref{def:localt}. 

\begin{Definition}\label{def:Koszult}
	Let $\pi: Y \to X$ be a $\H$-equivariant morphism of Noetherian schemes with $X$ classical. A t-structure on $\Coh^\H(Y)$ is {\it Koszul over $X$} if $- \otimes \pi^* (\cV)[-n]$ is t-exact for any $n \in \Z$ and any locally free sheaf $\cV \in \Coh^\H_n(X)$.
\end{Definition}

Local and Koszul t-structures are in correspondence via the following regrading procedure. 

\begin{Construction}\label{construct}
	The autoequivalence of $\Coh^\H(X)$ which acts by $[-n]$ on $\Coh_n^\H(X)$ identifies local t-structures with Koszul t-structures. 
\end{Construction}

We use subscripts as in $(\Coh^\H(Y)_K^{\leq 0}, \Coh^\H(Y)_K^{\geq 0})$ and $\Coh^\H(Y)_{\Kzl}^\heartsuit$ to indicate a Koszul t-structure. By Proposition \ref{prop:generatedtstructure} this extends to a t-structure on $\IndCoh^\H(Y)$ and $\QCoh^\H(Y)$ by taking $\IndCoh^\H(Y)^{\leq 0}_\Kzl$ and $\QCoh^\H(Y)^{\leq 0}_\Kzl$ to be $\la \Coh^\H(Y)^{\leq 0}_\Kzl \ra$. The former extension is compatible with filtered colimits, and from this and the fact that the t-structure restricts to $\Coh^\H(Y)$ it follows that
\begin{equation}\label{eq:ICKoszul}
	\IndCoh^\H(Y)^{\leq 0}_\Kzl \cong \Ind(\Coh^\H(Y)^{\leq 0}_\Kzl) \quad \IndCoh^\H(Y)^{> 0}_\Kzl \cong \Ind(\Coh^\H(Y)^{> 0}_\Kzl). 
\end{equation}
The extension to $\QCoh^\H(Y)$ is less well-behaved, but will still be useful in some constructions. 

We also refer to these extended t-structures as Koszul, as they still satisfy the condition of Definition \ref{def:Koszult}. We have $\QCoh^\H(Y)^{> 0}_\Kzl = (\Coh^\H(Y)^{\leq 0}_\Kzl)^\perp$, hence $- \otimes \pi^*(\cV)[n]$ is left t-exact for any locally free $\cV \in \Coh_n^{\H}(X)$ since its left adjoint $- \ot \pi^*(\cV^\vee)[-n]$ preserves $\Coh^\H(Y)_K^{< 0}$. That $- \otimes \pi^*(\cV)[n]$ is right t-exact follows since $- \otimes \pi^*(\cV^\vee)[-n]$ is also its right adjoint, and is left t-exact by the same argument. The same applies to $\IndCoh^\H(Y)$. 

Given a closed $\H$-equivariant subscheme $i: Z \to X$ and a complementary open  subscheme $j: U \to X$, we denote by $\pi_Z: Y_Z \to Z$ and $\pi_U: Y_U \to U$ the obvious base changes. Abusing notation we denote the base changes $i: Y_Z \to Y$ and $j: Y_U \to Y$ of $i$ and $j$ by the same letters. Our next goal is to show that a Koszul or local t-structure on $\Coh^\H(Y)$ induces one on $\Coh^\H(Y_Z)$ and $\Coh^\H(Y_U)$. This will require us to assume $X/\H$ has the resolution property, i.e. that any $\cF \in \Coh^\H(X)^\heartsuit$ can be written as the quotient of a locally free sheaf. 

\begin{Proposition}\label{prop:Ztstructure}
	Let $(\Coh^\H(Y)_K^{\le 0}, \Coh^\H(Y)_K^{\ge 0})$ be a t-structure which is Koszul over $X$, let $i: Z \to X$ be a closed $\H$-invariant subscheme, and suppose that $X/\H$ has the resolution property. Then there exists a unique t-structure on $\Coh^\H(Y_Z)$ such that $i_*$ is t-exact, and this t-structure is also Koszul over $X$. The same statements hold for local t-structures. 
\end{Proposition}

This result is an extension of \cite[Thm. 2.3.5]{Pol07}.
The proof will use the following Lemma, where $\cHom(\pi^*(\cF),-)$ denotes both the usual functor on $\QCoh^\H(Y)$ and the unique continuous functor on $\IndCoh^\H(Y)$
with the same restriction to $\IndCoh^\H(Y)^+ \cong \QCoh^\H(Y)^+$.  

\begin{Lemma}\label{lem:cHomlefttexact}
	Let $(\Coh^\H(Y)_K^{\le 0}, \Coh^\H(Y)_K^{\ge 0})$ be a t-structure which is Koszul over $X$, let $\cF \in \Coh^\H_{\,0}(X)^\heartsuit$, and suppose that $X/\H$ has the resolution property. Then $\cHom(\pi^*(\cF),-)$ is left t-exact for the induced t-structures on $\QCoh^\H(Y)$ and $\IndCoh^\H(Y)$. The same holds for a local t-structure and any $\cF \in \Coh^\H(X)^\heartsuit$. 
\end{Lemma}
\begin{proof}
	For $\QCoh^\H(Y)$ we may show the left adjoint $- \ot \pi^*(\cF)$ is right t-exact. By hypothesis $\cF$ has a resolution whose degree $-n$ term $\cV_n$ is locally free and of weight zero. Writing $\cF_n \in \QCoh^\H(X)^{\geq - n}$ for its stupid truncation, we have $\cF \cong \colim \cF_n$ since the standard t-structure is left separated. It thus suffices to show $- \ot \pi^* (\cF_n)$ is right t-exact for all~$n$. But $- \ot \pi^*(\cV_n)[n]$ is right t-exact for all $n$ by hypothesis. Given the exact triangles $\cF_{n-1} \to \cF_n \to \cV_n[n]$, the claim now follows by induction on $n$. 
	
	Since $\cHom(\pi^*(\cF),-)$ is continuous on $\IndCoh^\H(Y)$ and $\IndCoh(Y)^{\geq 0}_\Kzl \cong \Ind(\Coh(Y)^{\geq 0}_\Kzl)$ is closed under filtered colimits, it suffices to show $\cHom(\pi^*(\cF),\cF') \in \IndCoh(Y)^{\geq 0}_\Kzl$ for any $\cF' \in \Coh(Y)^{\geq 0}_\Kzl$. This is equivalent to $\Map_{\IndCoh^\H(Y)}(\cF'', \cHom(\pi^*(\cF),\cF'))$ being contractible for any $\cF'' \in \Coh(Y)^{< 0}_\Kzl$. But we have shown this for the corresponding mapping space in $\QCoh^\H(Y)$, and these are are equivalent since $\cF'$ and thus $\cHom(\pi^*(\cF),\cF')$ are bounded below for the standard t-structure. The local case follows the same way. 
\end{proof}

\begin{proof}[Proof of Prop. \ref{prop:Ztstructure}]
	We first claim $\QCoh^\H(Y_Z)^{\leq 0}_\Kzl := i_*^{-1} \QCoh^\H(Y)^{\leq 0}_\Kzl$ defines a t-structure on $\QCoh^\H(Y_Z)$ for which $i_*$ is t-exact. By Proposition \ref{prop:induce} it suffices to show $i_* i^!: \QCoh^\H(Y) \to \QCoh^\H(Y)$ is left t-exact. This follows from Lemma \ref{lem:cHomlefttexact} since $i_* i^! \cong \cHom(i_*(\cO_{Y_Z}),-) \cong \cHom(\pi^* i_* (\cO_Y),-)$. 
	
	This t-structure on $\QCoh^\H(Y_Z)$ restricts to $\Coh^\H(Y_Z)$ by Proposition \ref{prop:restrict}, given the fact that $i_*^{-1}(\Coh(Y)) = \Coh(Y_Z)$ \cite[Prop. 5.6.1.1]{LurSAG}, hence $i_*^{-1}(\Coh^\H(Y)) = \Coh^\H(Y_Z)$. We caution that $\Coh(Y_Z)$ is not necessarily the preimage of $\Coh(Y)$ in $\IndCoh(Y_Z)$. The restricted t-structure is uniquely determined by t-exactness of $i_*$ since $i_*$ is conservative on $\Coh^\H(Y_Z)$, and it is Koszul over $X$ by the projection formula.  The local case follows the same way. 
\end{proof}

\begin{Proposition}\label{prop:Utstructure}
	Let $(\Coh^\H(Y)_K^{\le 0}, \Coh^\H(Y)_K^{\ge 0})$ be a t-structure which is Koszul over $X$, let $j: U \to X$ be an open $\H$-invariant subscheme, and suppose that $X/\H$ has the resolution property. Then there exists a unique t-structure on $\Coh^\H(Y_U)$ such that $j^*$ is t-exact, and this t-structure is also Koszul over $X$. The same statements hold for local t-structures. 
\end{Proposition}

\begin{Lemma}\label{lem:i!i*v2}
	Under the hypotheses of Proposition \ref{prop:Ztstructure}, the natural map $\cH_K^0(i_*i^! \cF) \to \cH^0_\Kzl(\cF)$ is a monomorphism for any for any $\cF \in \IndCoh^\H(Y)^{\geq 0}_K$. 
\end{Lemma}
\begin{proof}
	Consider the exact triangle $\cI \to \cO_X \to i_*(\cO_Z)$ in $\Coh^\H_{\,0}(X)$. The counit $i_* i^!(\cF) \to \cF$ is equivalently the first map in the associated triangle
	$$ \cHom(\pi^*i_*(\cO_Z), \cF) \to \cHom(\pi^*(\cO_X), \cF) \to \cHom(\pi^*(\cI), \cF) $$ in $\IndCoh^\H(Y)$. 
	By Lemma \ref{lem:cHomlefttexact} all three terms belong to $\IndCoh^\H(Y)^{\geq 0}_\Kzl$, hence it follows from the associated long exact sequence in $\IndCoh^\H(Y)^\heartsuit_\Kzl$ that $\cH_\Kzl^0(i_*i^! (\cF)) \to \cH_\Kzl^0(\cF)$ is a monomorphism. 
\end{proof}

\begin{proof}[Proof of Prop. \ref{prop:Utstructure}]
	We first claim $\IndCoh(Y_U)^{\leq 0}_\Kzl := \la j^* \Coh^\H(Y)^{\leq 0}_\Kzl \ra$ defines a t-structure on $\IndCoh^\H(Y_U)$ for which $j^*$ is t-exact. By Proposition \ref{prop:induce} it suffices to show $j_* j^*$ is left t-exact. Since $j_* j^*$ is continuous and $\IndCoh^\H(Y)^{\ge 0}_K \cong \Ind(\Coh^\H(Y)^{\geq 0}_K)$ is closed under filtered colimits, it further suffices to show $j_* j^*(\cF) \in \IndCoh^\H(Y)^{\ge 0}_K$ for any $\cF \in \Coh^\H(Y)^{\geq 0}_K$. 
	
	Consider the exact triangle
	\begin{equation}\label{eq:ijtriangle}
		\colim_{Z} i_{Z*} i_Z^! (\cF) \to \cF \to j_* j^* (\cF)
	\end{equation}
	in $\IndCoh^\H(Y)$, where the colimit is over all $\H$-invariant closed subschemes $i_Z: Z \to X$ with set-theoretic support $X \smallsetminus U$ (again we write $i_Z$ for the base change $Y_Z \to Y$). Note that exactness of (\ref{eq:ijtriangle}) follows from its exactness in the nonequivariant setting \cite[Prop. 7.4.5]{GR14}, since all terms belong to $\IndCoh^\H(Y)^+$ and the forgetful functor is conservative on $\IndCoh^\H(Y)^+$ (and among closed subschemes supported on $X \smallsetminus U$, the $\H$-invariant ones are left cofinal). It follows from Proposition \ref{prop:Ztstructure} that each $i_{Z*} i_Z^! (\cF)$ belongs to $ \IndCoh^\H(Y)^{\ge 0}_K$, and by Lemma \ref{lem:i!i*v2} each $\cH^0_K(i_{Z*} i_Z^! (\cF)) \to \cH^0_K(\cF)$ is a monomorphism. But $\cH^0_K(\colim i_{Z*} i_Z^! (\cF)) \to \cH^0_K(\cF)$ is then a monomorphism, since the t-structure is compatible with filtered colimits, hence $\cH^0_\Kzl$ is continuous and we have $\cH^0_K(\colim i_{Z*} i_Z^! (\cF)) \cong \colim \cH^0_K(i_{Z*} i_Z^! (\cF))$. It now follows from the long exact sequence associated to (\ref{eq:ijtriangle}) that $j_* j^*(\cF) \in \IndCoh^\H(Y)^{\ge 0}_K$. 
	
	Finally, we show that this t-structure restricts to $\Coh^\H(Y_U)$. Given $\cF \in \Coh^\H(Y_U)$, it follows from $X$ being Noetherian, $\H$ being affine, and $\cF \cong j^*j_*(\cF)$ that
	we can find  $\tcF \in \Coh^\H(Y)$ so that $\cF$ is a retract of $j^* (\tcF)$ \cite[Prop. 9.5.2.3]{LurSAG}. Then $\tau^{\geq 0}_\Kzl \cF$ is a retract of $\tau^{\geq 0}_\Kzl j^* (\tcF) \cong j^* \tau^{\geq 0}_\Kzl (\tcF)$, and the claim follows since $\Coh^\H(Y)^{\geq 0}_\Kzl$ is closed under retracts. A similar argument shows this t-structure is Koszul over $X$, and the local case is the same. 
\end{proof}

\subsection{Koszul t-structures on bundles}\label{sec:koszullift}

In this section we prove the following result for lifting a Koszul t-structure on $X$ to a $\H$-equivariant vector bundle $\pi: V \to X$ on which 
$\C^\times$ acts with weight one. This is the specialization of Theorem \ref{thm:koszul} in the case $W = V_1 = V_2$. 

\begin{Proposition}\label{prop:koszul-t-bundles}
	Given a bounded, Noetherian, Koszul t-structure on $\Coh^{\H}(X)$, there exists a unique t-structure on $\Coh^{\H}(V)$ such that the pullback
	$$\sigma^*: \Coh^{\H}(V) \to \Coh^{\H}(X)$$
along the zero section $\sigma: X \to V$ is t-exact. This t-structure is bounded, Noetherian, and Koszul over $X$, and is Artinian if the t-structure on $\Coh^\H(X)$ is. 
\end{Proposition}

We write $m$ for the rank of $V$ and $\cV \in \Coh^\H_1(X)$ for its sheaf of sections. An essential role is played by the Koszul resolution of $\sigma_*(\cO_X)$. Its degree $-\ell$ term is $\pi^*(\Lambda^\ell(\cV^\vee))$, and we write $K_\ell \in \Coh^\H(V)$ for its stupid truncation to degrees~$\geq - \ell$. In particular, we have $K_0 \cong \cO_V$, $K_m \cong \sigma_*(\cO_X)$, and the cone over $K_{\ell -1} \to K_\ell$ is $\pi^*(\Lambda^\ell(\cV^\vee))[\ell]$. 

We begin by inducing a t-structure on $\QCoh^\H(V)$ for which $\pi^*$ is t-exact and $\pi_*$ left t-exact. This is done using Proposition \ref{prop:induce} with $\Phi = \pi^*: \QCoh^{\H}(X) \to \QCoh^\H(V)$ and $\catC' = \Coh^\H(X)^{\le 0}_K$. The composition $\Phi^R \Phi \cong \pi_* \pi^* \cong - \otimes \pi_*(\O_V)$ is left t-exact since $\pi_*(\O_V) \cong \bigoplus_\ell \Sym^\ell (\cV^\vee)$ and since the t-structure on $\Coh^\H(X)$ is Koszul. Explicitly, the resulting t-structure is given by
$$ \QCoh^\H(V)^{\leq 0}_{\Kzl} := \la \pi^*\Coh^\H(X)^{\leq 0}_{\Kzl} \ra, \quad \QCoh^\H(V)^{\geq 0}_{\Kzl} := (\pi^*\Coh^\H(X)^{< 0}_{\Kzl})^\perp. $$ 

\begin{Lemma}\label{lem:koszultexists}
The functors $\sigma^*$ and $\sigma_*$ are t-exact with respect to the Koszul t-structures on $\QCoh^\H(X)$ and $\QCoh^\H(V)$.
\end{Lemma}
\begin{proof}
Suppose $\cF \in \QCoh^\H(X)^{> 0}_K$ and take $\cF' \in \Coh^\H(X)^{\le 0}_K$. Then 
$$\Map_{\QCoh^\H(V)}(\pi^* \cF', \sigma_*(\cF)) \cong \Map_{\QCoh^\H(X)}(\cF', \pi_* \sigma_* \cF) \cong \Map_{\QCoh^\H(X)}(\cF', \cF)$$ 
is contractible. It follows that $\sigma_*$ is left t-exact and its left adjoint $\sigma^*$ is right t-exact. 

To see that $\sigma_*$ is right t-exact, fix $\cF \in \Coh^\H(X)^{\le 0}_K$ and notice that 
	$$\sigma_*(\cF) \cong \sigma_* \sigma^* \pi^*(\cF) \cong \pi^*(\cF) \otimes \sigma_*(\O_X) \cong \pi^*(\cF) \otimes K_m.$$
	We have an exact triangle
	$$\pi^*(\cF) \otimes K_{\ell-1} \to \pi^*(\cF) \otimes  K_{\ell} \to \pi^*(\cF \otimes \Lambda^\ell(\cV^\vee) [\ell])$$
	for any $\ell > 0$. We have $\cF \otimes \Lambda^\ell(\cV^\vee) [\ell] \in \Coh^\H(X)^{\leq 0}_\Kzl$ since the t-structure is Koszul. By induction it follows that $\pi^*(\cF) \otimes  K_{\ell} \in \la \pi^* \Coh^\H(X)_K^{\le 0} \ra$ for all $\ell \geq 0$, the $\ell = 0$ case holding since $K_0 \cong \O_V$. 
 It further follows that the right adjoint $\sigma^!$ of $\sigma_*$ is left t-exact. But $\sigma^!$ differs from $\sigma^*$  by tensoring with $\Lambda^m(\cV)[-m]$, hence $\sigma^*$ is also left t-exact since the t-structure is Koszul. 
\end{proof}

\begin{Lemma}\label{lem:pi}
The smallest stable subcategory of $\QCoh^\H(V)$ containing $\pi^*(\Coh^\H(X))$ and closed under retracts is $\Coh^\H(V)$.
\end{Lemma}
\begin{proof}
This is standard and follows the proof of \cite[Thm. 5.4.17]{CG97}, but we include a sketch since the cited statement only concerns K-theory. Write $j$ for the open embedding of $V$ into its projective completion $\P_V$. Given $\cF \in \Coh^\H(V)$, let $\tcF \in \Coh^\H(\P_V)$ be such that $\cF$ is a retract of $j^*(\tcF)$. On $\P_V \times_X \P_V$ we have the Beilinson resolution of the diagonal $\O_{\Delta} \in \Coh^\H(\P_V \times_X \P_V)$, whose terms are of the form $\cE_i := \pi_1^* \Lambda^i(\cV/\O_{\P_V}(-1)) \otimes \pi_2^* \O_{\P_V}(-i)$. Here $\pi_1,\pi_2: \P_V \times_X \P_V \to \P_V$ are the two natural projections. 

Using this resolution we compute the left-hand side of $\pi_{2*}(\pi_1^*(\tcF) \otimes \O_\Delta) \cong \tcF$. Each term $\pi_{2*}(\pi_1^*(\tcF) \otimes \cE_i)$ can be rewritten using base change and the projection formula as 
$$\pi^* \pi_*(\tcF \otimes  \Lambda^i(\cV/\O_{\P_V}(-1))) \otimes \O_{\P_V}(-i),$$
where we still write $\pi: \P_V \to X$ for the projection. Since $\pi_*(\tcF \otimes  \Lambda^i(\cV/\O_{\P_V}(-1))) \in \Coh^\H(X)$ and $\O_{\P_V}(-i)$ restricts to a trivial bundle on $V$, this implies $j^*(\tcF)$ is an iterated cone of pullbacks from $\Coh^\H(X)$, and the claim follows.  
\end{proof}

\begin{Lemma}\label{lem:koszulconserv}
If $\cF \in \QCoh^\H(V)_{\Kzl}^{+}$ then either $\cF \in \Coh^\H(V)^\perp$ or there exists $n$ such that $\cH^n_\Kzl(\sigma^*(\cF))$ and $\cH^n_\Kzl(\pi_*(\cF))$ are nonzero but $\cH^k_\Kzl(\sigma^*(\cF)) \cong \cH^k_\Kzl(\pi_*(\cF)) \cong 0$ for all $k < n$. In particular, $\sigma^*$ is conservative on $\QCoh^\H(V)_{\Kzl}^{\geq 0}$. 
\end{Lemma}
\begin{proof}
Suppose $\cF \in \QCoh^\H(V)_K^+$ with $\cF \not\in \Coh^\H(V)^\perp$. The latter condition means that there exists some $\cF' \in \Coh^\H(V)$ so that $\Map(\cF',\cF)$ is not contractible. By Lemma \ref{lem:pi} this means that there exists some $\cF'' \in \Coh^\H(X)$ such that $\Map(\pi^*(\cF''), \cF)$ is not contractible. By adjunction this subsequently implies that $\pi_*(\cF) \not\in \Coh^\H(X)^\perp$. Note that $\Coh^\H(V)^\perp \cong 0$ if $\H$ is an algebraic group, but in examples with infinite-type $\H$ this is not clear. 

Since $\pi_*$ is left t-exact it follows that there exists a largest $n$ such that $\pi_*(\cF) \in  \QCoh^\H(X)_K^{\ge n}$. To prove the claim it then suffices to show $\sigma^*(\cF) \in  \QCoh^\H(X)_{\Kzl}^{\ge n}$ and $\cH^n_\Kzl(\sigma^*(\cF)) \ncong 0$. To see this, first note that $\sigma^*(\cF) \cong \pi_* \sigma_* \sigma^* (\cF) \cong \pi_*(\cF \otimes \sigma_*(\O_X)) \cong \pi_*(\cF \otimes K_m)$. For all $\ell > 0$ we have an exact triangle 
$$\pi_*(\cF \otimes K_{\ell-1}) \to \pi_*(\cF \otimes  K_{\ell}) \to \pi_*(\cF) \otimes \Lambda^\ell(\cV^\vee) [\ell]$$
where we use $ \pi_*(\cF \otimes \pi^*(\Lambda^\ell(\cV^\vee)) \cong \pi_*(\cF) \otimes \Lambda^\ell(\cV^\vee)$ in writing the last term. But $\pi_*(\cF) \otimes \Lambda^\ell(\cV^\vee) [\ell]$ belongs to  $\QCoh^\H(X)_{\Kzl}^{\ge n}$ for all $\ell \geq 0$, hence by induction so does $\pi_*(\cF \otimes  K_{\ell})$ (recall that $K_0 \cong \cO_V$). But then it further follows by induction that $\cH^n_\Kzl(\pi_*(\cF \otimes  K_{\ell})) \ncong 0$ for all $\ell \geq 0$, and the claim follows taking $\ell = m$. 
\end{proof}

\begin{proof}[Proof of Proposition \ref{prop:koszul-t-bundles}]
Consider $\cF \in \Coh^\H(V)$ so that $\sigma^*(\cF) \in \Coh^\H(X)_\Kzl^{[a,b]}$. We have $\sigma^* \tau^{> b}_{\Kzl}(\cF) \cong \tau^{> b}_{\Kzl} \sigma^* (\cF) \cong 0$, so  $\tau^{> b}_{\Kzl}(\cF) \cong 0$ by Lemma \ref{lem:koszulconserv} (the map $\cF \to \tau^{> b}_{\Kzl}(\cF)$ rules out the possibility that $\tau^{> b}_{\Kzl}(\cF) \in \Coh^\H(V)^\perp$).  Now let us consider $\cF \to \tau^{\ge b}_{\Kzl}(\cF) \cong \cH^b_{\Kzl}(\cF)$. If we can show that $\cF' := \cH^b_{\Kzl}(\cF)$ is coherent, then by Proposition \ref{prop:restrict} and descending induction we may conclude that the Koszul t-structure restricts to $\Coh^\H(V)$. 

We first show that the map 
	\begin{equation}\label{eq:inj0}
	\cH^0_\Kzl(\pi^* \pi_* \cF') \to \cF' 
	\end{equation} 
	induced by the adjunction $\pi^* \pi_* \cF' \to \cF'$ is a monomorphism. By Lemma \ref{lem:koszulconserv} it suffices to show the map
	\begin{equation}\label{eq:inj}
		\cH^0_\Kzl(\pi_* (\cF')) \cong \cH^0_\Kzl(\sigma^* \pi^* \pi_* (\cF'))  \to \sigma^*(\cF')
		\end{equation}
	obtained by applying $\sigma^*$ to (\ref{eq:inj0}) is a monomorphism. On the other hand, (\ref{eq:inj}) is also obtained from the adjunction map $\cF' \to \sigma_* \sigma^*(\cF')$ by applying $\pi_*$ and then $\cH^0_\Kzl$. Since $\pi_*$ is left t-exact and $\sigma_* \sigma^*(\cF') \in \QCoh^\H(X)^\heartsuit$ by Lemma~\ref{lem:koszultexists}, it follows that (\ref{eq:inj}) is a monomorphism if $\cF' \to \sigma_* \sigma^* (\cF')$ is a monomorphism. Applying $\sigma^*$ again and using that it is conservative on $\QCoh^\H(V)_{\Kzl}^\heartsuit$ it suffices to show that $\sigma^*(\cF') \to \sigma^* \sigma_* \sigma^*(\cF')$ is a monomorphism. But this is clearly the case since this map splits. 

Since $\sigma^*(\cF')$ is coherent and (\ref{eq:inj}) is a monomorphism, it follows by Lemma \ref{lem:noetherian-extend} that $\cH^0_\Kzl(\pi_* (\cF'))$ is coherent, hence so is $\pi^*\cH^0_\Kzl(\pi_* (\cF'))$. Since $\sigma^*(\cF')$ is nonzero so is $\cH^0_\Kzl(\pi_* (\cF'))$ by Lemma \ref{lem:koszulconserv}, hence so is $\pi^*\cH^0_\Kzl(\pi_* (\cF'))$ since $\pi^*$ is conservative. We have $\pi^*\cH^0_\Kzl(\pi_* (\cF')) \cong \cH^0_\Kzl(\pi^* \pi_* (\cF'))$ since $\pi^*$ is t-exact, hence (\ref{eq:inj0}) being a monomorphism implies that $\cF'_1 :=\cH^0_\Kzl( \pi^*\pi_* (\cF'))$ is a nonzero coherent subobject of $\cF'$. 
	
	If $\cF'$ is not coherent, then repeating the construction with $\cF'/\cF'_1$ we obtain an ascending chain $\cF'_1 \subset \cF'_2 \subset \cdots$ of coherent subobjects of $\cF'$. But since $\sigma^*(\cF') \in \Coh^\H(X)^{\heartsuit}_K$ is Noetherian and $\sigma^*$ is conservative on $\QCoh^\H(V)_K^{\heartsuit}$ (Lemma \ref{lem:koszulconserv}) this is not possible. Thus $\cF'$ must be coherent. 
	
The functor $\sigma^*: \Coh^\H(V) \to \Coh^\H(X)$ is conservative since the scaling $\C^\times \subset \H$ ensures that the closure of any $\H$-orbit (and hence the support of any $\H$-equivariant coherent sheaf on~$V$) intersects~$X$ nontrivially. By Proposition \ref{prop:conservative} this means that the t-structure on $\Coh^\H(V)$ inherits the remaining properties from the t-structure on $\Coh^\H(X)$. Moreover, the t-structure on $\Coh^\H(V)$ is Koszul over $X$ since $\sigma^*(\cF \otimes \pi^*(\cV)) \cong \sigma^*(\cF) \otimes \cV$. 
\end{proof}
	
We end with the following compatibility, which will be used in the next section.

\begin{Lemma}\label{lem:koszul1}
	Let $j: V \to W$ be the inclusion of a $\H$-equivariant sub-bundle. 
	The functors $j_*$, $j^*$, and $j^!$ are t-exact with respect to the Koszul t-structures on $\Coh^\H(V)$ and $\Coh^\H(W)$. 
\end{Lemma}
\begin{proof}
First we show that $j^*$ is t-exact. Denote by $\pi_V,\pi_W$ the projections from $V,W$ to $X$ respectively. Since $\pi_W^* \cong j^* \pi_V^*$ it is clear that $j^*$ is right t-exact. 

To see left t-exactness we need to show that $ \Map_{\Coh^\H(V)}(\pi_V^* (\cF), j^* (\cF'))$ is contractible for any $\cF \in \Coh^\H(X)^{\le 0}_K$ and $\cF' \in \Coh^\H(W)^{> 0}_K$. By adjunction this is equivalent to checking that $\pi_{V*} j^*$ is left t-exact for the Koszul t-structure on $\QCoh^\H(X)$. Since  $\pi_{V*} j^* \cong \pi_{W*} j_* j^*$ and since $\pi_{W*}$ is left t-exact, it suffices to show that $j_* j^*$ is left t-exact. But $j_* j^* (-) \cong (-) \otimes j_* \O_V$ which can be seen to be t-exact using the Koszul resolution of $j_* \O_V$ (as in Lemma \ref{lem:koszultexists}).

Next note that $j^!$ is also t-exact since it differs from $j^*$ by tensoring with $\det(\cW/\cV[-1])$, which is t-exact. Lastly, $j_*$ is t-exact as it is right adjoint to $j^*$ and left adjoint to $j^!$.  
\end{proof}

\subsection{Koszul t-structures on intersections}\label{sec:koszul-ints} 

We now consider the general case of Theorem~\ref{thm:koszul}. We notate the maps used to define $Y$  as in the following Cartesian diagram. 
\begin{equation}\label{eq:bundles}
	\begin{tikzpicture}
		[baseline=(current  bounding  box.center),thick,>=\arrtip]
		\node (a) at (0,0) {$Y$};
		\node (b) at (3,0) {$V_1$};
		\node (c) at (0,-1.5) {$V_2$};
		\node (d) at (3,-1.5) {$W.$};
		\draw[->] (a) to node[above] {$i_1$} (b);
		\draw[->] (b) to node[right] {$j_1$} (d);
		\draw[->] (a) to node[left] {$i_2$}(c);
		\draw[->] (c) to node[above] {$j_2$} (d);
	\end{tikzpicture}
\end{equation}

\begin{proof}[Proof of Theorem \ref{thm:koszul}] 
We first claim $\QCoh^\H(Y)^{\leq 0}_\Kzl := i_{1*}^{-1} \QCoh^\H(V_1)^{\leq 0}_\Kzl$ defines a t-structure on $\QCoh^\H(Y)$ for which $i_{1*}$ is t-exact. By Proposition \ref{prop:induce} it suffices to show $i_* i^!: \QCoh^\H(Y) \to \QCoh^\H(Y)$ is left t-exact. It further suffices to show $j_{1*} i_{1*} i_1^!$ is left t-exact since $j_{1*}$ is conservative ($j_1$ is a closed immersion) and is t-exact by Lemma \ref{lem:koszul1}. We have 
$$j_{1*} i_{1*} i_1^! \cong j_{2*} i_{2*} i_1^! \cong j_{2*} j_2^! j_{1*},$$
which is t-exact since $j_{1*}$, $j_2^!$, and $j_{2*}$ are all t-exact by Lemma \ref{lem:koszul1}. 

The resulting t-structure restricts to $\Coh^\H(Y)$ by Proposition \ref{prop:restrict}, since $i_{1*}$ is t-exact and $i_{1*}^{-1}(\Coh^\H(V_1)) = \Coh^\H(Y)$ \cite[Prop. 5.6.1.1]{LurSAG}. It is Koszul over $X$ since $i_{1*}$ is t-exact and conservative. Moreover, it inherits the remaining properties (bounded, Noetherian/Artinian) from the t-structure on $\Coh^\H(X)$ by Proposition \ref{prop:conservative}.

Finally, to see that $\sigma_2^* i_{2*}$ is t-exact it suffices to check that $\sigma_2^* j_2^* j_{2*} i_{2*}$ is t-exact since $j_{2*}$ and $j_2^*$ are both t-exact (Lemma \ref{lem:koszul1}) and conservative. But we can rewrite this as $\sigma_2^* j_2^* j_{1*} i_{1*}$, which is t-exact since $i_{1*}$ is by construction and the rest are by Lemma \ref{lem:koszul1}. 
\end{proof}

Next we consider the local nature of Theorem \ref{thm:koszul}. Consider a $\H$-invariant closed subscheme $i: Z \to X$ and open subscheme $j: U := X \smallsetminus Z \to X$ as in Section \ref{sec:koszulity}. We again use the same notation for the base changes $i: Y_Z \to Y$ and $j: Y_U \to Y$. Recall from Proposition \ref{prop:Ztstructure} that if $X/\H$ has the resolution property, the Koszul t-structures on $\Coh^\H(X)$ and $\Coh^\H(Y)$ induce unique t-structures on $\Coh^\H(Z)$ and $\Coh^\H(Y_Z)$ such that $i_*$ is t-exact. By Proposition \ref{prop:Utstructure} the same is true for $\Coh^\H(U)$, $\Coh^\H(Y_U)$, and $j^*$. By the following result these t-structures are compatible with Theorem \ref{thm:koszul}, in the sense that the t-structure on $\Coh^\H(Y_Z)$ is determined by the one on $\Coh^\H(Z)$ by the same pattern as in Theorem \ref{thm:koszul} (likewise for $\Coh^\H(Y_U)$ and $\Coh^\H(U)$). 

\begin{Proposition}\label{prop:compatible}
Suppose $X/\H$ has the resolution property. 
The functors $\sigma^*_1 i_{1*}, \sigma^*_2 i_{2*}: \Coh^\H(Y_Z) \to \Coh^\H(Z)$ are t-exact with respect to the t-structures induced by Proposition \ref{prop:Ztstructure}. The functors $\sigma^*_1 i_{1*}, \sigma^*_2 i_{2*}: \Coh^\H(Y_U) \to \Coh^\H(U)$ are t-exact with respect to the t-structures induced by Proposition \ref{prop:Utstructure}. 
\end{Proposition}
\begin{proof}
Consider the Cartesian squares 
\begin{equation*}
	\begin{tikzpicture}[baseline=(current bounding box.center),thick,>=\arrtip]
		\newcommand*{\ha}{3}; \newcommand*{\vb}{-1.5};
		
		\node (a) at (0,0) {$Y_Z$};
		\node (a') at (\ha,0) {$V_1|_Z$};
		\node (a'') at (2*\ha,0) {$Z$};
		\node (b) at (0,\vb) {$Y$};
		\node (b') at (\ha,\vb) {$V_1$};
		\node (b'') at (2*\ha,\vb) {$X$,};
		
		\draw[->] (a) to node[above] {$i_1$} (a');
		\draw[<-] (a') to node[above] {$\sigma_1$} (a'');
		\draw[->] (b) to node[above] {$i_1$} (b');
		\draw[<-] (b') to node[above] {$\sigma_1$} (b'');
		
		\draw[->] (a) to node[left] {$i$} (b);
		\draw[->] (a') to node[right] {$i$} (b');
		\draw[->] (a'') to node[right] {$i$} (b'');
	\end{tikzpicture}
\end{equation*}
where for simplicity we overload the names of the maps. For $\cF \in \Coh^\H(Y_Z)$ we have $i_* \sigma^*_1 i_{1*} (\cF) \cong \sigma^*_1 i_* i_{1*} (\cF) \cong \sigma^*_1 i_{1*} i_* (\cF)$, where the first isomorphism is by base change. The claim follows since $i_*$ is conservative. Likewise we have the Cartesian squares
\begin{equation*}
	\begin{tikzpicture}[baseline=(current bounding box.center),thick,>=\arrtip]
		\newcommand*{\ha}{3}; \newcommand*{\hb}{3}; \newcommand*{\va}{-1.5};
		
		\node (a) at (0,0) {$Y_U$};
		\node (a') at (\ha,0) {$V_1|_U$};
		\node (a'') at (\ha+\hb,0) {$U$};
		\node (b) at (0,\va) {$Y$};
		\node (b') at (\ha,\va) {$V_1$};
		\node (b'') at (\ha+\hb,\va) {$X$.};
		
		\draw[->] (a) to node[above] {$i_1$} (a');
		\draw[<-] (a') to node[above] {$\sigma_1$} (a'');
		\draw[->] (b) to node[above] {$i_1$} (b');
		\draw[<-] (b') to node[above] {$\sigma_1$} (b'');
		
		\draw[->] (a) to node[left] {$j$} (b);
		\draw[->] (a') to node[right] {$j$} (b');
		\draw[->] (a'') to node[right] {$j$} (b'');
	\end{tikzpicture}
\end{equation*}
As in the proof of Proposition \ref{prop:Utstructure}, given $\cF \in \Coh^\H(Y_U)$ we can choose $\tcF \in \Coh^\H(Y)$ so that $\cF$ is a retract of $j^*(\tcF)$. The claim follows
$\sigma^*_1 i_{1*} j^*(\tcF) \cong \sigma^*_1 j^* i_{1*} (\tcF) \cong j^* \sigma^*_1 i_{1*}(\tcF)$, and since $\Coh^\H(U)^{\leq 0}_\Kzl$ and $\Coh^\H(U)^{\geq 0}_\Kzl$ are closed under retracts. 
\end{proof}

\subsection{Simple objects}\label{sec:simples}

We now discuss simple objects in the heart of the t-structure constructed in Theorem~\ref{thm:koszul}. As before we fix a bounded Koszul t-structure on $\Coh^\H(X)$, but now also assume it is finite-length. We study special cases in Propositions \ref{prop:koszul2} and \ref{prop:koszul3} before stating our most general result as Theorem \ref{thm:koszulsimples}. We begin with a $\H$-equivariant vector bundle $\pi: V \to X$ on which the scaling $\C^\times$ acts with weight one, as in Proposition \ref{prop:koszul-t-bundles}

\begin{Proposition}\label{prop:koszul2} 
	The functors 
	$$\pi^*: \Coh^\H(X)^\heartsuit_{\Kzl} \leftrightarrows  \Coh^\H(V)^\heartsuit_{\Kzl}: \sigma^*$$ 
	induce inverse bijections between simple objects on either side. 
\end{Proposition}
\begin{proof} 
Recall that $\pi^*$ is t-exact and that $\sigma^*$ is conservative on $\Coh^\H(V)$. It follows that $\cF \in \Coh^\H(V)^\heartsuit_{\Kzl}$ is simple if $\sigma^*(\cF)$ is simple. Since $\sigma^* \pi^*$ is the identity this implies that if $\cG \in \Coh^\H(X)^\heartsuit_{\Kzl}$ is simple, so is $\pi^*(\cG)$. The claim now follows if we show that any simple $\cF \in \Coh^\H(V)^\heartsuit_{\Kzl}$ is in the essential image of~$\pi^*$, since if $\cF \cong \pi^* (\cG)$ then  $\sigma^*(\cF) \cong \sigma^* \pi^* (\cG) \cong \cG$. 

By adjunction we have a nonzero map $\cF \to \sigma_* \sigma^* (\cF)$, whose target belongs to $\Coh^\H(V)^\heartsuit_{\Kzl}$ by Lemma \ref{lem:koszultexists}. It follows from the proof of that lemma that $\sigma_* \sigma^* (\cF) \cong \pi^* \sigma^*(\cF) \otimes \sigma_*(\O_X)$ has a filtration by the objects $\pi^* \sigma^*(\cF) \ot K_\ell$, that these also belong to $\Coh^\H(V)^\heartsuit_{\Kzl}$, and that the quotient of $\pi^* \sigma^*(\cF) \ot K_\ell$ by $\pi^* \sigma^*(\cF) \ot K_{\ell-1}$ is $\pi^* (\sigma^*(\cF) \otimes \Lambda^\ell(\cV^\vee)[\ell])$. It follows that $\cF$ is a subobject of $\pi^* (\sigma^*(\cF) \otimes \Lambda^\ell(\cV^\vee)[\ell])$ for some $\ell$. But the claim follows since the essential image of $\pi^*$ is closed under simple subobjects: by the first paragraph $\pi^*$ takes a composition series of any $\cG \in \Coh^\H(X)^\heartsuit_{\Kzl}$ to a composition series of $\pi^*(\cG)$. 
\end{proof}

Next we generalize to the case when the classical intersection $Y^\cl$ of $V_1$ and $V_2$ is a vector bundle. Equivalently, the fiberwise span of $V_1$ and $V_2$ is a sub-bundle $V_{12} \subset W$. 
In this case we will use the following special feature of $Y$.

\begin{Lemma}\label{lem:newsplits}
Suppose the classical intersection $Y^\cl$ of $V_1$ and $V_2$ is a sub-bundle of $W$. Then there is a canonical $\H$-equivariant splitting $p: Y \to Y^{\cl}$ of the natural map $\iota: Y^{\cl} \to Y$. Moreover,  $p_*(\O_{Y})  \cong \bigoplus_k \Lambda^k(\cW/\cV_{12})^\vee [k]$. 
\end{Lemma}
\begin{proof}
	The key point is that $Y^{\cl} \cong V_1 \times_{V_{12}} V_2$. We then have Cartesian squares 
	\begin{equation*}
	\begin{tikzpicture}[baseline=(current  bounding  box.center),thick,>=\arrtip]
		\newcommand*{\ha}{3}; \newcommand*{\hb}{3}; \newcommand*{\hc}{3};
		\newcommand*{\va}{-1.5}; \newcommand*{\vb}{-1.5};
		\node (aa) at (0,0) {$V_1 \times_{V_{12}} V_2$};
		\node (ab) at (\ha,0) {$Z$};
		\node (ac) at (\ha+\hb,0) {$V_1 \times_{V_{12}} V_2$};
		\node (ba) at (0,\va) {$V_{12}$};
		\node (bb) at (\ha,\va) {$V_{12} \times_W V_{12}$};
		\node (bc) at (\ha+\hb,\va) {$V_{12}$,};
		\draw[->] (aa) to node[left] {$ $} (ba); 
		\draw[->] (ab) to node[right] {$ $} (bb); 
		\draw[->] (ac) to node[right] {$ $} (bc); 
		\draw[->] (ba) to node[above] {$ $} (bb); 
		\draw[->] (bb) to node[above] {$ $} (bc); 
		\draw[->] (ab) to node[above] {$p$} (ac); 
		\draw[->] (aa) to node[above] {$\iota $} (ab);
	\end{tikzpicture}
\end{equation*} 
where $Z = (V_{12} \times_W V_{12}) \times_{V_{12}} V_1 \times_{V_{12}} V_2 \cong (V_{12} \times_W V_1) \times_{V_{12}} V_2 \cong V_1 \times_W V_2$. The composition along the top is the identity and defines the splitting $p$ since $V_1 \times_{V_{12}} V_2 \cong Y^{\cl}$. 

The last claim follows by base change and the fact that the pushforward of the structure sheaf of $V_{12} \times_W V_{12}$ to $V_{12}$ is $\bigoplus_k \Lambda^k(\cW/\cV_{12})^\vee [k]$ (where one can see this using the Koszul resolution of $\O_{V_{12}}$ as a sheaf on $W$).  
\end{proof}

\begin{Proposition}\label{prop:koszul3}
Suppose the classical intersection $Y^\cl$ of $V_1$ and $V_2$ is a sub-bundle of $W$ and consider the maps 
	$X \xleftarrow{\pi} Y^\cl \xrightarrow{\iota} Y \xrightarrow{p} Y^{\cl} \xleftarrow{\sigma} X$. 
Then the functors
	$$\iota_*\pi^*: \Coh^\H(X)^\heartsuit_{\Kzl} \leftrightarrows  \Coh^\H(Y)^\heartsuit_{\Kzl}: \sigma^*p_*$$ 
induce inverse bijections between simple objects on either side.
	\end{Proposition}
\begin{proof}
	Since $Y^\cl$ is an $\H$-equivariant bundle over $X$, $\Coh^\H(Y^\cl)$ inherits a Koszul t-structure via Proposition \ref{prop:koszul-t-bundles}. We begin by showing that $p_*$ and $p^*$ are t-exact with respect to this t-structure.  
	
	To show that $p_*$ is t-exact it suffices to check that pushforward by the composition $Y \xrightarrow{p} Y^{\cl} \xrightarrow{i} W$ is t-exact. But this composition can be factored as $Y \xrightarrow{i_1} V_1 \to W$ where both maps have t-exact pushforward. Likewise, to show $p^*$ is t-exact it suffices to show that the composition $i_* p^* p_*$ is t-exact. Since $p^* p_*(-)$ is tensoring with $p_*(\O_Y)$ this follows from Lemma \ref{lem:newsplits} and the fact that all our t-structures are Koszul. 
	
		In light of Proposition \ref{prop:koszul2} it now suffices to show that $\iota_*$ and $p_*$ induce inverse bijections between isomorphisms classes of simple objects in $\Coh^\H(Y^\cl)^\hs_\Kzl$ and $\Coh^\H(Y)^\hs_\Kzl$. We can prove this in essentially the same way as Proposition \ref{prop:koszul2}. 

Since $p_*$ is conservative it follows that $\cG \in \Coh^\H(Y)^\hs_{\Kzl}$ is simple if $p_*(\cG)$ is simple. Since $p_* \iota_*$ is the identity this implies that if $\cF \in \Coh^\H(Y^\cl)^\hs_{\Kzl}$ is simple, so is $\iota_*(\cF)$. The claim now follows if we show that any simple $\cG \in \Coh^\H(Y)^\hs_{\Kzl}$ is in the essential image of~$\iota_*$, since if $\cG \cong \iota_* (\cF)$ then  $p_*(\cG) \cong p_* \iota_* (\cF) \cong \cF$. By adjunction we have a nonzero map $p^* p_* (\cG) \to \cG$ in $\Coh^\H(Y)^\hs_{\Kzl}$, so it suffices to show that every composition factor of $p^*p_*(\cG)$ is in the essential image of $\iota_*$. This is true of $p^*(\cF)$ for any $\cF \in \Coh^\H(Y^\cl)^\hs_\Kzl$.
	
	Using the computation of $p_*(\O_Y)$ from Lemma \ref{lem:newsplits} we find that $\cH^{-\ell}(\O_Y) \cong \iota_*\pi^* \Lambda^\ell(\cW/\cV_{12})^\vee$.  It follows that for any $\ell$ we have an exact triangle
	$$p^*(\cF) \ot \tau^{\leq \ell-1} \cO_Y \to p^*(\cF) \ot \tau^{\leq \ell} \cO_Y \to p^*(\cF) \ot \iota_*\pi^*(\Lambda^\ell(\cV_{12}^\perp))[\ell].$$
	On the other hand, we have
	$$p^* (\cF) \ot \iota_*\pi^*(\Lambda^\ell(\cV_{12}^\perp))[\ell] \cong \iota_*(\iota^* p^*(\cF) \ot \pi^*(\Lambda^\ell(\cV_{12}^\perp))[\ell]) 
	\cong \iota_*(\cF \otimes \pi^*(\Lambda^\ell(\cV_{12}^\perp))[\ell]).$$
	As the last expression is manifestly in $\Coh^\H(Y)^\hs_\Kzl$, it follows by induction on $\ell$ that the sheaves $p^*(\cF) \ot \tau^{\leq \ell} \cO_Y$ are as well and define a filtration of $p^*(\cF)$ in $\Coh^\H(Y)^\hs_\Kzl$. The summands of the associated graded sheaf are then $\iota_*(\cF \otimes \pi^*(\Lambda^\ell(\cV_{12}^\perp))[\ell])$. But the argument of the preceding paragraph also shows that the essential image of $\iota_*$ contains any composition factor of such a pushforward, hence any composition factor of $p^*(\cF)$.   
\end{proof}

Finally, to understand simples when $V_1$ and $V_2$ do not necessarily intersect in a bundle we will use the following two results. These hold in the general context of a morphism $\pi: Y \to X$ as in Section \ref{sec:koszulity}, and assume $X/\H$ has the resolution property so that the t-structure on $Y$ restricts to open and closed $\H$-invariant subschemes. Below we interpret sheaves such as $j_*(\cF)$ as objects in $\IndCoh^\H(Y)$, similarly for $i^!(\cF)$, etc. 

\begin{Proposition}\label{prop:extend}
	Let $j: U \to X$ be a $\H$-invariant open subscheme. 
	If $\cF \in \Coh^\H(Y_U)_K^{\heartsuit}$ is simple, then there is a unique simple $\tcF \in \Coh^\H(Y)_K^{\heartsuit}$ such that $j^*(\tcF) \cong \cF$. Explicitly, $\tcF$ is the socle of $\cH_K^0(j_* \cF)$. Conversely, if $\tcF \in \Coh^\H(Y)_K^{\heartsuit}$ is simple then $j^*(\tcF)$ is either simple or zero. 
\end{Proposition}
\begin{proof}
	Choose some $\cG \in \Coh^\H(Y)$ so that $\cF$ is a retract of $j^* (\cG)$. Since $j^*$ is t-exact we can assume $\cG \in \Coh^\H(Y)_K^\heartsuit$. Since $\Coh^\H(Y)_K^\heartsuit$ is finite-length we can assume $\cG$ is the minimal object with this property. We now take $\tcF$ to be the quotient of $\cG$ by all strictly smaller sub-objects. Then $\tcF$ is by construction simple with $j^*(\tcF) \cong \cF \oplus \cF'$ for some $\cF' \in \Coh^\H(Y)_K^\heartsuit$. 
	
	Consider the map $\tcF \to j_*(\cF')$ adjoint to the surjection $j^*(\tcF) \to \cF'$. If $\cF' \ne 0$ then this map must be injective since $\tcF$ is simple. But then applying $j^*$ we get an injection 
	$$\cF \oplus \cF' \cong j^*(\tcF) \to j^* j_*(\cF') \cong \cF'$$
which is not possible. Thus $j^*(\tcF) \cong \cF$. 
	
	Next we show $\tcF$ is the socle of $\cH^0_K(j_* \cF)$. Suppose $\tcF' \in \Coh^\H(Y)^\heartsuit$ is another simple such that $j^*(\tcF') \cong \cF$. Then $\Map(\tcF',  \cH^0_K(j_*(\cF))) \cong \Map(j^*(\tcF'), \cF)$ is nonzero, hence $\tcF' \hookrightarrow \cH^0_K(j_*(\cF))$ since the former is simple. Take the exact sequence
$$0 \to \tcF \to \cH^0_K(j_* \cF) \to \cQ \to 0$$
and consider the composition $\tcF' \to \cQ$. If $\tcF' \to \cQ$ is nonzero it is a monomorphism since $\tcF'$ is simple. Applying $j^*$ we get a monomorphism $j^*(\tcF') \to j^*(\cQ) \cong 0$, hence $j^*(\tcF') \cong 0$. This is a contradiction since $\Map(\tcF',  \cH^0_K(j_*(\cF))) \cong \Map(j^*(\tcF'), \cF)$ is then contractible. Thus $\tcF' \to \cQ$ must be zero so that $\tcF' \into \cH^0_K(j_* \cF)$ factors through $\tcF$ and hence $\tcF' \cong \tcF$ since both are simple.
	
Finally, suppose $\tcF \in \Coh^\H(Y)_K^{\heartsuit}$ is any simple such that $j^*(\tcF)$ is nonzero. Since $\Coh^\H(Y_U)_K^{\heartsuit}$ is Noetherian we may choose a simple quotient $j^*(\tcF) \to \cF$. By adjunction this gives a nonzero map $\tcF \to \cH^0_K(j_*(\cF))$. But it follows from the rest of the proof that $\tcF$ is then the socle of $\cH^0_K(j_*(\cF))$, and that $j^*(\tcF) \cong \cF$. 
\end{proof}

\begin{Proposition}\label{prop:simples}
	Let $i: Z \to X$ be a $\H$-invariant closed subscheme and $j: U \to X$ its complement. Then $i_*$ takes simples to simples. Moreover, for every simple $\cF \in \Coh^\H(Y)_K^{\heartsuit}$, either $\cF$ is the socle of $\cH^0_K(j_* j^* \cF))$ or $\cF \cong i_* \cH^0_K(i^! \cF)$ where $\cH^0_K(i^! \cF)$ is simple. 
\end{Proposition}
\begin{proof}
	We first show $i_*$ preserves simples. If $\cF' \in \Coh^\H(Y_Z)^\heartsuit_K$ is simple, we may choose a simple quotient $i_*(\cF') \to \cF$ since $\Coh^\H(Y)^\heartsuit$ is Noetherian and $i_*(\cF')$ is nonzero. By adjunction we have a nonzero map $\cF' \to i^!(\cF)$. This factors through a nonzero map $\cF' \to \cH^0_K(i^! (\cF))$ since $i^!$ is left t-exact. Since $\cF'$ is simple this is a monomorphism, hence we have a monomorphism $i_*(\cF') \into i_*\cH^0_K( i^!(\cF))$. But $i_*\cH^0_K( i^!(\cF)) \hookrightarrow \cF$ is a monomorphism by Lemma \ref{lem:i!i*v2}, hence $i_*(\cF') \cong \cF$. 
	
	Conversely, suppose $\cF \in \Coh^\H(Y)^\heartsuit_K$ is simple. Then since $i_*\cH^0_K( i^!(\cF)) \hookrightarrow \cF$ is a monomorphism and $i_*$ is conservative, either $i_*\cH^0_K( i^!(\cF)) \cong \cF$ or $\cH^0_K( i^!(\cF)) \cong 0$. In light of Proposition~\ref{prop:extend}, the claim follows if we show that $\cH^0_K(i^! \cG) = 0$ implies $j^*(\cF)$ is nonzero. 
	
	We can assume that $Z$ is reduced. If $j^* \cG = 0$ then there exists some possibly non-reduced $Z'$ with $Z'_{red} = Z$ such that $\cH^0_K(i'^! \cG) \ne 0$ where $i': Y_{Z'} \to Y$. But $\O_{Z'}$ has a finite filtration by objects pushed forward from $Z$ and, since $\cH^0_K(\cHom(\O_{Y_{Z'}}, i'^! \cG)) \cong \cH^0_K(i'^! \cG) \ne 0$, this means that there exists some nonzero $\cF \in \Coh_0^\H(Z)^{\le 0}$ such that $\cH^0_K(\cHom(i_{red*} \pi^* \cF,  i'^! \cG)) \ne 0$ where $i_{red}: Y_Z \to Y_{Z'}$ is the natural map. But this means that
	$$\cH^0_K(\cHom(\pi^* \cF, i^! \cG)) \cong \cH^0_K(\cHom(\pi^* \cF, i_{red}^! i'^! \cG)) \ne 0$$
	which is a contradiction since, by assumption, $\cH^k_K(i^! \cG) = 0$ for $k \le 0$ and $\cHom(\pi^* \cF, -)$ is left t-exact. 
\end{proof}

\begin{Proposition}
	Under the above hypotheses, the t-structures on $\Coh^\H(Y_Z)$ and $\Coh^\H(Y_U)$ are bounded and finite-length.
\end{Proposition}
\begin{proof}
	The category $\Coh^\H(Y_Z)^\heartsuit$ is bounded and finite-length because $i_*: \Coh^\H(Y_Z) \to \Coh^\H(Y)$ is t-exact and conservative and $\Coh^\H(Y)^\heartsuit$ is bounded and finite-length. 
	
	To see $\Coh^\H(Y_U)^\heartsuit$ is bounded consider $\cF \in \Coh^\H(Y_U)^\heartsuit$ and an extension $\tcF \in \Coh^\H(Y)^\heartsuit$ such that $\cF$ is a retract of $j^* \tcF$. Since $\tcF$ has bounded cohomology and $j^*$ is t-exact it follows that $j^*(\tcF)$ and thus also $\cF$ have bounded cohomology. Moreover, since $\Coh^\H(Y)^\heartsuit$ has finite-length we can consider a decomposition series $\tcF_1 \subset \dots \subset \tcF_m = \tcF$ where each quotient $\tcF_{i+1}/\tcF_i$ is simple. By Proposition \ref{prop:extend} each $j^*(\tcF_{i+1}/\tcF_i)$ is either zero or simple which implies that $\cF$ also has a finite composition series. 
\end{proof}

We can now state our main result about simples. 

\begin{Theorem}\label{thm:koszulsimples}
	Consider $\pi: Y = V_1 \times_W V_2 \to X$ where $\Coh^\H(X)$ carries a finite-length Koszul t-structure. Then there exists a bijection between simples in $\Coh^\H(X)_K^{\heartsuit}$ and simples in $\Coh^\H(Y)_K^{\heartsuit}$ compatible with the bijection from Proposition \ref{prop:koszul3}.
\end{Theorem}
\begin{proof}
	Consider a simple in $\Coh^\H(Y)_K^{\heartsuit}$. By Proposition \ref{prop:simples} it is of the form $i_* \cF$ for some simple $\cF \in \Coh^\H(Y_Z)_K^{\heartsuit}$ where $Z$ is its reduced, irreducible support in $X$. Note that the t-structure here on $\Coh^\H(Y_Z)$ is, by Proposition \ref{prop:compatible}, equivalently the one induced from $\Coh^\H(Z)$ by applying Theorem \ref{thm:koszul} or the one induced from $\Coh^\H(Y)$ by applying Proposition~\ref{prop:Ztstructure}. 
	
	Inside $Z$ consider a dense open $U \subset Z$ over which $V_1|_U$ and $V_2|_U$ intersect in a bundle. By Proposition \ref{prop:simples} $\cF|_{Y_U}$ is simple and $\cF$ is its unique extension to $Y_Z$.  By Proposition \ref{prop:koszul3} the simple $\cF|_{Y_U}$ is in bijection with a simple in $\Coh^\H(U)$ which, by Proposition \ref{prop:simples}, has a unique extension to a simple $\cF' \in \Coh^\H(Z)^{\heartsuit}_K$. Thus we get a map on simples 
	$$ \Coh^\H(Y)^\heartsuit_K \ni i_* \cF \mapsto i_* \cF' \in \Coh^\H(X)^\heartsuit_K.$$ 
	This map is clearly invertible and provides the claimed bijection between simples. 
\end{proof}

\subsection{On the choice of $\eta$}\label{sec:eta}

The data we started with at the beginning of this section includes the central subgroup $\eta: \C^\times \to \H$. We now explain that the Koszul t-structure on $\Coh^\H(Y)$ does not depend on this choice. 

Consider the category $\Coh^{\H \times \C^\times}(Y)$ where the extra $\C^\times$ acts trivially on $X$ and with weight one on $W$. The subgroup $\C^\times \subset \H \times \C^\times$ given by $t \mapsto (\eta(t^{-1}), t)$ acts trivially on $X$ and $W$ and hence also on $Y$. This induces a decomposition 
$$\Coh^{\H \times \C^\times}(Y) \cong \bigoplus_{n \in \Z} \Coh^{\H \times \C^\times}_n(Y).$$
Notice that $\la 1 \ra$, the equivariant shift with respect to the extra $\C^\times$, induces an isomorphism between $\Coh^{\H}_n(Y)$ and $\Coh^{\H}_{n+1}(Y)$. In particular, all the summands on the right hand side are equivalent. 

Now consider the composition 
$$Y/\H \xrightarrow{f} Y/(\H \times \C^\times) \xrightarrow{g} Y/\H,$$
where $f$ and $g$ are induced by the maps $\H \to \H \times \C^\times \mapsto \H$ given by $h \mapsto (h,1)$ and $(h,t) \mapsto h \eta(t)$. Then $g^*: \Coh^\H(Y) \to \Coh^{\H \times \C^\times}(Y)$ identifies $\Coh^\H(Y)$ with $\Coh^{\H \times \C^\times}_0(Y)$, with $f^*$ serving as the inverse. The following is an immediate consequence of the fact that under the map $(h,t) \mapsto h \eta(t)$ we have $(1,t) \mapsto \eta(t)$. 

\begin{Lemma}\label{lem:eta}
The equivalence $g^*: \Coh^\H(Y) \to \Coh^{\H \times \C^\times}_0(Y)$ is t-exact for the Koszul t-structures constructed with respect to $\eta$ on the left and the extra $\C^\times$ on the right. 
\end{Lemma}

Thus we find that $\Coh^\H(Y)_K^{\heartsuit} \cong \Coh_0^{\H \times \C^\times}(Y)_K^{\heartsuit}$. In particular, up to isomorphism, the category $\Coh^\H(Y)_K^{\heartsuit}$ does not depend on the choice of $\eta$. 

\section{Koszul-perverse sheaves}\label{sec:R}

We now use the results of Section \ref{sec:Koszul} to define the category of Koszul-perverse coherent sheaves on the space of triples $\hR_{G,\N}$. The main results are summarized as follows, where the Koszul-perverse t-structure on $\Coh^\hGO(\Gr_G)$ is the image of the perverse t-structure under regrading (as in Section \ref{sec:koszulity}).

\begin{Theorem}\label{thm:koszul-perverse}
There exists a unique t-structure on $\Coh^{\hGO}(\hR_{G,N})$ such that 
$$\sigma^*_1 i_{1*}: \Coh^\hGO(\hR_{G,N}) \to \Coh^\hGO(\Gr_G)$$
is t-exact with respect to the Koszul-perverse t-structure on $\Coh^\hGO(\Gr_G)$. Symmetrically, it is the unique t-structure such that $\sigma^*_2 i_{2*}$ is t-exact. This t-structure is bounded and finite length. Up to loop and Koszul shifts, the simple objects $\cP_{\l^\vee,\mu}$ in its heart are indexed by dominant pairs $(\l^\vee,\mu)$ of $G$, with $\cP_{\l^\vee,\mu}$ determined by the fact that its restriction to $\hR_{\l^\vee}$ is the pushforward from $\hR_{\l^\vee}^{\cl}$ of the (shifted) bundle corresponding to $\mu$. 	
\end{Theorem}


\subsection{The space of triples}\label{sec:Rdefs}
We begin by recalling the construction of \cite{BFN}. We fix a connected reductive complex algebraic group $G$ and a finite-dimensional representation $N$. There are two natural bundles over $\Gr_G$ with fibers $\N_\cO$, namely the twisted product 
$$\T_{G,N} := G_\K \times_{G_\O} \N_\O$$ 
and the untwisted product $\Gr_G \times \N_\O$. We also use the notation $\Gr_G \ttimes \N_\O$ for $\cT_{G,\N}$ and similar constructions. We let $j_1: \Gr_G \times \N_\O \to \Gr_G \times \N_\cK$ denote the natural embedding and $j_2$ the map
\begin{align*}
	j_2: \T_{G,N} \rightarrow \Gr_G \times \N_\K, \quad [g, s] \mapsto ([g], gs).
\end{align*}
These maps realize $\T_{G,N}$ and $\Gr_G \times \N_\O$ as closed ind-subschemes of $\Gr_G \times \N_\K$. 

\begin{Definition}\label{def:Rspace}
	The space of triples $\hR_{G,N}$ is the fiber product 
	\begin{equation}\label{eq:R}
		\begin{tikzpicture}[baseline=(current  bounding  box.center),thick,>=\arrtip]
			\node (a) at (0,0) {$\hR_{G,N}$};
			\node (b) at (3,0) {$\Gr_G \times N_\O$};
			\node (c) at (0,-1.5) {$\T_{G,N}$};
			\node (d) at (3,-1.5) {$\Gr_G \times N_\K$};
			\draw[->] (a) to node[above] {$i_1$} (b);
			\draw[->] (b) to node[right] {$j_1$} (d);
			\draw[->] (a) to node[left] {$i_2$}(c);
			\draw[->] (c) to node[above] {$j_2$} (d);
		\end{tikzpicture}
	\end{equation}
	in the category of (dg) ind-schemes.
\end{Definition}

As $G$ and $\N$ are fixed, we generally write $\hR, \T, \Gr$ instead of $\hR_{G,N}, \T_{G,N}, \Gr_G$. Recall the group $\hGO := (G_\cO \rtimes \C^\times) \times \C^\times$. It has a natural action on the spaces appearing in (\ref{eq:R}): the inner $\C^\times$ acts by multiplication on the loop variable $t$ in $\cO := \C[[t]]$, while the outer $\C^\times$ acts by scalar multiplication on $\N$  (and hence on $\N_\O$, $\N_\K$). The maps $j_1$ and $j_2$ are $\hGO$-equivariant, hence we obtain an induced $\hGO$-action on $\hR$ for which $i_1$ and $i_2$ are $\hGO$-equivariant. 

We will use the following notation for various approximations of $\hR$. We fix a presentation $\Gr \cong \colim \Gr_\al$ of $\Gr$ as an ind-scheme. We can assume the $\Gr_\al$ are $G_\O$-invariant, otherwise replace each $\Gr_\al$ with the closure of the image of $G_\O \times \Gr_\al \to \Gr$.  If $G$ is semisimple we can take the $\Gr_\al$ to be the closures $\Gr_{\le \l^\vee}$ of the $G_\O$-orbits $\Gr_{\l^\vee}$ (where $\l^\vee \in P^\vee$), but in general these will only present the reduced locus of $\Gr$. Nonetheless, the $\Gr_{\le \l^\vee}$ still play an important role in describing simple objects in $\KPcohGN$. 

Writing $\T_\al := \T \times_\Gr \Gr_\al$ and $\hR_\al := \T_\al \times_{\N_\K} \N_\O$, we then have $\T \cong \colim \T_\al$ and $\hR \cong \colim \hR_\al$. These colimits are taken in the category of ind-schemes, which admits left exact filtered colimits along ind-closed immersions. Each $\T_\al$ is a classical scheme since $\T \to \Gr$ is flat. It follows that $\T_\al \to \N_\K \cong \colim_{k>0} t^{-k} \N_\O$ factors through~$t^{-k'} \N_\O$ for some $k'$, hence $\hR_\al \cong \colim_{k \geq k'} \hR_\al^k$, where $\hR_\al^k$ is defined by the square below.
\begin{equation}\label{eq:R^k}
	\begin{tikzpicture}[baseline=(current  bounding  box.center),thick,>=\arrtip]
		\node (a) at (0,0) {$\hR_{\al}^k$};
		\node (b) at (3,0) {$\Gr_{\al} \times N_\O$};
		\node (c) at (0,-1.5) {$\T_{\al}$};
		\node (d) at (3,-1.5) {$\Gr_{\al} \times t^{-k} N_\O$};
		\draw[->] (a) to node[above] {$i_1$} (b);
		\draw[->] (b) to node[right] {$j_1$} (d);
		\draw[->] (a) to node[left] {$i_2$}(c);
		\draw[->] (c) to node[above] {$j_2$} (d);
	\end{tikzpicture}
\end{equation}
Note that the classical locus of $\hR_\al^k$ is independent of $k$, hence $\hR^{\cl}_\al$ is a scheme even though $\hR_\al$ itself is only an ind-scheme. 

We can write each $\hR_\al^k$ as an inverse limit of finite type schemes. To do this we consider the quotient bundle $\T^\ell_{\al} \to \Gr_{\al}$ of $\T_\al$ whose fiber over $[g] \in \Gr_\al$ is $N_\O/(g^{-1}t^\ell N_\O)$. To see that this is a bundle consider the $G_\O$-bundle $G_\K \to \Gr$ and define $G_\K^{\al} \to \Gr_\al$ by base change. Fixing $g_0 \in G_\K^\al$ choose local sections $\{s_i\}$ of $N_\O$ so that their restriction to $N_\O/g_0^{-1}t^\ell N_\O$ form a basis. Now consider the images of these sections under the map $G_\K^\al \times N_\O \to G_\K^\al \times t^{-k}N_\O/t^\ell N_\O$ given by $(g,s) \mapsto (g,[gs])$. Over $g_0$ they still form a basis of $g_0N_\O/t^\ell N_\O \subset t^{-k}N_\O/t^\ell N_\O$. Thus they must remain linearly independent in a neighborhood of $g_0$. It follows that the sections $\{s_i\}$ must also be linearly independent when restricted to $N_\O/g^{-1} t^\ell N_\O$ for $g$ in a neighborhood of $g_0$.  This shows that the rank of $N_\O/g^{-1} t^\ell N_\O$ is locally constant and that the sections above provide a local trivialization. 
 
We now define $\hR_{\al}^{k,\ell}$ by the Cartesian diagram 
\begin{equation}\label{eq:R^kl}
	\begin{tikzpicture}[baseline=(current  bounding  box.center),thick,>=\arrtip]
		\node (a) at (0,0) {$\hR_{\al}^{k,\ell}$};
		\node (b) at (3.2,0) {$\Gr_{\al} \times N_\O/t^\ell N_\O$};
		\node (c) at (0,-1.5) {$\T^\ell_{\al}$};
		\node (d) at (3.2,-1.5) {$\Gr_{\al} \times (t^{-k} N_\O/t^\ell N_\O)$.};
		\draw[->] (a) to node[above] {$i_1$} (b);
		\draw[->] (b) to node[right] {$j_1$} (d);
		\draw[->] (a) to node[left] {$i_2$}(c);
		\draw[->] (c) to node[above] {$j_2$} (d);
	\end{tikzpicture}
\end{equation}
If we take the base change of (\ref{eq:R^kl}) with respect to the map 
\begin{equation}\label{eq:local2}
	\Gr_\al \times (t^{-k}N_\O/t^{\ell+1}N_\O) \to \Gr_\al \times (t^{-k}N_\O/t^{\ell}N_\O)
\end{equation}
we recover (\ref{eq:R^kl}), but with $\ell$ replaced by $\ell+1$. Taking the inverse limit recovers (\ref{eq:R^k}). In particular, since (\ref{eq:local2}) is faithfully flat and affine, we find that $\hR^k_\al$ is the inverse limit of the $\hR^{k,\ell}_\al$ along faithfully flat, affine maps.

The ind-scheme $\hR$ is not of ind-finite type when $\N$ is nonzero, but instead satisfies the following weaker finiteness conditions. Recall from Section \ref{sec:convnot} that tamely presented schemes are a class of well-behaved coherent schemes, a basic example being the affine space $\N_\O$. An ind-scheme is ind-tamely presented if it is a filtered limit of schemes which are tamely presented and truncated (i.e. their structure sheaves are cohomologically bounded) under almost finitely presented closed immersions. 

\begin{Proposition}\label{prop:R}
The ind-scheme $\hR$ is ind-tamely presented, and all maps in  (\ref{eq:R}) are almost ind-finitely presented ind-closed immersions.  
\end{Proposition}
\begin{proof}
Consider approximations $\hR_\al^k$ as above. Since $N_\O$ is tamely presented it follows that so is $\T_\al$. Then $\hR_\al^k$ is as well since $\hR_\al^k \to \T_\al$ is a base change of the finitely presented map $\N_\O \to t^{- k} \N_\O$. Moreover $\hR_\al^k$ is truncated since $\T_\al$ is classical and $\N_\O \to t^{- k} N_\O$ is of finite Tor-dimension. The maps $\hR_\al^k \to \hR_\al^{k + 1}$ are almost finitely presented closed immersions since they are base changes of the diagonal of $t^{- k} N_\O \to t^{- (k + 1)} N_\O$, which has these properties. Thus $\hR_\al \cong \colim_{k \geq k'} \hR_\al^k$ is ind-tamely presented by definition. But the maps $\hR_\al \to \hR_\be$ are almost ind-finitely presented ind-closed immersions since they are base changes of $\Gr_\al \to \Gr_\be$, hence $\hR \cong \colim \hR_\al$ is ind-tamely presented \cite[\tmproptamefiltcolims]{CWtm}. 
	
Since $j: N_\O \to N_\K$ is an almost ind-finitely presented ind-closed immersion so are its base changes $j_2$ and $i_2$. We can rewrite $j_1$ as
the composition 
\begin{equation*}
	\begin{tikzpicture}[baseline=(current  bounding  box.center),thick,>=\arrtip]
		\newcommand*{\hp}{3}; \newcommand*{\hpp}{3}; \newcommand*{\hppp}{11};
		\newcommand*{\vb}{-.8}; \newcommand*{\vc}{-4};
		\node (a) at (0,0) {$\Gr \ttimes \N_\O$};
		\node (a') at (\hp,0) {$\Gr \ttimes \N_\K$};
		\node (a'') at (\hp+\hpp,0) {$\Gr \times \N_\K$};
		\node (b) at (0,\vb) {$[g,s]$};
		\node (b') at (\hp,\vb) {$[g,s]$};
		\node (b'') at (\hp+\hpp,\vb) {$([g],gs).$};
		
		\draw[->] (a) to node[above] {$\id \ttimes j $} (a');
		\draw[->] (a') to node[above] {$\sim $} (a'');
		\draw[|->] (b) to node[above] {$ $} (b');
		\draw[|->] (b') to node[above] {$ $} (b'');
	\end{tikzpicture}
\end{equation*}
The first map is an almost ind-finitely presented closed immersion since it is locally a base change of $j$. But then $j_1$ and its base change $i_1$ are as well, since the second map is an isomorphism (with inverse $([g],s') \mapsto [g, g^{-1}s']$). 
\end{proof}

Recall that the ind-scheme $\hR$ has a moduli description as the space of sections of the associated $\N$-bundle of a partially trivialized $G$-bundle on the formal disk with doubled origin. Omitting the trivialization this becomes a moduli description of the quotient stack $\hR/G_\cO$. We may formalize this as the following result, which we think of as affirming that Definition~\ref{def:Rspace} defines the correct derived structure on $\hR$ (as (\ref{eq:Rsymmdiagram}) is ultimately more fundamental than~(\ref{eq:R})). By convention we use the same notation for an equivariant map and the induced map of quotient stacks when no confusion arises. We also note that $\T \cong \hGK \times_\hGO \N_\O$ has a canonical action of $\hGK$, and that the quotient $\T/\hGK$ is canonically isomorphic to $\N_\O/\hGO$. 

\begin{Proposition}\label{prop:symmetricRdef}
There is a Cartesian diagram 
	\begin{equation}\label{eq:Rsymmdiagram}
		\begin{tikzpicture}
			[baseline=(current  bounding  box.center),thick,>=\arrtip]
			\node (a) at (0,0) {$\hR/\hGO$};
			\node (b) at (3,0) {$\N_\cO/\hGO$};
			\node (c) at (0,-1.5) {$\N_\cO/\hGO$};
			\node (d) at (3,-1.5) {$\N_\cK/\hGK,$};
			\draw[->] (a) to node[above] {$\pi_1$} (b);
			\draw[->] (b) to node[right] {$ $} (d);
			\draw[->] (a) to node[left] {$\pi_2$}(c);
			\draw[->] (c) to node[above] {$ $} (d);
		\end{tikzpicture}
	\end{equation}
where $\pi_1$ is the composition of $i_1$ and the projection $\wt{\pi}_1: \Gr \times \N_\O \to \N_O$, and $\pi_2$ is the composition of $i_2$ with the quotient map $\wt{\pi}_2: \T/\hGO \to \T/\hGK \cong \N_\O/\hGO$. 
\end{Proposition}

\begin{Lemma}\label{lem:Cartesianstacks}
Let $\cG$ be a group ind-scheme. Given a Cartesian square of $\cG$-ind-schemes and $\cG$-equivariant maps, the induced diagram of quotient stacks is also Cartesian. 
\end{Lemma}
\begin{proof}
Let $X \to Y$ and $Y' \to Y$ be $\cG$-equivariant maps of $\cG$-ind-schemes, and let $\wt{X}' \cong X/\cG \times_{Y/\cG} Y'/\cG$. It follows from \cite[Thm. 6.1.3.9(4)]{LurHTT} that $X \cong X/\cG \times_{Y/\cG} Y$ and $Y' \cong Y'/\cG \times_{Y/\cG} Y$. It then follows that $X' \cong \wt{X}' \times_{X/\cG} X$, and that since $X \to X/\cG$ is an effective epimorphism so is its base change $X' \to \wt{X}'$ \cite[Prop. 6.2.3.15]{LurHTT}. Thus $\wt{X}'$ is the geometric realization of its Cech nerve, which is equivalently the base change of the Cech nerve of $X \to X/\cG$. By definition this is the action groupoid of the $\cG$-action on $X'$, hence $\wt{X}' \cong X'/\cG$. 
\end{proof}

\begin{proof}
Note that by construction $\hR \cong \T \times_{\N_\K} \N_\O$, and that tautologically $(\hGO \backslash \hGK) \times \T \cong ((\hGO \backslash \hGK) \times \N_\K) \times_{\N_\K} \N_\O$. Applying Lemma \ref{lem:Cartesianstacks} to the associated Cartesian squares with $\cG = \hGO$ and $\cG = \hGK$, respectively, we obtain Cartesian squares
\begin{equation*}
	\begin{tikzpicture}[baseline=(current  bounding  box.center),thick,>=\arrtip]
		\newcommand*{\ha}{3}; \newcommand*{\hb}{3}; \newcommand*{\va}{1.5};
		\node[matrix] at (0,0) {
			\node (aa) at (0,0) {$\hR/\hGO$};
			\node (ab) at (\ha,0) {$\N_\O/\hGO$};
			\node (ba) at (0,-\va) {$\T/\hGO$};
			\node (bb) at (\ha,-\va) {$\N_\K/\hGO$};
			\draw[->] (aa) to node[above] {$\pi_1 $} (ab); 
			\draw[->] (aa) to node[left] {$i_2 $} (ba);
			\draw[->] (ba) to node[above] {$ $} (bb);
			\draw[->] (ab) to node[above] {$ $} (bb); \\
		};
		\node[matrix] at (5.9,-.1) {
			\node (aa) at (0,0) {$\T/\hGO$};
			\node (ab) at (\hb,0) {$\N_\K/\hGO$};
			\node (ba) at (0,-\va) {$\N_\O/\hGO$};
			\node (bb) at (\hb,-\va) {$\N_\K/\hGK.$};
			\draw[->] (aa) to node[above] {$ $} (ab); 
			\draw[->] (aa) to node[left] {$ $} (ba);
			\draw[->] (ba) to node[above] {$ $} (bb);
			\draw[->] (ab) to node[above] {$ $} (bb); \\
		};
	\end{tikzpicture}
\end{equation*}
The desired Cartesian square now arises by stacking these. 
\end{proof}

We have the following counterpart of Proposition \ref{prop:R} for the diagram (\ref{eq:Rsymmdiagram}). Recall that a geometric stack is admissible if it has an affine morphism to a locally Noetherian geometric stack, and ind-geometric stacks is admissible it it has a presentation by these. 

\begin{Proposition}\label{prop:Rquotient}
The quotient $\hR/\hGO$ is an admissible, ind-tamely presented ind-geometric stack, and the maps $\pi_1$ and $\pi_2$ in (\ref{eq:Rsymmdiagram}) are ind-proper and almost ind-finitely presented. 
\end{Proposition}

\begin{proof}
Fix subschemes $\hR_\al^k$ of $\hR$ as before, noting that these are $\hGO$-invariant since the $\Gr_\al$ are $\hGO$-invariant in $\Gr$. By \cite[\tmproptamequotients]{CWtm} each $\hR_\al^k/\hGO$ is a tamely presented geometric stack since $\hR_\al^k$ is a tamely presented scheme and $\hGO$ is a classical affine group scheme. Each $\hR_\al^k/\hGO$ is admissible since we have the affine map $\hR_\al^k/\hGO \to \Gr_\al/\hGO$. The first claim now follows as in the proof of Proposition \ref{prop:R}. 	

Next we show $\pi_2$ is ind-proper and almost ind-finitely presented, the case of $\pi_1$ following by symmetry.  Writing
$$\hR/\hGO \xrightarrow{i_2} \T/\hGO \xrightarrow{\wt{\pi}_2} N_\O/\hGO$$
for the defining factorization of $\pi_2$, it suffices to show the claim for each factor. The quotient map $\wt{\pi}_2$ is a base change of $\Gr/\hGO \to \pt/\hGO$, so it suffices to show $\Gr_\al/\hGO \to \pt/\hGO$ is proper and finitely presented for any $\al$. Since $\Gr_\al$ is finite type the action of $\hGO$ factors through a finite-type quotient group $\hGO^\al$. The map $\Gr_\al/\hGO \to \pt/\hGO$ is then obtained by base change from $\Gr_\al/\hGO^\al \to \pt/\hGO^\al$. Since $\hGO^\al$ is an algebraic group these quotients coincide with the associated fppf quotients, and the map between them is a relative algebraic space (e.g. as can be shown using \cite[Cor. 8.1.1]{LMB00}). That this map is proper and finitely presented now follows since its pullback along the flat cover $\pt \to \pt/\hGO^\al$ is so. 

Similarly, $i_2$ is obtained by base change from $N_\O/\hGO \to N_\K/\hGO$, hence is ind-proper and almost ind-finitely presented if each $N_\O/\hGO \to t^{- k}N_\O/\hGO$ is proper and finitely presented. But this map is affine since its pullback along the flat cover $t^{- k}N_\O \to t^{- k}N_\O/\hGO$ is so, and is then proper and finitely presented for the same reason. 
\end{proof}

\subsection{Koszul-perversity}\label{sec:Koszul-perverse}
Before defining the Koszul-perverse t-structure on $\Coh^\hGO(\hR)$, we recall the perverse t-structure on $\Coh^\hGO(\Gr)$.  The category $\Coh^\hGO(\Gr)$ can be written as the colimit of the categories $\Coh^\hGO(S)$ over all closed $\hGO$-invariant subschemes $S \subset \Gr$. If such a subscheme contains a $G_\cO$-orbit $\Gr^{\l^\vee} \subset \Gr$, we use the same symbol $i_{\l^\vee}$ for the locally closed embeddings of $\Gr^{\l^\vee}$ into both $S$ and $\Gr$. 

\begin{Definition} \cite{AB10}
	Let $S \subset \Gr$ be a closed $\hGO$-invariant subscheme of $\Gr$. The perverse t-structure $(\Coh^{\hGO}(S)^{\leq 0}_p, \Coh^{\hGO}(S)^{\geq 0}_p)$ on $\Coh^{\hGO}(S)$ is defined as follows. Given $\cF \in \Coh^{\hGO}(S)$ we have
	\begin{enumerate}
		\item $\cF \in \Coh^{\hGO}(S)^{\leq 0}_p$ if and only if $i_{\l^\vee}^*(\cF) \in \QCoh^{\hGO}(\Gr^{\l^\vee})^{\leq - \frac12 \dim \Gr^{\l^\vee}}$ for all $\Gr^{\l^\vee} \subset S$, 
		\item $\cF \in \Coh^{\hGO}(S)^{\geq 0}_p$ if and only if $i_{\l^\vee}^!(\cF) \in \QCoh^{\hGO}(\Gr^{\l^\vee})^{\geq - \frac12 \dim \Gr^{\l^\vee}}$ for all $\Gr^{\l^\vee} \subset S$.
	\end{enumerate}
The perverse t-structure on $\Coh^\hGO(\Gr)$ is the unique t-structure such that for any closed $\hGO$-invariant subscheme $S$, the pushforward $i_{S*}: \Coh^\hGO(S) \to \Coh^\hGO(\Gr)$ is t-exact for the perverse t-structure on $\Coh^\hGO(S)$. 
\end{Definition}

Here we implicitly retain the convention of \cite{BFM,CW1} that the block of $\Coh^\hGO(\Gr)$ consisting of sheaves with odd-dimensional support is cohomologically graded by $\Z + \frac12$, so that expressions like $\frac12 \dim \Gr^{\l^\vee}$ make sense. 

The perverse t-structure on each $\Coh^\hGO(S)$ is local. Regrading it as in Construction \ref{construct} gives us a Koszul t-structure, which we denote by $(\Coh^\hGO(S)^{\leq 0}_\Kzl, \Coh^\hGO(S)^{\geq 0}_\Kzl)$ and refer to as the {\it Koszul-perverse t-structure}. We use the same terminology and notation for the induced t-structure on $\Coh^\hGO(\Gr)$. 

We will reduce Theorem \ref{thm:koszul-perverse} to Theorem \ref{thm:koszul} using the following standard Lemma, which is already implicit in the definition of the perverse t-structure on $\Coh^{\hGO}(\Gr)$.

\begin{Lemma}\label{lem:colimittstructure}
	Let $\catC \cong \colim \catC_\al$ be a filtered colimit in $\Catinfty$ such that each $\catC_\al$ is stable and each $F_{\al\be}: \catC_\al \to \catC_\be$ is exact. Suppose each $\catC_\al$ is equipped with a t-structure such that $F_{\al\be}$ is t-exact for all $\be \geq \al$. 
\begin{enumerate}
	\item There exists a unique t-structure on $\catC$ such that every $F_\al: \catC_\al \to \catC$ is t-exact.
	\item If every $F_{\al\be}: \catC_\al^{\heartsuit} \to \catC_\be^{\heartsuit}$ sends simples to simples then so does every $F_\al: \catC_\al^{\heartsuit} \to \catC^{\heartsuit}$.  	\item If every $F_{\al\beta}$ is conservative then every simple in $\catC^{\heartsuit}$ is the image of a simple in some~$\catC_\al^\heartsuit$.
\end{enumerate}
\end{Lemma}
\begin{proof}
If $(\catC_\al^{\leq 0}, \catC_\al^{\geq 0})$ denotes the t-structures on $\catC_\al$ then the t-structure $(\catC^{\leq 0}, \catC^{\geq 0})$ on $\catC$ is given by the essential images of the $\catC^{\leq 0}_\al$ and $\catC^{\geq 0}_\al$. To see this defines a t-structure, suppose $X \in \catC^{\leq 0}$ and $Y \in \catC^{> 0}$. Since the indexing diagram is filtered, we have $X \cong F_\al(X_\al)$ and $Y \cong F_\al(Y_\al)$ for some $\al$ and some $X_\al \in \catC^{\leq 0}$, $Y_\al \in \catC^{> 0}$. We then have $\Map_\catC(X,Y) \cong \colim_{\be \geq \al} \Map_{\catC_\be}(F_{\al\be}(X_\al), F_{\al\be}(Y_\al)$ \cite{Roz}, which is contractible since each term on the right is contracitble by t-exactness of the $F_{\al\be}$. Exactness of the $F_\al$ implies $\catC^{\leq 0}$ and $\catC^{\geq 0}$ are preserved by $[1]$ and $[-1]$, respectively. Finally, we can write any $X \in \catC$ as $X \cong F_\al(X_\al)$ for $X_\al \in \catC_\al$, hence $X$ fits into a triangle $F_\al(\tau^{\leq 0}(X_\al)) \to X \to F_\al(\tau^{> 0}(X_\al))$ with outer terms in $\catC^{\leq 0}$ and $\catC^{> 0}$. Now let $(\catC^{\leq 0}_{t}, \catC^{\geq 0}_{t})$ be a possibly different t-structure such that each $F_\al$ is t-exact. By definition we have $\catC^{\leq 0} \subset \catC^{\leq 0}_{t}$ and $\catC^{> 0} \subset \catC^{> 0}_{t}$. But then 
$ \catC^{\leq 0}_{t} = {}^\perp(\catC^{> 0}_{t}) \subset {}^\perp(\catC^{> 0}) = \catC^{\leq 0},$
hence $\catC^{\leq 0} = \catC^{\leq 0}_{t}$.

For (2), consider a simple $X_\al \in \catC_\al^{\heartsuit}$. Each $F_{\al\be}(X_\al)$ is nonzero since it is simple, hence $F_\al(X_\al)$ is nonzero. Consider a nonzero sub-object $i: Y \hookrightarrow F_\al(X_\al)$. Arguing as above we can assume, after possibly replacing $\al$ with a larger $\al$, that $i = F_\al(i_\al)$ for some $i_\al: Y_\al \to X_\al$ in $\catC_{\al}^\heartsuit$. Since $X_\al$ is simple $i_\al$ is an isomorphism, hence so is $i$. For (3), note that since the $F_{\al\beta}$ are conservative so is $F_\al$. It follows that if $X = F_\al(X_\al)$ is simple then so is $X_\al$. 
\end{proof}

\begin{Lemma}\label{lem:Cohcoliminvlim}
Let $\cG$ be a classical affine group scheme and $X \cong \lim X_\al$ a filtered limit of Noetherian $\cG$-schemes along flat $\cG$-equivariant affine morphisms. Then the natural functor $\colim \Coh^\cG(X_\al) \to \Coh^\cG(X)$ is an equivalence. 
\end{Lemma}
\begin{proof}
Fixing some $\al$, we have $X \cong \lim_{\be \geq \al} X_\be$ since the limit is filtered. Let $A, A_\be \in \Alg(\QCoh^\cG(X_\al))$ denote the images of $\cO_X$ and $\cO_{X_\be}$ under the projections $X/\cG \to X_\al/\cG$ and $X_\be/\cG \to X_\al/\cG$, respectively. These projections are affine since their base changes to $X_\al$ are \cite[Lem. 9.3.1.1]{LurSAG}, hence by Barr-Beck $\QCoh^\cG(X)$ and $\QCoh^\cG(X_\be)$ are the categories of $A$- and $A_\be$-modules in $\QCoh^\cG(X_\al)$. The natural map $\colim A_\be \to A$ is an isomorphism since by hypothesis it is after forgetting equivariance, which is conservative. The functor $\Alg(\QCoh^\cG(X_\al)) \to \PrL$ taking an algebra to its category of modules in $\QCoh^\cG(X_\al)$ preserves filtered colimits \cite[Cor. 4.2.3.5, Cor. 4.8.5.13]{LurHA}, hence $\QCoh^\cG(X) \cong \colim \QCoh^\cG(X_\be)$. Since each $X_\be$ is Noetherian and $\cG$ is classical and affine, $\QCoh^\cG(X_\be)^{[a,b]}$ is compactly generated by $\Coh^\cG(X_\be)^{[a,b]}$ for all $a \leq b$  \cite[Prop. 9.5.2.3, Prop. C.6.5.4]{LurSAG}. Since the projections are flat, it then follows that $\QCoh^\cG(X)^{[a,b]}$ is compactly generated by $\Coh^\cG(X)^{[a,b]}$ and that $\Coh^\cG(X)^{[a,b]} \cong \colim \Coh^\cG(X_\be)^{[a,b]}$ in $\Catinfty$ \cite[Lem. 7.5.3.11]{LurHA}. The claim then follows since by boundedness $\Coh^\cG(X) \cong \colim_{a < b} \Coh^\cG(X)^{[a,b]}$, similarly for each $\Coh^\cG(X_\be) $. 	
\end{proof}

\begin{proof}[Proof of Theorem \ref{thm:koszul-perverse}]
We fix a presentation $\hR \cong \colim \hR_\al$ as in Section \ref{sec:Rdefs}. For any $\al$ we fix a $k'$ such that $\hR_\al \cong \colim_{k \geq k'} \hR_\al^k$, and have $\hR_{\al}^k \cong \lim \hR_\al^{k,\ell}$. By Theorem \ref{thm:koszul} there exists a t-structure on $\Coh^\hGO(\hR_\al^{k,\ell})$ 
defined by taking  $(\Coh^\hGO(\hR_\al^{k,\ell})^{\leq 0}_{\Kzl}, \Coh^\hGO(\hR_\al^{k,\ell})^{\geq 0}_{\Kzl})$ to be the preimage of $(\Coh^\hGO(\Gr_\al)^{\leq 0}_{\Kzl},\Coh^\hGO(\Gr_\al)^{\geq 0}_{\Kzl})$ under $\sigma_2^* i_{2*}$. A simple base change argument implies that pullback along $\hR^{k,\ell+1}_{\al} \to \hR^{k,\ell}_{\al}$ is t-exact, hence by Lemmas \ref{lem:colimittstructure} and \ref{lem:Cohcoliminvlim} we obtain an induced t-structure on $\Coh^\hGO(\hR_\al^{k})$. 

Another base change argument shows that $(\Coh^\hGO(\hR_\al^{k})^{\leq 0}_{\Kzl},\Coh^\hGO(\hR_\al^{k})^{\geq 0}_{\Kzl})$ is again the preimage of $(\Coh^\hGO(\Gr_\al)^{\leq 0}_{\Kzl}, \Coh^\hGO(\Gr_\al)^{\geq 0}_{\Kzl})$ under $\sigma_2^* i_{2*}$. From this description it follows that pushforward along the natural map $\hR_\al^{k} \to \hR_\al^{k+1}$ is t-exact, as this commutes with $i_{2*}$. Applying Lemma \ref{lem:colimittstructure} again we obtain a t-structure on $\hR_\al$, with $(\Coh^\hGO(\hR_\al)^{\leq 0}_{\Kzl}, \Coh^\hGO(\hR_\al)^{\geq 0}_{\Kzl})$ still characterized as the preimages under $\sigma_2^* i_{2*}$. It follows as in the finite-rank case of Proposition \ref{prop:compatible} that pushforward along the maps $\hR_\al \to \hR_\be$ is t-exact. Applying Lemma \ref{lem:colimittstructure} again we obtain the indicated t-structure on $\Coh^\hGO(\hR)$, which is still the preimage under $\sigma_2^* i_{2*}$ of $\Coh^\hGO(\Gr)^{\leq 0}_{\Kzl},\Coh^\hGO(\Gr)^{\geq 0}_{\Kzl})$. Since this latter t-structure is bounded and finite-length so is the former since $\sigma_{2*} i_{2*}$ is conservative on coherent sheaves. 

It follows from \cite[Prop. 4.11]{AB10} (or Proposition \ref{prop:simples}) that every simple object in $\Coh^\hGO(\Gr)_K^\heartsuit$ restricts to a (shifted) $\hGO$-equivariant vector bundle on the open $G_\O$-orbit $\Gr_\l$ in its support, and conversely every such (appropriately shifted) vector bundle can be extended to a unique simple object. Up to loop and Koszul shifts (i.e. up to forgetting from $\hGO$- to $G_\O$-equivariance) such simples are thus labeled by $(P^\vee \times P)/W$, or equivalently by \emph{dominant pairs}:  pairs $(\l^\vee, \mu) \in P^\vee \times P$ such that $\l^\vee$ is dominant and $\mu$ is dominant for the Levi factor of $P_{\lambda^\vee}$. 

The classification of simples in $\Coh^\hGO(\hR)_K^{\heartsuit}$ is now a consequence of Theorem \ref{thm:koszulsimples}, Proposition \ref{prop:koszul3}, and Lemma \ref{lem:colimittstructure}.  More precisely, Theorem \ref{thm:koszulsimples} states that simples in $\Coh^\hGO(\hR_{\l^\vee}^{k,\ell})_K^\heartsuit$ are the unique extensions of simples of the form $\iota_* \pi^*(\cF)$, where $\cF$ is a simple in some $\Coh^\hGO(\Gr_{\l^\vee})_K^\heartsuit$ and $\iota_* \pi^*: \Coh^\hGO(\Gr_{\l^\vee}) \to \Coh^\hGO(\hR_{\l^{\vee}}^{k,\ell})$ is as in Proposition~\ref{prop:koszul3}. It follows from this description that the pullbacks via $\hR_{\l^\vee}^{k,\ell+1} \to \hR_{\l^\vee}^{k,\ell}$ and pushforwards via $\hR_{\l^\vee}^k \to \hR_{\l^\vee}^{k+1}$ preserve simples. Moreover, by Proposition \ref{prop:compatible}, pushforwards via $\hR_{\l^\vee} \to \hR_{\mu^\vee}$ also preserve simples. This gives the description of simples $\cP_{\l^\vee,\mu}$ in $\Coh^\hGO(\hR)_K^\heartsuit$ as the unique extensions of simples $\iota_* \pi^*(\cF)$ where $\cF$ is a simple in $\Coh^\hGO(\Gr_{\l^\vee})_K^{\heartsuit}$, and $\iota_* \pi^*: \Coh^\hGO(\Gr_{\l^\vee}) \to \Coh^{\hGO}(\hR_{\l^\vee})$. 
\end{proof}

\subsection{The category $\KPcohGNeta$}\label{sec:eta2}

There are interesting examples, such as $(G,N) = (GL_n,\C^n)$ considered in Section \ref{sec:GLn}, where $G$ carries a natural cocharacter $\eta: \C^\times \to G$ which acts with weight one on $N$. One can use this cocharacter to define a Koszul-perverse t-structure on $\Coh^{G_\O \rtimes \C^\times}(\hR_{G,N})$, and we denote the resulting heart $\KPcohGNeta$. As a corollary of the discussion in Section \ref{sec:eta} we find that 
$$\KPcohGN \cong \bigoplus_{\Z} \KPcohGNeta,$$
where the Koszul shift $[1]\la -1 \ra$ on  $\KPcohGN$ permutes the summands on the right. 

\section{Convolution}\label{sec:convolution}

Next we discuss convolution in $\Coh^\hGO(\hR_{G,N})$ and its compatibility with the Koszul-perverse t-structure. The following is the main result.

\begin{Theorem}\label{thm:convsecmainthm}
If $\cF \in \Coh^{\hGO}(\hR_{G,N})$ is Koszul-perverse, then left and right convolution with $\cF$ are t-exact with respect to the Koszul-perverse t-structure.
The monoidal unit is Koszul-perverse, hence $\KPcohGN$ inherits a monoidal structure from $\Coh^{\hGO}(\hR_{G,N})$. 
\end{Theorem}

We begin with a general discussion of infinite-type convolution products in Section~\ref{sec:convgen}. In Section~\ref{sec:alternative} we study the specific geometry that gives rise to convolution in $\Coh^{\hGO}(\hR_{G,N})$. We then consider the interaction between convolution and Koszul-perversity in Section \ref{sec:KPconv}, proving the above result. 

\subsection{Generalities}\label{sec:convgen} The convolution product on $\Coh^\hGO(\hR_{G,N})$ is an instance of a standard general construction. The only nonstandard aspect here is that $\hR_{G,N}$ lacks most of the finiteness hypotheses present in more traditional situations. In this section we briefly review convolution products in this generality, both to make explicit what remaining hypotheses we have and make use of, and to fix notation for later use in Section \ref{sec:rigidity1}. 

Consider the following setup. We let $Y \to Z$ be a morphism of stacks and set $X := Y \times_Z Y$. Associated to this data are the following morphisms:
\begin{itemize}
	\item the projections $\pi_1,\pi_2: X \to Y$,
	\item the involution $\iota: X \to X$ which exchanges these projections,
	\item the diagonal map $e: Y \to X$ whose compositions with $\pi_1$ and $\pi_2$ are the identity, 
	\item the lifted diagonal map $\delta: X \times_Y X \to X \times X$, and 
	\item the multiplication map $m: X \times_Y X \to X$.
\end{itemize}
By convention in $X \times_Y X$ the left and right copies of $X$ map to $Y$ by $\pi_2$ and $\pi_1$, respectively. Explicitly, $\delta$ and $m$ are defined by the following Cartesian square and cube.  
\begin{equation}\label{eq:mdeltasquare}
	\begin{tikzpicture}
		[baseline=(current  bounding  box.center),thick,>=\arrtip]
		
		\node[matrix] at (0,0) { 
		\node (b) at (3.1,0) {$X \times X$};
		\node (a) at (0,0) {$X \times_Y X$};
		\node (d) at (3.1,-2.1) {$Y \times Y$};
		\node (c) at (0,-2.1) {$Y$};
		\draw[->] (a) to node[above] {$\delta $} (b);
		\draw[->] (b) to node[right] {$\pi_2 \times \pi_1 $} (d);
		\draw[->] (a) to node[right] {$\hspace{2cm} $}(c);
		\draw[->] (c) to node[above] {$\Delta_Y $} (d);\\
	};
\node[matrix] at (6.5,0) { 
		\newcommand*{\ha}{1.5}; \newcommand*{\hb}{1.5}; \newcommand*{\hc}{1.5};
		\newcommand*{\va}{-.9}; \newcommand*{\vb}{-.9}; \newcommand*{\vc}{-.9}; 
		\node (ab) at (\ha,0) {$X$};
		\node (ad) at (\ha+\hb+\hc,0) {$X \times_Y X$};
		\node (ba) at (0,\va) {$Y$};
		\node (bc) at (\ha+\hb,\va) {$X$};
		\node (cb) at (\ha,\va+\vb) {$Y$};
		\node (cd) at (\ha+\hb+\hc,\va+\vb) {$X$};
		\node (da) at (0,\va+\vb+\vc) {$Z$};
		\node (dc) at (\ha+\hb,\va+\vb+\vc) {$Y$};
		
		\draw[<-] (ab) to node[above,pos=.6] {$ $} (ad);
		\draw[->] (ab) to node[above left, pos=.25] {$\pi_2 $} (ba);
		\draw[->] (ab) to node[right,pos=.2] {$ $} (cb);
		\draw[->] (ad) to node[below right] {$ $} (bc);
		\draw[->] (ad) to node[right] {$m $} (cd);
		\draw[->] (ba) to node[left] {$ $} (da);
		\draw[<-] (cb) to node[above,pos=.25] {$\pi_1 $} (cd);
		\draw[->] (cb) to node[above left, pos=.25] {$ $} (da);
		\draw[->] (cd) to node[below right] {$\pi_2 $} (dc);
		\draw[<-] (da) to node[above,pos=.75] {$ $} (dc);
		
		\draw[-,line width=6pt,draw=white] (ba) to  (bc);
		\draw[<-] (ba) to node[above,pos=.75] {$\pi_1 $} (bc);
		\draw[-,line width=6pt,draw=white] (bc) to  (dc);
		\draw[->] (bc) to node[right,pos=.2] {$ $} (dc);\\
	};
	\end{tikzpicture}
\end{equation}

To define convolution of coherent sheaves we impose the following conditions on this data: 
\begin{equation}\label{eq:convhypos}
	\begin{aligned}
		(1) \quad & \text{$X$ is an admissible, ind-tamely presented ind-geometric stack, } \\
		(2) \quad & \text{$Y$ is a weakly smooth geometric stack, and} \\
		(3) \quad & \text{$\pi_1$ and $\pi_2$ are ind-proper and almost ind-finitely presented.}
	\end{aligned}
\end{equation}
That $Y$ is weakly smooth means in particular that coherent sheaves admit pullback along the diagonal $\Delta_Y$, and that this property is stable under tamely presented base change (i.e. $\Delta_Y$ has \stable coherent pullback). We refer to Section \ref{sec:convnot} for further details on our terminology, but these conditions will in particular apply to $X = \hR/\hGO$, $Y = \N_\O/\hGO$, and $Z = \N_\K/\hGK$. They of course also apply when $Y \to Z$ is a proper map of smooth varieties (or if such a map is equivariant for an algebraic group, the induced map of quotient stacks). 

By (1) the category $\Coh(X)$ of coherent sheaves on $X$ is well-defined. Since $Y$ is weakly smooth and $\pi_2 \times \pi_1$ is ind-proper and almost ind-finitely presented, it follows from (\ref{eq:mdeltasquare}) that $X \times_Y X$ is also admissible and ind-tamely presented, and that $\delta$ has \stable coherent pullback \cite[\tmpropindPbaseprops]{CWtm}. Likewise, it follows that $m$ is ind-proper and almost ind-finitely presented since $\pi_1$ and $\pi_2$ are. In particular, coherent sheaves admit pullback along $\delta$ and pushforward along $m$, and the convolution product of $\cF, \cG \in \Coh(X)$ is then defined as
\begin{equation}\label{eq:convprod2}
	\cF * \cG := m_* \delta^* (\cF \boxtimes \cG) \in \Coh(X).
\end{equation}

That convolution is associative up to isomorphism follows from the compatibility of \stable coherent pullback and ind-proper, almost ind-finitely presented pushforward under base change \cite[\tmdefcohonindgstks]{CWtm}. That is, it follows from this that $\cF \conv (\cF' \conv \cF'')$ and $(\cF \conv \cF') \conv \cF''$ are both isomorphic to $m_{123*} \delta_{123}^*(\cF \boxtimes \cF' \boxtimes \cF'')$, 
where  $m_{123}$ and $\delta_{123}$ are the natural maps from $X \times_Y X \times_Y X$ to $X$ and $X \times X \times X$, respectively. Note further that $e$ is ind-proper and almost ind-finitely presented since $\pi_1$ is and since $\pi_1 \circ e$ is the identity \cite[\igpropindPtwoofthreeprops]{CWig}. Thus coherent sheaves admit pushforward along $e$, and one checks that $e_*(\cO_Y) \in \Coh(X)$ is a unit object up to isomorphism. These associativity and unit isomorphisms can be extended to a tower of higher coherence data, as summarized by the following assertion. 

\begin{Proposition}\label{prop:monoidalconv}
Under the hypotheses (\ref{eq:convhypos}) the convolution product on $\Coh(X)$ extends to a monoidal structure whose unit object is $e_*(\cO_Y)$. 
\end{Proposition}
\begin{proof}
Taking the Cech nerve of $Y \to Z$ yields a groupoid object in $\indGStkk$. Recall from \cite[Cor. 9.4.4.5]{GR17} that this gives rise to an algebra in $\Corr(\indGStkk)$ whose underlying object is $X$, and whose multiplication and identity maps are the correspondences 
\begin{equation}\label{eq:convcorrs}
	\begin{tikzpicture}[baseline=(current bounding box.center),thick,>=\arrtip]
		\newcommand*{\ha}{2.7}; \newcommand*{\hb}{2.2}; \newcommand*{\hc}{1.5}; \newcommand*{\gap}{3};
		
		\node (a) at (0,0) {$X \times X$};
		\node (a') at (\ha,0) {$X \times_Y X$};
		\node (a'') at (\ha+\hb,0) {$X$};
		
		\node (b) at (\ha+\hb+\gap,0) {$\pt$};
		\node (b') at (\ha+\hb+\gap+\hc,0) {$Y$};
		\node (b'') at (\ha+\hb+\gap+\hc+\hc,0) {$X$};
		
		\draw[<-] (a) to node[pos=.6,above] {$\delta$} (a');
		\draw[->] (a') to node[above] {$m$} (a'');
		
		\draw[<-] (b) to node[above] {} (b');
		\draw[->] (b') to node[pos=.45,above] {$e$} (b'');
		
		\node at (\ha+\hb+\gap/2,0) {$\text{and}$};
	\end{tikzpicture}
\end{equation}

Recall next that the assignment $X \mapsto \Coh(X)$ extends to a lax symmetric monoidal functor $\Corr(\indGStkkadtm)_{prop,coh} \to \Catinfty$ \cite[\tmsecindextcohcase]{CWtm}. Here $\Corr(\indGStkkadtm)_{prop,coh}$ denotes the 1-full subcategory of $\Corr(\indGStkk)$ which only includes correspondences $W \xleftarrow{h} Y \xrightarrow{f} Z$ such that $W$ and $Z$ are admissible and ind-tamely presented, $h$ has \stable coherent pullback, and $f$ is ind-proper and almost ind-finitely presented. To construct the desired monoidal structure on $\Coh(X)$, it thus suffices to show that the algebra defined by the Cech nerve of $Y \to Z$ is actually an algebra in $\Corr(\indGStkkadtm)_{prop,coh}.$ The structure maps of this algebra are generated under compositions and products by the correspondences (\ref{eq:convcorrs}). The claim then follows since, as discussed earlier, the hypotheses (\ref{eq:convhypos}) imply that $\delta$ has \stable coherent pullback and that $m$ and $e$ are ind-proper and almost ind-finitely presented. 
\end{proof}

Note that for Proposition \ref{prop:monoidalconv} we only need that $X$ is reasonable. However, to show $\Coh(X)$ is rigid we will need the stronger conditions in (1), hence we impose these now. Similarly, for Proposition \ref{prop:monoidalconv} we only need that $\pi_1$ and $\pi_2$ are cohomologically ind-proper, i.e. they admit a well-behaved pushforward of coherent sheaves. But in proving rigidity we implicitly use properties of $\pi_1^!$ and $\pi_2^!$ that are awkward to impose without genuine ind-properness. 

\subsection{Convolution on the space of triples}\label{sec:alternative} 
We now specialize the discussion of Section~\ref{sec:convgen} to the case where $Y \to Z$ is the natural map $N_\O/\hGO \to N_\K/\hGK$. By Proposition \ref{prop:symmetricRdef}, this results in $X$ being identified with the quotient $\hR/\hGO$. 

\begin{Proposition}\label{prop:convolution}
The hypotheses (\ref{eq:convhypos})	are satisfied when $Y \to Z$ is the map $N_\O/\hGO \to N_\K/\hGK$. In particular, $\Coh^\hGO(\hR)$ is a monoidal category with respect to the convolution product $\cF \conv \cG := m_* \delta^*(\cF \boxtimes \cG)$, where 
\begin{equation}\label{eq:Rconvcorr}
	\begin{tikzpicture}[baseline=(current  bounding  box.center),thick,>=\arrtip]
		\newcommand*{\ha}{4.5}; \newcommand*{\hb}{3.8};
		
		\node (a) at (0,0) {$\hR/\hGO \times \hR/\hGO$};
		\node (a') at (\ha,0) {$\hR/\hGO \times_{\N_\O/\hGO} \hR/\hGO$};
		\node (a'') at (\ha + \hb,0) {$\hR/\hGO$};
		
		\draw[<-] (a) to node[midway,above] {$\delta$} (a');
		\draw[->] (a') to node[midway,above] {$m$} (a'');
	\end{tikzpicture}
\end{equation}
is the natural correspondence. 
\end{Proposition}
\begin{proof}
Conditions (1) and (3) were established in Proposition \ref{prop:Rquotient}. Condition (2) follows from \cite[\tmproptamequotients, \tmpropwprosmoothcohdiagb]{CWtm}. That is, one first notes that  $N_\O$ is weakly smooth since it is a filtered limit $\N_\O \cong \lim \N_\O/t^\ell \N_\O$ of smooth affine schemes along flat maps. One then shows $\N_\O/\hGO$ is tamely presented using the fact that $\hGO$ can be written as a filtered limit of algebraic groups. Finally, one shows $\Delta_{\N_\O/\hGO}$ has \stable coherent pullback by observing that its base change along the flat cover by $\N_\O \times \N_\O$ is the twisted diagonal $\N_\O \times G_\O \to \N_\O \times \N_\O$, $(s, g) \mapsto (s, gs)$, which has \stable coherent pullback since as above we can write $\N_\O \times \N_\O \cong \lim \N_\O/t^\ell \N_\O \times \N_\O/t^\ell \N_\O$. 
\end{proof} 

Our next goal is to describe the convolution correspondence (\ref{eq:Rconvcorr}) as the quotient of a correspondence of $\hGO$-equivariant maps of ind-schemes. This will also connect our description of convolution more directly to that of \cite{BFN}. 

Write $\Gr \ttimes \hR$ for the quotient $\hGK \times_{\hGO} \hR$. It is a Zariski-locally trivial $\hR$-bundle over $\Gr$, and in particular has a projection $p: \Gr \ttimes \hR \to \Gr$. We define $\T \ttimes \hR$ as its base change along the projection $ \T \to \Gr$, and $\hR \ttimes \hR$ as the further base change along $i_2: \hR \to \T$. Note that these ind-schemes are all ind-tamely presented since they are all locally products of the reasonable ind-schemes $\hR$, $\T$, and $\Gr$. 

Let us write $\rd$ for the map $\id \ttimes (\pi_1 \times \id): \Gr \ttimes \hR \to \Gr \ttimes (\N_\O \times \hR) \cong \T \ttimes \hR$. That is, it is the map obtained by passing to $\hGO$-quotients from
$$\id \times (\pi_1 \times \id):  G_\K \times \hR \to G_\K \times (\N_\O \times \hR).$$ 
We also write $\tm: \Gr \ttimes \hR \to \T$ for the composition
$$\Gr \ttimes \hR \xrightarrow{\id \ttimes i_2} \Gr \ttimes \T \cong \Gr \ttimes \Gr \ttimes N_\O \xrightarrow{m_{\Gr} \ttimes \id} \Gr \ttimes \N_\O \cong \T, $$ 
where $m_\Gr$ is the convolution map of $\Gr$. We now consider the following ind-scheme $\convspace$, whose classical locus is denoted by  $q(p^{-1}(\hR \times \hR))$ in \cite[Sec. 3(i)]{BFN}. 

\begin{Proposition}\label{prop:convspacedef}
There exists an ind-tamely presented ind-scheme $\convspace$ fitting into a diagram
\begin{equation}\label{eq:conv}
	\begin{tikzpicture}[baseline=(current  bounding  box.center),thick,>=\arrtip]
		\newcommand*{\hp}{3}; \newcommand*{\hpp}{6}; \newcommand*{\hppp}{11};
		\newcommand*{\vb}{-1.5}; \newcommand*{\vc}{-4};
		\node (a) at (0,0) {$\hR \ttimes \hR$};
		\node (a') at (\hp,0) {$\convspace$};
		\node (a'') at (\hpp,0) {$\hR$};
		\node (b) at (0,\vb) {$\T \ttimes \hR$};
		\node (b') at (\hp,\vb) {$\Gr \ttimes \hR$};
		\node (b'') at (\hpp,\vb) {$\T$};
		
		\draw[<-] (a) to node[above] {$\td $} (a');
		\draw[->] (a') to node[above] {$m $} (a'');
		\draw[<-] (b) to node[above] {$\rd $} (b');
		\draw[->] (b') to node[above] {$\tm $} (b'');
		
		\draw[->] (a) to node[left] {$i_2 \ttimes \id $} (b);
		\draw[->] (a') to node[right] {$i $} (b');
		\draw[->] (a'') to node[right] {$i_2 $} (b'');
	\end{tikzpicture}
\end{equation}
in which both squares are Cartesian and all maps are $\hGO$-equivariant. The maps $\rd$ and $\tm$ respectively have \stable coherent pullback and are ind-proper and almost ind-finitely presented, hence $\td$ and $\tm$ have the same properties. 
\end{Proposition}
\begin{proof}
We define $\convspace$ and the maps $\td$ and $i$ so that the left square in (\ref{eq:conv}) is Cartesian. Let us show that $i$ is also the pullback of $i_2$ along $\tm$. We begin by embedding the left square of (\ref{eq:conv}) into the following sequence of Cartesian squares. 
\begin{equation}\label{eq:3} 
	\begin{tikzpicture}[baseline=(current  bounding  box.center),thick,>=\arrtip]
		\newcommand*{\hp}{-3}; \newcommand*{\hpp}{-6}; \newcommand*{\hppp}{-9};
		\newcommand*{\vb}{-1.5}; \newcommand*{\vc}{-4};
		\node (a) at (0,0) {$\N_\O$};
		\node (a') at (\hp,0) {$\hR$};
		\node (a'') at (\hpp,0) {$\hR \ttimes \hR$};
		\node (a''') at (\hppp,0) {$\convspace$};
		\node (b) at (0,\vb) {$\N_\K$};
		\node (b') at (\hp,\vb) {$\T$};
		\node (b'') at (\hpp,\vb) {$\T \ttimes \hR$};
		\node (b''') at (\hppp,\vb) {$\Gr \ttimes \hR$};
		
		\draw[<-] (a) to node[above] {$\pi_1 $} (a');
		\draw[<-] (a') to node[above] {$p $} (a'');
		\draw[<-] (a'') to node[above] {$\td $} (a''');
		\draw[<-] (b) to node[above] {$ $} (b');
		\draw[<-] (b') to node[above] {$p $} (b'');
		\draw[<-] (b'') to node[above] {$\rd $} (b''');
		
		\draw[->] (a) to node[right] {$j $} (b);
		\draw[->] (a') to node[right] {$i_2 $} (b');
		\draw[->] (a'') to node[right] {$i_2 \ttimes \id $} (b'');
		\draw[->] (a''') to node[left] {$i $} (b''');
	\end{tikzpicture}
\end{equation}
Unwinding the definitions, one checks that the bottom composition coincides with 
$$\Gr \ttimes \hR \xrightarrow{\,\tm\,} \T \xrightarrow{\, \,} \N_\cK.$$ 
Explicitly, both are the composition of the evident map $\Gr \ttimes \hR \to \Gr \ttimes (\Gr \times \N_\K)$ and the map $\Gr \ttimes (\Gr \times \N_\K) \to \N_\K$ obtained by passing to quotients from $G_\K \times G_\K \times \N_\K \to \N_K$,  $(g_1, g_2, s) \mapsto g_1 s$. 
Note that $\tm$ does not coincide with $p \circ \rd$ even though their compositions with~$\T \to \N_\cK$ do. 
We can now refactor the boundary of (\ref{eq:3}) as the boundary of the following diagram.
\begin{equation}\label{eq:4}
	\begin{tikzpicture}[baseline=(current  bounding  box.center),thick,>=\arrtip]
		\newcommand*{\hp}{3}; \newcommand*{\hpp}{6}; \newcommand*{\hppp}{11};
		\newcommand*{\vb}{-1.5}; \newcommand*{\vc}{-4};
		\node (a) at (0,0) {$\convspace$};
		\node (a') at (\hp,0) {$\hR$};
		\node (a'') at (\hpp,0) {$\N_\cO$};
		\node (b) at (0,\vb) {$\Gr \ttimes \hR$};
		\node (b') at (\hp,\vb) {$\T$};
		\node (b'') at (\hpp,\vb) {$\N_\cK$};
		
		\draw[->] (a) to node[above] {$m $} (a');
		\draw[->] (a') to node[above] {$\pi_1 $} (a'');
		\draw[->] (b) to node[above] {$\tm $} (b');
		\draw[->] (b') to node[above] {$ $} (b'');
		
		\draw[->] (a) to node[left] {$i $} (b);
		\draw[->] (a') to node[right] {$i_2 $} (b');
		\draw[->] (a'') to node[right] {$j $} (b'');
	\end{tikzpicture}
\end{equation}
The left square is then Cartesian since the outer and right squares are. 
Since $\tm$ is ind-proper and almost of ind-finite presentation it follows by base change that so is $m$.

Next we show $\rd$ has \stable coherent pullback. Since $i_2$ is ind-proper and almost ind-finitely presented so is its base change $i_2 \ttimes \id$. It will then follow that $\td$ has \stable coherent pullback and that $\convspace$ is ind-tamely presented \cite[\tmpropindPbaseprops]{CWtm}. Consider first the twisted fiberwise diagonal map $\id \ttimes \Delta:  \Gr \ttimes \N_\O \to \Gr \ttimes (\N_\O \times \N_\O)$. Note that its source and target can equivalently be written as $\T$ and $\T \ttimes \N_\O$. Zariski-locally $\id \ttimes \Delta$ is a base change of $\Delta: \N_\O \to \N_\O \times \N_\O$ along a proper, almost finitely presented projection, hence it has \stable coherent pullback since $\N_\O$ is weakly smooth.

Now consider the diagram
\begin{equation*}
	\begin{tikzpicture}[baseline=(current  bounding  box.center),thick,>=\arrtip]
		\newcommand*{\ha}{3.2}; \newcommand*{\hb}{3.2}; \newcommand*{\hc}{3.2};
		\newcommand*{\vb}{-1.7}; 
		\node (a) at (0,0) {$\Gr \ttimes \hR$};
		\node (a') at (\ha,0) {$\T \ttimes \hR$};
		\node (a'') at (\ha+\hb,0) {$\T \ttimes \T$};
		\node (a''') at (\ha+\hb+\hc,0) {$\Gr \ttimes \T$};
		\node (b) at (0,\vb) {$\Gr \ttimes \N_\cO$};
		\node (b') at (\ha,\vb) {$\T \ttimes \N_\cO$};
		\node (b'') at (\ha+\hb,\vb) {$\T \ttimes \N_\cK$};
		\node (b''') at (\ha+\hb+\hc,\vb) {$\Gr \ttimes \N_\cK.$};
		
		\draw[->] (a) to node[above] {$\rd $} (a');
		\draw[->] (a') to node[above] {$\id \ttimes i_2 $} (a'');
		\draw[->] (a'') to node[above] {$\pi_{\Gr} \ttimes \id $} (a''');
		\draw[->] (b) to node[above] {$\id \ttimes \Delta $} (b');
		\draw[->] (b') to node[above] {$\id \ttimes j $} (b'');
		\draw[->] (b'') to node[above] {$\pi_{\Gr} \ttimes \id $} (b''');
		
		\draw[->] (a) to node[left, pos=.4] {$id \ttimes \pi_1 $} (b);
		\draw[->] (a') to node[right, pos=.4] {$\id \ttimes \pi_1 $} (b');
		\draw[->] (a'') to node[right, pos=.4] {$\id \ttimes j_1 $} (b'');
		\draw[->] (a''') to node[right, pos=.4] {$\id \ttimes j_1 $} (b''');
	\end{tikzpicture}
\end{equation*}
Here we continue the convention that e.g. $\id \ttimes j$ acts by the identity on the base of the projection $\T \ttimes \N_\O \to \T$ and by $j$ on its fibers. Since the outermost and right two squares are evidently Cartesian, it follows that the left is as well. 
But $\id \ttimes \pi_1$ is ind-proper and almost ind-finitely presented since $\pi_1$ is, hence $\rd$ has \stable coherent pullback since $\id \ttimes \Delta$ does. 
\end{proof}

The relation between the diagrams (\ref{eq:Rconvcorr}) and (\ref{eq:conv}) involves twisted products, defined as follows. First note that $(\T \ttimes \hR)/\hGO$ and $\T/\hGO \times \hR/\hGO$ are both quotients of $\hGK \times \N_\O \times \hR$. The latter is its quotient by the evident $\hGO^{\,3}$ action, while the former is its quotient by the subgroup $\hGO^{\,2} \subset \hGO^{\,3}$ embedded via $\id \times \Delta$. In particular, there is a natural map $(\T \ttimes \hR)/\hGO \to \T/\hGO \times \hR/\hGO$, and we define a map $u$ by the following Cartesian square. 
\begin{equation*}
	\begin{tikzpicture}[baseline=(current  bounding  box.center),thick,>=\arrtip]
		\node (ab) at (4,0) {$\hR/\hGO \times \hR/\hGO$};
		\node (ac) at (8,0) {$(\hR \ttimes \hR)/\hGO$};
		\node (bb) at (4,-1.5) {$\T/\hGO \times \hR/\hGO$};
		\node (bc) at (8,-1.5) {$(\T \ttimes \hR)/\hGO$};
		\draw[->] (ac) to node[above] {$\tpr$} (ab);
		\draw[->] (bc) to node[above] {$ $} (bb);
		\draw[->] (ab) to node[left] {$i_2 \times \id$} (bb);
		\draw[->] (ac) to node[right] {$i_2 \ttimes \id$} (bc);
	\end{tikzpicture}
\end{equation*}
By construction the bottom map is flat, hence so is $u$. We then define the twisted external product of $\cF, \cG \in \Coh^{\hGO}(\hR)$ as
$$ \cF \,\tbox\, \cG := u^*(\cF \boxtimes \cG) \in \Coh^{\hGO}(\hR \ttimes \hR). $$

We now return to the notation of Proposition \ref{prop:convolution}, setting $Y = \N_\O/\hGO$ and $X = \hR/\hGO$. 

\begin{Proposition}\label{prop:oneconvtorulethemall}
	There exists an isomorphism $X \times_Y X \cong \convspace/\hGO$ fitting into a diagram
	\begin{equation*}
		\begin{tikzpicture}[baseline=(current  bounding  box.center),thick,>=\arrtip]
			\newcommand*{\hp}{3.5}; \newcommand*{\hpp}{7};
			\newcommand*{\vb}{-1.5};
			\node (a) at (0,0) {$X \times X$};
			\node (a') at (\hp,0) {$X \times_Y X$};
			\node (a'') at (\hpp,0) {$X$};
			\node (b) at (0,\vb) {$\hR/\hGO \times \hR/\hGO$};
			\node (b') at (\hp,\vb) {$\convspace/\hGO$};
			\node (b'') at (\hpp,\vb) {$\hR/\hGO$.};
			
			\draw[<-] (a) to node[above] {$\delta $} (a');
			\draw[->] (a') to node[above] {$m$} (a'');
			\draw[<-] (b) to node[above] {$u \circ \td $} (b');
			\draw[->] (b') to node[above] {$m $} (b'');
			
			\draw[->] (a) to node[below,rotate=90] {$\sim $} (b);
			\draw[->] (a') to node[below,rotate=90] {$\sim $} (b');
			\draw[->] (a'') to node[below,rotate=90] {$\sim $}  (b'');
		\end{tikzpicture}
	\end{equation*}
	In particular, for any $\cF, \cG \in \Coh^{\hGO}(\hR)$ we have $\cF \conv \cG \cong m_*\td^*(\cF \tbox \cG)$. 
\end{Proposition}
\begin{proof}
	Let us distinguish the factors of $N_\O/\hGO$ appearing in the definition of $X \times_Y X$ with superscripts, so that
	$$ X \times_Y X \cong \NO^\li/\hGO^\li \times_{\NK/\hGK} \NO^\mi/\hGO^\mi  \times_{\NK/\hGK} \NO^\ri/\hGO^\ri. $$
	We write $\hR^{\li\mi} := \NO^\li/\hGO^\li \times_{\NK/\hGK} \NO^\mi/\hGO^\mi$, similarly for $\T^{\li\mi}$, $\Gr^{\li\mi}$, $\hR^{\mi\ri}$, etc. We decompose the defining square of $X \times_Y X$ as follows. 
	\begin{equation}\label{eq:tripleproductdiagram}
		\begin{tikzpicture}[baseline=(current  bounding  box.center),thick,>=\arrtip]
			\newcommand*{\ha}{4}; \newcommand*{\hb}{4};
			\newcommand*{\va}{-1.5}; \newcommand*{\vb}{-1.5};
			\node (aa) at (0,0) {$\hR^{\li\mi}/\hGO^\li$};
			\node (ab) at (\ha,0) {$(\hR^{\li\mi} \ttimes \hR^{\mi\ri})/\hGO^{\li}$};
			\node (ac) at (\ha+\hb,0) {$X \times_Y X$};
			\node (ba) at (0,\va) {$\T^{\li\mi}/\hGO^\li$};
			\node (bb) at (\ha,\va) {$(\T^{\li\mi} \ttimes \hR^{\mi\ri})/\hGO^\li$};
			\node (bc) at (\ha+\hb,\va) {$(\Gr^{\li\mi} \ttimes \hR^{\mi\ri})/\hGO^\li$};
			\node (ca) at (0,\va+\vb) {$\NO^\mi/\hGO^\mi$};
			\node (cc) at (\ha+\hb,\va+\vb) {$\hR^{\mi\ri}/\hGO^\mi$};
			\foreach \s/\t in {ac/bc, bc/cc} {
				\draw[->] (\s) to (\t); 
			};
			\draw[->] (aa) to node[left] {$i_2$} (ba); 
			\draw[->] (ab) to node[right] {$i_2 \ttimes \id$} (bb); 
			\draw[->] (ba) to node[left] {$\wt{\pi}_2$} (ca); 
			\draw[->] (bb) to node[above] {$p $} (ba); 
			\draw[->] (bc) to node[above] {$\rd$} (bb); 
			\draw[->] (ac) to node[above] {$ $} (ab); 
			\draw[->] (ab) to node[above] {$p $} (aa);
			\draw[->] (cc) to node[above] {$\pi_1$} (ca);
		\end{tikzpicture}
	\end{equation} 
	Here the top left and bottom  rectangles are Cartesian, hence the maps from $X \times_Y X$ to $\hR^{\li\mi}/\hGO^\li$ and $\hR^{\mi\ri}/\hGO^\mi$ factor as shown. It follows that the top right square is also Cartesian. Since the bottom and left maps in this square are the $\hGO^\li$-quotients of those appearing in the Cartesian square used to define $\convspace$, we have $X \times_Y X \xrightarrow{\sim} \convspace/\hGO^\li$ by Lemma \ref{lem:Cartesianstacks}. 
	
	The map $\delta: X \times_Y X \to X \times X$ is the product of the top and right compositions in (\ref{eq:tripleproductdiagram}). On the other hand, it is immediate from the definition of $u$ that under $X \times_Y X \cong \convspace/\hGO^\li$ the right composition is identified with $\pi_2 \circ u \circ \td$ and the top composition with $\pi_1 \circ u \circ \td$. Thus $u \circ \td$ and $\delta$ are also identified under this isomorphism. 
	
	Finally, observe that $X \times_Y X$ is also the fiber product of $\hR^{13}/\hGO^1$ and $\hR^{23}/\hGO^2$ over $\N^3_\O/\hGO^3$. The corresponding Cartesian diagram decomposes as follows. 
	\begin{equation}\label{eq:tripleproductdiagram2}
		\begin{tikzpicture}[baseline=(current  bounding  box.center),thick,>=\arrtip]
			\newcommand*{\ha}{0}; \newcommand*{\hb}{4.3}; \newcommand*{\hc}{4};
			\newcommand*{\va}{-1.5}; \newcommand*{\vb}{-1.5};
			\node (ab) at (\ha,0) {$X \times_Y X$};
			\node (ad) at (\ha+\hb+\hc,0) {$\hR^{\li\ri}/\hGO^\li$};
			\node (bb) at (\ha,\va) {$(\Gr^{\li\mi} \ttimes \hR^{\mi\ri})/\hGO^\li$};
			\node (bc) at (\ha+\hb,\va) {$(\Gr^{\li\mi} \ttimes \T^{\mi\ri})/\hGO^\li$};
			\node (bd) at (\ha+\hb+\hc,\va) {$\T^{\li\ri}/\hGO^\li$};
			\node (cb) at (\ha,\va+\vb) {$\hR^{\mi\ri}/\hGO^\mi$};
			\node (cc) at (\ha+\hb,\va+\vb) {$\T^{\mi\ri}/\hGO^\mi$};
			\node (cd) at (\ha+\hb+\hc,\va+\vb) {$\NO^\ri/\hGO^\ri$};
			\foreach \s/\t in {ab/bb, bb/cb, bc/cc} {
				\draw[->] (\s) to (\t); 
			};
			\draw[->] (ab) to node[above] {$m$} (ad);
			\draw[->] (bb) to node[above] {$\id \ttimes i_2$} (bc);
			\draw[->] (bc) to node[above] {$\tm$} (bd);
			\draw[->] (cb) to node[above] {$i_2$} (cc);
			\draw[->] (cc) to node[above] {$\wt{\pi}_2$} (cd);
			\draw[->] (ad) to node[right] {$i_2$} (bd);
			\draw[->] (bd) to node[right] {$\wt{\pi}_2$} (cd);
		\end{tikzpicture}
	\end{equation}
	The bottom two squares are Cartesian by inspection. The projection $X \times_Y X \to \hR^{\mi\ri}/\hGO^\mi$ thus factors as indicated, and the resulting top square is then Cartesian as well. But comparing with (\ref{eq:conv}) we see that the bottom and right maps in this square are the $\hGO^\li$-quotients of those appearing in the Cartesian square used to define the map $m: \convspace \to \hR$. Thus the isomorphism $X \times_Y X \xrightarrow{\sim} \convspace/\hGO^\li$ also identifies the two maps denoted by $m$. 
\end{proof}

\subsection{Convolution and Koszul-perversity}\label{sec:KPconv}
We now prove the main result of the section. 

\begin{proof}[Proof of Theorem \ref{thm:convsecmainthm}]
	We show that $\cF_1 \conv -$ is right t-exact for any $\cF_1 \in \KPcohGN$, a parallel argument showing left t-exactness and the same claims for $- \conv \cF_1$. Given $\cF_2 \in \Coh^{\hGO}(\hR)^{\leq 0}_\Kzl$, we must show that $\sigma_2^* i_{2*}(\cF_1 \conv \cF_2)$ belongs to $\Coh^{\hGO}(\Gr)_{\Kzl}^{\leq 0}$. In the notation of (\ref{eq:conv}) we have
	\begin{align*}
		\sigma_2^* i_{2*}(\cF_1 * \cF_2) 
		&\cong \sigma_2^* i_{2*} \tm_* \td^* (\cF_1 \tbox \cF_2) \\
		&\cong \sigma_2^* \tm_* \rd^* (i_{2*}(\cF_1) \tbox \cF_2).
	\end{align*}
	By construction $i_{2*}(\cF_1)$ belongs to the Koszul-perverse heart $\Coh^{\hGO}(\T)^\heartsuit_\Kzl$. Generalizing Proposition \ref{prop:koszul-t-bundles} to the infinite-rank case as in the proof of Theorem \ref{thm:koszul-perverse}, we see that $\Coh^{\hGO}(\T)^\heartsuit_\Kzl$ is of finite length. By induction on the length of $i_{2*}(\cF_1)$ we may assume it is simple. By the corresponding generalization of Proposition \ref{prop:koszul2}, we then have $i_{2*}(\cF_1) \cong \pi^* (\cF')$ for some simple $\cF' \in \Coh^{\hGO}(\Gr)^\heartsuit_\Kzl$, where $\pi: \T \to \Gr$ is the projection. But $(\pi \ttimes \id) \circ \rd$ is the identity, so we then have
	\begin{align*}
		\sigma_2^* \tm_* \rd^* (i_{2*}(\cF_1) \tbox \cF_2) 
		&\cong \sigma_2^* \tm_* \rd^* (\pi^* (\cF') \tbox \cF_2) \\
		&\cong \sigma_2^* \tm_* (\cF' \tbox \cF_2). 
	\end{align*}
	
	Now observe that we can factor $\tm$ as the bottom composition in
	\begin{equation*}
		\begin{tikzpicture}[baseline=(current bounding box.center),thick,>=\arrtip]
			\newcommand*{\ha}{3.3}; \newcommand*{\hb}{3.3}; \newcommand*{\vb}{-1.5};
			
			\node (a) at (\ha,0) {$\Gr \ttimes \Gr$};
			\node (a') at (\ha+\hb,0) {$\Gr$};
			\node (b) at (0,\vb) {$\Gr \ttimes \hR$};
			\node (b') at (\ha,\vb) {$\Gr \ttimes \T$};
			\node (b'') at (\ha+\hb,\vb) {$\T$,};
			
			\draw[->] (a) to node[above] {$m_\Gr$} (a');
			\draw[->] (b) to node[above] {$\id \ttimes i_2$} (b');
			\draw[->] (b') to node[above] {$m_\Gr \ttimes \id_{\N_\O}$} (b'');
			\draw[->] (a) to node[pos=.4, left] {$\id \ttimes \sigma_2$} (b');
			\draw[->] (a') to node[right] {$\sigma_2$} (b'');
		\end{tikzpicture}
	\end{equation*}
	where $m_\Gr$ is the convolution map of $\Gr$. The right square is Cartesian, hence 
	\begin{align*}
		\sigma_2^* \tm_* (\cF' \tbox \cF_2) 
		&\cong \sigma_2^* (m_\Gr \ttimes \id_{\N_\O})_* (id \ttimes i_2)_* (\cF' \tbox \cF_2) \\
		&\cong m_{\Gr*} (\id \ttimes \sigma_2)^* (id \ttimes i_2)_* (\cF' \tbox \cF_2) \\
		&\cong m_{\Gr*} (\cF' \tbox \sigma_2^*i_{2*}(\cF_2)) \\
		&\cong \cF' \conv \sigma_2^*i_{2*}(\cF_2).
	\end{align*}
	But convolution in $\Coh^{\hGO}(\Gr)$ is t-exact for the regraded perverse t-structure since it is so for the ordinary perverse t-structure, and since it is additive with respect to weights for the scaling $\C^\times$. The claim follows.
	
	Finally, if $\Gr_e \subset \Gr$ is the identity point, it follows from the definitions that the restriction of $\pi_1: \hR \to \N_\O$ to $\hR_e^{\cl}$ is an isomorphism, which in turn identifies the (quotient of the) ind-closed immersion $\hR_e^{\cl} \to \hR$ as the unit map $e: Y \to X$. The monoidal unit is thus $e_*(\cO_{\hR_e^{\cl}})$, which is clearly Koszul-perverse. It follows that $\KPcohGN$ inherits a monoidal structure. 
\end{proof}

In general the functors $\sigma_1^*i_{1*}$ and $\sigma_2^*i_{2*}$ are not monoidal. Nevertheless, the proof of Theorem~\ref{thm:convsecmainthm} (which can be adapted for $\sigma_1^*i_{1*}$) shows that they do induce ring homomorphisms at the level of K-theory. 

\section{Rigidity}\label{sec:rigid}

In this section we discuss the rigidity of $\Coh^\hGO(\hR_{G,N})$. The following is the main result. 

\begin{Theorem}\label{thm:adjoints}
The monoidal categories $\Coh^\hGO(\hR_{G,N})$ and $\KPcohGN$ are rigid. The left and right duals of $\cF$ are given by $\cF^L \cong \D_1(\cF^*) \cong \D_2(\cF)^* $ and $\cF^R \cong \D_1(\cF)^* \cong \D_2(\cF^*)$, respectively. 
\end{Theorem}

Here $\D_1$ and $\D_2$ are two relative duality functors on $\Coh^\hGO(\hR_{G,N})$, and $(-)^*$ is a certain involution which exchanges them. These ingredients and the proof of Theorem \ref{thm:adjoints} are described first in a general context in Section~\ref{sec:rigidity1}. We then describe them more explicitly in terms of the geometry of $\hR_{G,N}$ in Section \ref{sec:duality}.

\subsection{Duals in convolution categories}\label{sec:rigidity1} 

We begin with some generalities on rigidity. Let $\catC$ be a monoidal category with product $\conv: \catC \times \catC \to \catC$ and unit $1 \in \catC$. Recall that $\cF \in \catC$ is right dualizable if there exists $\cF^R \in \catC$ and maps $u: 1 \to \cF^R \conv \cF$, $c: \cF \conv \cF^R \to 1$ such that the compositions
$$ \cF \xrightarrow{\id \conv u} \cF \conv \cF^R \conv \cF \xrightarrow{\id \conv c} \cF \quad\quad \cF^R \xrightarrow{\id \conv u} \cF^R \conv \cF \conv \cF^R \xrightarrow{\id \conv c} \cF^R $$
are homotopic to the identity maps. The right dual $\cF^R$ is unique up to isomorphism if it exists. Likewise, $\cF$ is left dualizable if it is the right dual of another object $\cF^L$, and $\catC$ is rigid if every object is both left and right dualizable. 

The duality axioms imply that $\cF^R \conv - $ is right adjoint to $\lambda_\cF := \cF \conv -$, while $- \conv \cF^L$ is right adjoint to $\rho_\cF := - \conv \cF$. Conversely, suppose we know that $\lambda_\cF$ has a right adjoint $\lambda_\cF^R$, for example because $\catC$ is presentable and $\cF \conv -$ preserves colimits. 
If $\cF$ is right dualizable, its right dual is necessarily given by $\lambda_\cF^R(1)$. Moreover, the following variant of \cite[Lem. I.1.9.1.5]{GR17} lets us reformulate right dualizability of $\cF$ as a condition on $\lambda_\cF^R$. Its statement refers to the Beck-Chevalley transformation $\lambda_\cF^R(-) \conv \cG \to \lambda_\cF^R(- \conv \cG)$ induced by the associativity isomorphism $( \cF \conv -) \conv \cG \cong  \cF \conv (- \conv \cG)$, as well as its counterpart involving $\rho_\cF^R$. 

\begin{Lemma}\label{lem:dualityviaBCmaps}
The object $\cF \in \catC$ is right (resp. left) dualizable if and only if the transformation $\lambda_\cF^R(-) \conv \cG \to \lambda_\cF^R(- \conv \cG)$ (resp. $\cG \conv \rho_\cF^R(-) \to \rho_\cF^R(\cG \conv -)$) is an isomorphism for every $\cG \in \catC$. 
\end{Lemma}
\begin{proof}
We consider right dualizability explicitly, left dualizability being similar. The only if direction follows since, unwinding the definitions, the given transformation is the associativity isomorphism $(\cF^R \conv -) \conv \cG \cong \cF^R \conv (- \conv \cG)$. For the if direction, we let $\wt{c}: \lambda_\cF \lambda_\cF^R \to \id_\catC$ denote the counit of the adjunction, and set $c := \wt{c}(1): \cF \conv \lambda_\cF^R(1) \to 1$. 
To show $c$ realizes $\lambda_\cF^R(1)$ as a right dual to $\cF$, it suffices to show that for all $\cG, \cH \in \catC$ the induced map
\begin{equation}\label{eq:adjlem1} 
	\Maps_\catC(\cG, \lambda_\cF^R(1) \conv \cH) \xrightarrow{\lambda_\cF} \Maps_\catC(\cF \conv \cG, \cF \conv \lambda_\cF^R(1) \conv \cH) \xrightarrow{(c \conv \id_{\cH}) \circ -} \Maps_\catC(\cF \conv \cG, \cH) 
\end{equation}
is an isomorphism \cite[Lem. 4.6.1.6]{LurHA}. On the other hand, by adjunction the composition
\begin{equation}\label{eq:adjlem2}
	\Maps_\catC(\cG, \lambda_\cF^R(\cH)) \xrightarrow{\lambda_\cF} \Maps_\catC(\cF \conv \cG, \cF \conv \lambda_\cF^R(\cH)) \xrightarrow{\wt{c}(\cH) \circ -} \Maps_\catC(\cF \conv \cG, \cH) 
\end{equation}
is an isomorphism, and by hypothesis the left and middle terms in (\ref{eq:adjlem1}) and(\ref{eq:adjlem2}) are isomorphic. Unwinding the unit/counit axioms and the definition of the Beck-Chevalley map now yields that the composition of $\cF \conv \lambda_\cF^R(1) \conv \cH \congto \cF \conv \lambda_\cF^R(\cH)$ with $\wt{c}(\cH)$ is homotopic to $c \conv \id_{\cH}$, hence that (\ref{eq:adjlem1}) is an isomorphism since (\ref{eq:adjlem2}) is. 
\end{proof}

We now return to the setting of Section \ref{sec:convgen}, fixing a map $Y \to Z$ and writing $X \cong Y \times_Z Y$ as before. Recall from Proposition \ref{prop:monoidalconv} that $\Coh(X)$ has a monoidal structure given by convolution provided certain hypotheses are satisfied. 

\begin{Proposition}\label{prop:cohrigidity}
Under the hypotheses (\ref{eq:convhypos}), the monoidal category $\Coh(X)$ is rigid. 
\end{Proposition}
\begin{proof}
Consider the category $\IndCoh(X)$, which is compactly generated by $\Coh(X)$ since $X$ is admissible and ind-tamely presented \cite[\tmsubsecgeomtame]{CWtm}. The monoidal structure on $\Coh(X)$ extends uniquely to a monoidal structure on $\IndCoh(X)$ which preserves colimits in each variable \cite[Lem. 5.3.2.11]{LurHA}. Note that $\cF \in \Coh(X)$ is right dualizable in $\Coh(X)$ if and only if it is in $\IndCoh(X)$. The only if direction is trivial. The if direction follows from convolution in $\IndCoh(X)$ being continuous in each variable and $1 \cong e_*(\cO_Y)$ being coherent, which implies that $\Map_{\IndCoh(X)}(\cF^R,-) \cong \Map_{\IndCoh(X)}(1,  - \conv \cF)$ is continuous and thus that $\cF^R$ is coherent. The same holds for left dualizability. 

Since $\lambda_\cF: \IndCoh(X) \to \IndCoh(X)$ preserves colimits it has a right adjoint $\lambda_\cF^R$. By Lemma \ref{lem:dualityviaBCmaps}, to show $\cF$ is right dualizable we may  show $\lambda_\cF^R(-) \conv \cG \to \lambda_\cF^R(- \conv \cG)$ is an isomorphism for any $\cG \in \IndCoh(X)$. We may further assume $\cG \in \Coh(X)$ since $\lambda_\cF$ preserves compact objects, hence $\lambda_\cF^R$ is continuous \cite[Prop. 5.5.7.2]{LurHTT}. 

Observe next that the square diagram witnessing the associativity isomorphism $( \cF \conv -) \conv \cG \cong  \cF \conv (- \conv \cG)$ can be decomposed as follows. 
\begin{equation}\label{eq:lotsofsquares}
\begin{tikzpicture}
	[baseline=(current  bounding  box.center),thick,>=\arrtip]
	\newcommand*{\ha}{4.0}; \newcommand*{\hb}{4.2}; \newcommand*{\hc}{4.0};
	\newcommand*{\va}{-1.5}; \newcommand*{\vb}{-1.5}; \newcommand*{\vc}{-1.5};
	
	\node (aa) at (0,0) {$\IC(X)$};
	\node (ab) at (\ha,0) {$\IC(X \times X)$};
	\node (ac) at (\ha+\hb,0) {$\IC(X  \times_Y X)$};
	\node (ad) at (\ha+\hb+\hc,0) {$\IC(X)$};
	
	\node (ba) at (0,\va) {$\IC(X \times X)$};
	\node (bb) at (\ha,\va) {$\IC(X \times X \times X)$};
	\node (bc) at (\ha+\hb,\va) {$\IC(X \times X \times_Y X)$};
	\node (bd) at (\ha+\hb+\hc,\va) {$\IC(X \times X)$};
	
	\node (ca) at (0,\va+\vb) {$\IC(X \times_Y X)$};
	\node (cb) at (\ha,\va+\vb) {$\IC(X \times_Y X \times X)$};
	\node (cc) at (\ha+\hb,\va+\vb) {$\IC(X \times_Y X \times_Y X) $};
	\node (cd) at (\ha+\hb+\hc,\va+\vb) {$\IC(X \times_Y X) $};
	
	\node (da) at (0,\va+\vb+\vc) {$\IC(X) $};
	\node (db) at (\ha,\va+\vb+\vc) {$\IC(X \times X) $};
	\node (dc) at (\ha+\hb,\va+\vb+\vc) {$\IC(X \times_Y X) $};
	\node (dd) at (\ha+\hb+\hc,\va+\vb+\vc) {$\IC(X) $};
	
	\draw[->] (aa) to node[above] {$- \boxtimes \cG$} (ab);
	\draw[->] (ab) to node[above] {$\delta^* $} (ac);
	\draw[->] (ac) to node[above] {$m_* $} (ad);
	
	\draw[->] (ba) to node[above] {$- \boxtimes \cG$} (bb);
	\draw[->] (bb) to node[above] {$\delta^*_{23} $} (bc);
	\draw[->] (bc) to node[above] {$m_{23*} $} (bd);
	
	\draw[->] (ca) to node[above] {$- \boxtimes \cG$} (cb);
	\draw[->] (cb) to node[above] {$\delta^*_{23} $} (cc);
	\draw[->] (cc) to node[above] {$m_{23*} $} (cd);
	
	\draw[->] (da) to node[above] {$- \boxtimes \cG $} (db);
	\draw[->] (db) to node[above] {$\delta^* $} (dc);
	\draw[->] (dc) to node[above] {$m_* $} (dd);
	
	\draw[->] (aa) to node[left] {$\cF \boxtimes - $} (ba);
	\draw[->] (ab) to node[right] {$\cF \boxtimes - $} (bb);
	\draw[->] (ac) to node[right] {$\cF \boxtimes - $} (bc);
	\draw[->] (ad) to node[right] {$\cF \boxtimes - $} (bd);
	
	\draw[->] (ba) to node[left] {$\delta^* $} (ca);
	\draw[->] (bb) to node[right] {$\delta^*_{12} $} (cb);
	\draw[->] (bc) to node[right] {$\delta^*_{12} $} (cc);
	\draw[->] (bd) to node[right] {$\delta^* $} (cd);
	
	\draw[->] (ca) to node[left] {$m_* $} (da);
	\draw[->] (cb) to node[right] {$m_{12*} $} (db);
	\draw[->] (cc) to node[right] {$m_{12*} $} (dc);
	\draw[->] (cd) to node[right] {$m_{*} $} (dd);
\end{tikzpicture} 
\end{equation} 
Here we abbreviate $\IndCoh$ to $\IC$, and write e.g. $m_{12}$ for $m \times \id_X$. Associated to (\ref{eq:lotsofsquares}) we then have the following factorization of $\lambda_\cF^R(-) \conv \cG \to \lambda_\cF^R(- \conv \cG)$ into Beck-Chevalley maps. 
\begin{equation}\label{eq:longeqblock}
\begin{aligned}
	m_* \delta^*  (- \boxtimes \cG) (\cF \boxtimes -)^R \delta_* m^! &\to m_* \delta^*   (\cF \boxtimes -)^R (- \boxtimes \cG) \delta_* m^!  \\  
	&\to m_* \delta^*   (\cF \boxtimes -)^R \delta_{12*} (- \boxtimes \cG)  m^!  \\  
	&\to m_* \delta^*   (\cF \boxtimes -)^R \delta_{12*} m^!_{12} (- \boxtimes \cG)   \\  
	&\to m_* (\cF \boxtimes -)^R \delta_{12*} \delta^*_{23}    m^!_{12} (- \boxtimes \cG)  \\  
	&\to m_* (\cF \boxtimes -)^R \delta_{12*} m^!_{12} \delta^*    (- \boxtimes \cG)   \\  
	&\to (\cF \boxtimes -)^R  m_{23*}  \delta_{12*} m^!_{12} \delta^*     (- \boxtimes \cG) \\  
	&\to (\cF \boxtimes -)^R  \delta_{*} m^! m_{*} \delta^* (- \boxtimes \cG)
\end{aligned}
\end{equation}

Since $X$ is admissible and ind-tamely presented so are all of its Cartesian powers \cite[\tmsecindtamemorphisms]{CWtm}. The fiber product $X \times_Y X$ is then ind-tamely presented since $X$ is and since $m$ is ind-proper and almost ind-finitely presented, and it is admissible since $X \times X$ is and since $\delta$ is affine. Similar reasoning applies to the other stacks in (\ref{eq:lotsofsquares}). Moreover, unwinding the definitions one sees that the four bottom right squares are associated to Cartesian diagrams of stacks. It now follows that the first, second, third, fifth, sixth, and seventh maps in (\ref{eq:longeqblock}) are isomorphisms by 
\cite[\igpropeFXReFXindgeomcoh]{CWig}, 
\cite[\igpropeFXfunctoriality]{CWig}
\cite[\igpropeFXuppershriekcompatindgeom]{CWig}, 
\cite[\tmpropupshriekupstargenICcaseIG]{CWtm}, 
\cite[\igpropeFXRlowerstar]{CWig}, and 
\cite[\igpropupshrieklowstargeom]{CWig}, respectively. 

Finally, note that the composition $p_{23} \circ \delta_{12}: X \times_Y X \times X \to X \times X$ is ind-proper and almost ind-finitely presented since it is the base change of $\pi_2: X \to Y$ along $\pi_1 \circ p_1: X \times X \to Y$ (where $p_{23}$, $p_1$ are projections indicated by subscripts). Conversely $p_1: X \times_Y X \times X \to X$ is the base change of $\pi_1 \circ p_1$ along $\pi_2$, and since $\pi_1 \circ p_1$ is ind-tamely presented with weakly smooth target the  fourth map above is an isomorphism by \cite[\tmpropeFXRupperstarcohpullindcase]{CWtm}. It follows that $\cF$ is right dualizable, and a similar argument shows it is left dualizable. 
\end{proof}

We now derive the formulas for $\cF^L$ and $\cF^R$ given in Theorem \ref{thm:adjoints}. We define $\omega_1, \omega_2 \in \IndCoh(X)$ by $\omega_1 := \pi^!_1(\cO_Y)$ and $\omega_2 := \pi^!_2(\cO_Y)$, and set
\begin{equation*}
\D_1 := \cHom(-, \omega_1) \quad \D_2 := \cHom(-, \omega_1). 
\end{equation*}
We remind the reader that in the context of $\IndCoh(X)$, which does not have an internal tensor product, $\cHom(\cF, \omega_1)$ is defined as $(\cF \boxtimes -)^R \Delta_{X*}(\omega_1)$ for any $\cF \in \Coh(X)$. Recall the involution $\iota: X \to X$ induced by exchanging the factors in $X = Y \times_Z Y$. We let $(-)^*$ denote the induced involution $\iota_* \cong \iota^*$ of $\IndCoh(X)$. 

\begin{Proposition}\label{prop:cohdualformula}
For any $\cF \in \Coh(X)$ we have $\lambda_\cF^R(1) \cong \D_1(\cF)^* \cong \D_2(\cF^*)$ and $\rho_\cF^R(1) \cong \D_1(\cF^*) \cong \D_2(\cF)^*$.
\end{Proposition}
\begin{proof}
By definition $\lambda_\cF^R(1) \cong (\cF \boxtimes -)^R \delta_* m^! e_*(\cO_Y)$. We define a map $a_1$ so that the diagram 
\begin{equation}\label{eq:antidiagcube}
	\begin{tikzpicture}[baseline=(current  bounding  box.center),thick,>=\arrtip]
					\newcommand*{\ha}{1.5}; \newcommand*{\hb}{1.5}; \newcommand*{\hc}{1.5}; \newcommand*{\hd}{1.5}; \newcommand*{\he}{1.5};
		\newcommand*{\va}{-.9}; \newcommand*{\vb}{-.9}; \newcommand*{\vc}{-.9}; 
		\node (ab) at (\ha,0) {$X$};
		\node (ad) at (\ha+\hb+\hc,0) {$X \times_Y X$};
		\node (af) at (\ha+\hb+\hc+\hd+\he,0) {$X$};
		\node (ba) at (0,\va) {$Y$};
		\node (bc) at (\ha+\hb,\va) {$X$};
		\node (cb) at (\ha,\va+\vb) {$Y$};
		\node (cd) at (\ha+\hb+\hc,\va+\vb) {$X$};
		\node (cf) at (\ha+\hb+\hc+\hd+\he,\va+\vb) {$Y$};
		\node (da) at (0,\va+\vb+\vc) {$Z$};
		\node (dc) at (\ha+\hb,\va+\vb+\vc) {$Y$};
		
		\draw[<-] (ab) to node[above,pos=.6] {$\pi'_1 $} (ad);
		\draw[->] (ab) to node[above left, pos=.25] {$\pi_2 $} (ba);
		\draw[->] (ab) to node[right,pos=.2] {$\pi_1 $} (cb);
		\draw[->] (ad) to node[below right] {$\pi'_2 $} (bc);
		\draw[->] (ad) to node[right] {$m $} (cd);
		\draw[->] (ba) to node[left] {$ $} (da);
		\draw[<-] (cb) to node[above,pos=.25] {$ $} (cd);
		\draw[->] (cb) to node[above left, pos=.25] {$ $} (da);
		\draw[->] (cd) to node[below right] {$ $} (dc);
		\draw[<-] (da) to node[above,pos=.75] {$ $} (dc);
		
		\draw[->] (af) to node[right] {$\pi_1 $} (cf);
		
		\draw[-,line width=6pt,draw=white] (ba) to  (bc);
		\draw[<-] (ba) to node[below,pos=.75] {$\pi_1 $} (bc);
		\draw[-,line width=6pt,draw=white] (bc) to  (dc);
		\draw[->] (bc) to node[right,pos=.2] {$\pi_2 $} (dc);
		
		\draw[<-] (ad) to node[above,pos=.5] {$a_1 $} (af);
		\draw[<-] (cd) to node[above,pos=.5] {$e $} (cf);
	\end{tikzpicture}
\end{equation}
consists of Cartesian squares. As $\pi'_1, \pi'_2: X \times_Y X \to X$ are the projections onto the first and second factors, $a_1$ has the property that $\pi'_1 \circ a_1 \cong \id_X$ and $\pi'_2 \circ a_1 \cong \iota$. The composition $\delta \circ a_1$ is thus an antidiagonal with respect to $\iota$, satisfying $\delta \circ a_1 \cong (\id_X \times \iota) \circ \Delta_X$. We then have 
\begin{align}\label{eq:dualformula}
\begin{split}
(\cF \boxtimes -)^R \delta_* m^! e_*(\cO_Y) &\cong (\cF \boxtimes -)^R \delta_* a_{1*} \pi^!_1(\cO_Y)\\ 
&\cong (\cF \boxtimes -)^R (\id_X \times \iota)_* \Delta_{X*} \pi^!_1(\cO_Y)\\ 
&\cong \iota_* (\cF \boxtimes -)^R \Delta_{X*} \pi^!_1(\cO_Y),
\end{split}
\end{align}
the first isomorphism following from \cite[\igpropupshrieklowstargeom]{CWig} and the third since $\cF \boxtimes \iota^*(-) \cong (\id_X \times \iota)^*(\cF \boxtimes -)$. But the last term is $\D_1(\cF)^*$ by definition. Replacing the rightmost map in (\ref{eq:antidiagcube}) with~$\pi_2$ defines a map $a_2:X \to X \times_Y X$ such that $\delta \circ a_2 \cong (\iota \times \id_X) \circ \Delta_X$. Varying (\ref{eq:dualformula}) with~$a_2$ in place of $a_1$ now leads to $\rho_\cF^R(1) \cong \D_2(\cF^*)$. The case of left duals is similar. 
\end{proof}

Unwinding the definitions, it follows that for any $\cF \in \Coh(X)$ the overall isomorphism
$$ \cF \cong (\cF^R)^L \cong \D_1(\D_1(\cF)^{**}) \cong \D_1(\D_1(\cF))$$
is the canonical map from the leftmost term to the rightmost. In particular, the object $\omega_1 \in \IndCoh(X)$ is dualizing, likewise for $\omega_2$. 

It follows from Proposition \ref{prop:cohrigidity} and \ref{prop:cohdualformula} that $\D_1(\cF \conv \cG)^* \cong \D_1(\cG)^* \conv \D_1(\cF)^*$ functorially in $\cF, \cG \in \Coh(X)$, similarly for $\D_2$. Since $(\cF \conv \cG)^* \cong \cG^* \conv \cF^*$ this implies the following. 

\begin{Corollary}\label{cor:Ds}
	There exists an isomorphism $\D_1(- \conv -) \cong \D_1(-) \conv \D_1(-)$ of functors $\Coh(X) \times \Coh(X) \to \Coh(X)$, likewise for $\D_2$. 
\end{Corollary}

\subsection{Duals on the space of triples}\label{sec:duality}

We now discuss some aspects of rigidity particular to $\Coh^\hGO(\hR)$. We begin by recalling the following standard fact, which together with Propositions~\ref{prop:cohrigidity} and~\ref{prop:cohdualformula} completes the proof of Theorem \ref{thm:adjoints}.

\begin{Proposition}
Let $\catC$ be a monoidal stable $\infty$-category equipped with a t-structure such that $1 \in \catC^\heartsuit$, and such that $X \conv - $ and $- \conv X$ are t-exact for any $X \in \catC^\heartsuit$. Then $\catC^\heartsuit$ is closed under left and right duals.
\end{Proposition}
\begin{proof}
Suppose $X \in \catC^\heartsuit$. Then $X^R \conv -$ is left t-exact since $X \conv -$ is right t-exact, and $- \conv X^R$ is right t-exact since $- \conv X$ is left t-exact. Then $X^R \cong 1 \conv X^R \cong X^R \conv 1$ belongs to $\catC^\heartsuit$ since $1$ does. The same argument shows $X^L \in \catC^\heartsuit$. 
\end{proof}

Note that $(-)^*$ is manifestly t-exact with respect to the Koszul-perverse t-structure on $\Coh^\hGO(\hR)$, hence it follows that $\D_1$ and $\D_2$ are as well. 

Next we consider how to compute duals in $\Coh^\hGO(\hR)$. We will see that the involution $(-)^*$ and the sheaves $\omega_1$ and $\omega_2$ are determined in an explicit way by their counterparts in case where $\N \cong 0$ and $\hR \cong \Gr$. Note however that the presentation of $\hR/\hGO$ as a quotient of $\hR$ breaks the symmetry between $\pi_1$ and $\pi_2$, hence necessarily obscures the involution $(-)^*$ and the symmetry between $\omega_1$ and $\omega_2$. 

Write $\pi_{1,\Gr}$ and $\pi_{2,\Gr}$  for the projections $\Gr/\hGO \to \pt/\hGO$ in the $\N \cong 0$ case. Unwinding our conventions, $\pi_{1,\Gr}$ is obtained from $\Gr \to \pt$ by taking $\hGO$-quotients. Next write $\omega_{1,\Gr} := \pi_{1,\Gr}^!(\cO_\pt)$ and $\omega_{2,\Gr} := \pi_{2,\Gr}^!(\cO_\pt)$, so that $\omega_{1,\Gr}$ is the usual equivariant dualizing sheaf on $\Gr$ and $\omega_{2,\Gr}$ is its image under $(-)^*$. Recall from Proposition \ref{prop:symmetricRdef} the factorizations $\pi_1 \cong \wt{\pi}_1 \circ i_1$ and $\pi_2 \cong \wt{\pi}_2 \circ i_2$, and write $p_{1}: \Gr \times \N_\O \to \Gr$ and $p_2: \T \to \Gr$ for the projections. 

\begin{Proposition}
We have $\omega_1 \cong i_1^! p_1^*(\omega_{1,\Gr})$ and $\omega_2 \cong i_2^! p_2^*(\omega_{2,\Gr})$. 
\end{Proposition}
\begin{proof}
Note that $\pi_1'$ is the base change of $\pi_{1,\Gr}$ along $p: \N_\O/\hGO \to \pt/\hGO$, while $p_1$ the base change of $p$ along $\pi_{1,\Gr}$. The first claim now follows since $i_1^!\pi'^!_{1}p^*(\cO_\pt) \cong i_1^!p^*_1 \pi_{1,\Gr}^!(\cO_\pt)$ \cite[\tmpropupshriekupstargenICcaseIG]{CWtm}, and the second follows similarly. 
\end{proof}

To describe $(-)^*$ explicitly, note that $\iota \cong \pi'_2 \circ a_1$, where $\pi'_2: \convspace/\hGO \to \hR/\hGO$ is the base change of $\pi_2$ along $\pi_1$ and, as in the proof of Proposition \ref{prop:cohdualformula}, the map $a_1$ is defined by the following Cartesian square. 
\begin{equation*}
	\begin{tikzpicture}[baseline=(current  bounding  box.center),thick,>=\arrtip]
		\node (ab) at (0,0) {$\hR/\hGO$};
		\node (ac) at (3,0) {$\convspace/\hGO$};
		\node (bb) at (0,-1.5) {$\N_\O/\hGO$};
		\node (bc) at (3,-1.5) {$\hR/\hGO$};
		\draw[->] (ab) to node[above] {$a_1$} (ac);
		\draw[->] (bb) to node[above] {$e $} (bc);
		\draw[->] (ab) to node[left] {$\pi_1$} (bb);
		\draw[->] (ac) to node[right] {$m$} (bc);
	\end{tikzpicture}
\end{equation*}
Note that $a_1$ lifts to a $\hGO$-equivariant map $\hR \to \convspace$ (which is not true of its counterpart~$a_2$). Recall from the proof of Proposition \ref{prop:oneconvtorulethemall} that $\pi'_2$ factors through (the quotient of) the map $i: \convspace \to \Gr \ttimes \hR$. It follows that we can recover $\iota$ as $\pi_2 \circ  s$ where $s := i \circ a_1$. 

\begin{Proposition}
The following square is Cartesian 
\begin{equation*}
	\begin{tikzpicture}[baseline=(current  bounding  box.center),thick,>=\arrtip]
		\node (ab) at (0,0) {$\hR$};
		\node (ac) at (3,0) {$\Gr \ttimes \hR$};
		\node (bb) at (0,-1.5) {$\Gr$};
		\node (bc) at (3,-1.5) {$\Gr \ttimes \Gr$};
		\draw[->] (ab) to node[above] {$s$} (ac);
		\draw[->] (bb) to node[above] {$s_{\Gr}$} (bc);
		\draw[->] (ab) to node[left] {$p$} (bb);
		\draw[->] (ac) to node[right] {$\id \ttimes p$} (bc);
	\end{tikzpicture}
\end{equation*}
\end{Proposition}
\begin{proof}
Here $s_{\Gr}: \Gr \to \Gr\ttimes \Gr$ is the counterpart of $s$ when $\N \cong 0$ (this was denoted just by $s$ in \cite{CW1}). By definition we have the following pair of Cartesian squares. 
\begin{equation*}
	\begin{tikzpicture}[baseline=(current  bounding  box.center),thick,>=\arrtip]
		\newcommand*{\hp}{-3}; \newcommand*{\hpp}{-6}; 
		\newcommand*{\vb}{-1.5}; 
		\node (a) at (0,0) {$\Gr \ttimes \hR$};
		\node (a') at (\hp,0) {$\convspace$};
		\node (a'') at (\hpp,0) {$\hR$};
		\node (b) at (0,\vb) {$\T$};
		\node (b') at (\hp,\vb) {$\hR$};
		\node (b'') at (\hpp,\vb) {$\N_\O$.};
		
		\draw[<-] (a) to node[above] {$i$} (a');
		\draw[<-] (a') to node[above] {$a_1$} (a'');
		\draw[<-] (b) to node[above] {$i_2 $} (b');
		\draw[<-] (b') to node[above] {$e $} (b''); 
		
		\draw[->] (a) to node[right] {$\tm $} (b);
		\draw[->] (a') to node[right] {$m $} (b');
		\draw[->] (a'') to node[left] {$\pi_1 $}  (b'');
	\end{tikzpicture}
\end{equation*}
Recall that $\tm$ factors as $\Gr \ttimes \hR \xrightarrow{\id \times i_2} \Gr \ttimes \T \xrightarrow{m_\Gr \ttimes \id_{\N_\O}} \T$. The base change of $i_2 \circ e$ along $m_\Gr \ttimes \id_{\N_\O}$ is the map $s_{\Gr} \ttimes \id_{\N_\O}: \T \to \Gr \ttimes \T$. But this map is the base change of $s_{\Gr}$ along the projection $\Gr \ttimes \T \to \Gr \ttimes \Gr$, and the claim follows. 
\end{proof}

Consider now the following Cartesian squares. 
\begin{equation*}
	\begin{tikzpicture}[baseline=(current  bounding  box.center),thick,>=\arrtip]
		\node (ab) at (0,0) {$\hGO \backslash \hR$};
		\node (ac) at (3,0) {$\hGO \backslash \Gr \ttimes \hR$};
		\node (bb) at (0,-1.5) {$\hGO \backslash \Gr$};
		\node (bc) at (3,-1.5) {$\hGO \backslash \Gr \ttimes \Gr$};
		\node (ad) at (6,0) {$\hGO \backslash \hR$};
		\node (bd) at (6,-1.5) {$\hGO \backslash \Gr$};
		\draw[->] (ab) to node[above] {$s$} (ac);
		\draw[->] (bb) to node[above] {$s_{\Gr}$} (bc);
		\draw[->] (ab) to node[left] {$p$} (bb);
		\draw[->] (ac) to node[right] {$\id \ttimes p$} (bc);
		\draw[->] (ac) to node[above] {$\pi_2$} (ad);
		\draw[->] (bc) to node[above] {$\pi_2$} (bd);
		\draw[->] (ad) to node[right] {$p$} (bd);
	\end{tikzpicture}
\end{equation*}
If we base change along $\hGO \backslash \Gr_{\le \l^\vee} \to \hGO \backslash \Gr$ then the bottom row becomes
$$\hGO \backslash \Gr_{\le {\l^\vee}^*} \xrightarrow{s_{Gr}} \hGO \backslash \Gr \ttimes \Gr_{\le \l^\vee} \xrightarrow{\pi_2} \hGO \backslash \Gr_{\le {\l^\vee}},$$
where the composition is the isomorphism $\iota$. Subsequently the top row becomes 
$$\hGO \backslash \hR_{\le {\l^\vee}^*} \xrightarrow{s} \hGO \backslash \Gr \ttimes \hR_{\le \l^\vee} \xrightarrow{\pi_2} \hGO \backslash \hR_{\le {\l^\vee}}$$
where the composition is again $\iota$. The left map factors through $\hGO \backslash \Gr_{\le {\l^\vee}^*} \ttimes \hR_{\le \l^\vee}$ so we can rewrite this composition as
$$\hGO \backslash \hR_{\le {\l^\vee}^*} \xrightarrow{s} \hGO \backslash \Gr_{\le {\l^\vee}^*} \ttimes \hR_{\le \l^\vee} \xrightarrow{\pi_2} \hGO \backslash \hR_{\le {\l^\vee}}.$$
Noting that $\iota^*(-) \cong s^* \pi_2^* (-) \cong s^*(\O_{\Gr_{\le {\l^\vee}^*}} \tbox (-))$ this gives us the following.

\begin{Proposition}\label{prop:*-atlas}
	If $\cF \in \Coh^\hGO(\hR_{\le \l^\vee})$ then $\cF^* \cong s^*(\O_{\Gr_{\le {\l^\vee}^*}} \tbox \cF)$ where $s$ is characterized by the following Cartesian square.
\begin{equation*}
	\begin{tikzpicture}[baseline=(current bounding box.center),thick,>=\arrtip]
		\newcommand*{\ha}{3.2}; \newcommand*{\va}{-1.5};
		
		\node (a) at (0,0) {$\hR_{\le {\l^\vee}^*}$};
		\node (a') at (\ha,0) {$\Gr_{\le {\l^\vee}^*} \ttimes \hR_{\le \l^\vee}$};
		\node (b) at (0,\va) {$\Gr_{\le {\l^\vee}^*}$};
		\node (b') at (\ha,\va) {$\Gr_{\le {\l^\vee}^*} \ttimes \Gr_{\le \l^\vee}$};
		
		\draw[->] (a) to node[above] {$s$} (a');
		\draw[->] (b) to node[above] {$s_{\Gr}$} (b');
		\draw[->] (a) to node[left] {$p$} (b);
		\draw[->] (a') to node[right] {$\id \ttimes p$} (b');
	\end{tikzpicture}
\end{equation*}
\end{Proposition}

\section{Factorization}\label{sec:factorization}

In this section we use the factorization structure on $\hR_{G,N}$, introduced in \cite{BFN19}, to produce renormalized $r$-matrices in $\Coh^\hGO(\hR_{G,\N})$. The following is the main result. 


\begin{Theorem}\label{thm:rmatrix}
	The category $\Coh^\hGO(\hR_{G,\N})$ admits a canonical system of renormalized $r$-matrices, which restricts to one on $\KPcohGN$. 
\end{Theorem}

Explicitly, this means that for any $\cF, \cG \in \Coh^\hGO(\hR)$ we have a nonzero map $\rmat{\cF,\cG}: \cF \conv \cG \to \cG \conv \cF \{\La(\cF,\cG)\}$ for some $\Lambda(\cF,\cG) \in \Z$, and that collectively these satisfy various properties familiar from the theory of quantum groups \cite[Def. 4.1]{CW1}. 

Theorem \ref{thm:rmatrix} is an application of the general construction of renormalized $r$-matrices from factorization structures explained in \cite[Section 5]{CW1}. The idea is the following. First, we deform $\cF \conv \cG$ and $\cG \conv \cF$ to sheaves $\widetilde{\cF \conv \cG}$ and $\widetilde{\cG \conv \cF}$ over $\hR_{\A^2} \to \A^2$ so that the general fibers of each are $\cF \boxtimes \cG$, while the fibers along the diagonal are $\cF \conv \cG$ and $\cG \conv \cF$ respectively. If $z_1$ and $z_2$ are the coordinates on $\A^2$, then after multiplying by a suitable power of $z_1-z_2$ the evident generic isomorphism has a nonzero extension across the diagonal, and specializing it to $(0,0)\in \A^2$ we obtain $\rmat{\cF,\cG}$ (the integer ${\Lambda(\cF,\cG)}$ being the exponent of $z_1-z_2$ appearing in this procedure, hence the resulting loop shift). 

The different sections below explain the features of the factorizable space of triples needed to apply \cite[Theorem 5.10]{CW1}: we consider the global categories $\Coh^{\hGOX}(\hR_{X^I})$ in Section~\ref{sec:fact-spaces}, their convolution structures in Section~\ref{sec:globalconv}, and their unital structures in Section~\ref{sec:globalunital}. Most of the constructions are natural generalizations of results from Sections~\ref{sec:convolution} and~\ref{sec:rigid}.

\subsection{Factorization spaces}\label{sec:fact-spaces}

We first recall the factorization version $\hR_{\curv^\bullet}$ of the space of triples, where $X$ is a smooth curve \cite{BFN19}. We begin with a discussion of global jet and loop spaces along the lines of \cite{KV04}. Let $S$ be a classical affine scheme and $f_I = (f_i)_{i \in I}$ a tuple of maps  $S \to \curv^I$. We write $z_I \subset S \times \curv$ for the union of the graphs of the $f_i$. The formal neighborhood of $z_I$ is the colimit of all closed subschemes of $S \times \curv$ with reduced locus $z_I$ --- letting $\cO_{\hz_I}$ denote the inverse limit of coordinate rings of such subschemes, we also have the scheme-theoretic formal neighborhood $\widehat{z}_I := \Spec \cO_{\hz_I}$. Since this contains $z_I$ as a closed subscheme we also have the punctured formal neighborhood $\hz_I \smallsetminus z_I$.

Now let $Y$ be a smooth affine scheme of finite type. The $I$-jet space $\jets{Y}{I}$ and $I$-loop space $\loops{Y}{I}$ of $Y$ are defined by having $S$-points
\begin{align*}
\jets{Y}{I}(S) &:= \{f_I: S \to \curv^I, \rho: \hz_I \to Y \},\\
\loops{Y}{I}(S) &:= \{f_I: S \to \curv^I, \rho: \hz_I \smallsetminus z_I \to Y \}.
\end{align*}
More generally, one can define
$$
Y_{\K,X^{I/J}}(S) := \{f_{I \sqcup J}: S \to \curv^{I \sqcup J}, \rho: \hz_{I \sqcup J} \smallsetminus z_{I} \to Y \}.
$$
Notice that this recovers $Y_{\K,X^I}$ if $J = \varnothing$ and $Y_{\O,X^J}$ if $I = \varnothing$. It is shown in \cite{KV04} that $\jets{Y}{I}$ is a (classical) scheme over $X^I$ and $\loops{Y}{I}$ an ind-scheme over $\curv^I$, and that collectively they give rise to factorization spaces $\jets{Y}{\bullet}$ and $\loops{Y}{\bullet}$ over $\curv$.

\begin{Lemma}\label{lem:flat}
Suppose $Y=N$. Then $N_{\K,X^{I/J}}$ is an ind-scheme which can be presented as $\colim_k N^k_{\K,X^{I/J}}$  where each $N^k_{\K,X^{I/J}}$ is an inverse limit of $N^{k,\ell_I/\ell_J}_{\K,X^{I/J}}$ via flat maps. Moreover, each $N^{k,\ell_I/\ell_J}_{\K,X^{I/J}}$ is a smooth scheme, flat over $X^{I \sqcup J}$ and all the faces in (\ref{eq:Ncart}) are Cartesian.
\begin{equation}\label{eq:Ncart}
	\begin{tikzpicture}[baseline=(current  bounding  box.center),thick,>=\arrtip]
		\newcommand*{\ha}{1.8}; \newcommand*{\hb}{1.8}; \newcommand*{\hc}{1.8};
		\newcommand*{\va}{-1.1}; \newcommand*{\vb}{-1.1}; \newcommand*{\vc}{-1.1}; 
		\node (ab) at (\ha,0) {$N^{k+1,\ell_I/\ell_J+1}_{\K,X^{I/J}}$};
		\node (ad) at (\ha+\hb+\hc,0) {$N^{k+1,\ell_I+1/\ell_J+1}_{\K,X^{I/J}}$};
		\node (ba) at (0,\va) {$N^{k,\ell_I/\ell_J+1}_{\K,X^{I/J}}$};
		\node (bc) at (\ha+\hb,\va) {$N^{k,\ell_I+1/\ell_J+1}_{\K,X^{I/J}}$};
		\node (cb) at (\ha,\va+\vb) {$N^{k+1,\ell_I/\ell_J}_{\K,X^{I/J}}$};
		\node (cd) at (\ha+\hb+\hc,\va+\vb) {$N^{k+1,\ell_I+1/\ell_J}_{\K,X^{I/J}}$};
		\node (da) at (0,\va+\vb+\vc) {$N^{k,\ell_I/\ell_J}_{\K,X^{I/J}}$};
		\node (dc) at (\ha+\hb,\va+\vb+\vc) {$N^{k,\ell_I+1/\ell_J}_{\K,X^{I/J}}$};
		
		\draw[->] (ab) to node[above] {} (ad);
		\draw[->] (ab) to node[above] {} (ba);
		\draw[->] (ab) to node[left] {} (cb);
		\draw[->] (ad) to node[above] {} (bc);
		\draw[->] (ba) to node[above] {} (da);
		\draw[->] (cb) to node[above] {} (cd);
		\draw[->] (cb) to node[above] {} (da);
		\draw[->] (cd) to node[above] {} (dc);
		\draw[->] (da) to node[above] {} (dc);
		
		\draw[-,line width=6pt,draw=white] (ba) to  (bc);
		\draw[->] (ba) to node[above,pos=.75] {} (bc);
		\draw[-,line width=6pt,draw=white] (bc) to  (dc);
		\draw[->] (bc) to node[right,pos=.2] {} (dc);
		\draw[->] (ad) to node[right] {} (cd);
	\end{tikzpicture}
\end{equation}
\end{Lemma}
\begin{proof}
We denote by $D_I, D_J \subset X \times X^{I \sqcup J} \xrightarrow{\varpi} X^{I \sqcup J}$ the natural divisors whose fiber over $\ux \in X^{I \sqcup J}$ is $\sum_{i \in I} x_i$ and $\sum_{j \in J} x_j$ respectively. Since $(Y \times Y')_{\K,X^{I/J}} \cong Y_{\K,X^{I/J}} \times_{X^{I \sqcup J}} Y'_{\K,X^{I/J}}$ we can assume that $N$ is one-dimensional. Then, as in the argument of \cite[Section 2.5]{KV04}, $N_{\K,X^{I/J}}$ is the represented by the inductive-projective limit of the total space of the vector bundles 
$$\varpi_*(\O_{X'}(kD_I) / \O_{X'}(-\ell_I D_I - \ell_J D_J)))$$ 
over $X^{I \sqcup J}$ where $X' = X \times X^{I \sqcup J}$ .  The fact that $N^{k,\ell_I+1/\ell_J}_{\K,X^{I/J}} \to N^{k,\ell_I/\ell_J}_{\K,X^{I/J}}$ is flat follows since the kernel of 
$$\varpi_*(\O_{X'}(kD_I) / \O_{X'}(-(\ell_I+1)D_I - \ell_J D_J)) \to \varpi_*(\O_{X'}(kD_I) / \O_{X'}(-\ell_I D_I - \ell_J D_J))$$
is the vector bundle $\varpi_*(\O_{X'} / \O_{X'}(-D_I))$. A similar argument shows that $N^{k,\ell_I/\ell_J+1}_{\K,X^{I/J}} \to N^{k,\ell_I/\ell_J}_{\K,X^{I/J}}$ is flat. 

Finally, the left face of (\ref{eq:Ncart}) is Cartesian because the following square is Cartesian
\begin{equation}\label{eq:Ncart2}
\xymatrix{
\varpi_*(\O_{X'}(kD_I) / \O_{X'}(-\ell_I D_I - (\ell_J+1)D_J )) \ar[r] \ar[d] & \varpi_*(\O_{X'}((k+1)D_I) / \O_{X'}(-\ell_I D_I - (\ell_J+1)D_J )) \ar[d] \\
\varpi_*(\O_{X'}(kD_I) / \O_{X'}(-\ell_I D_I - \ell_J D_J)) \ar[r] & \varpi_*(\O_{X'}((k+1)D_I) / \O_{X'}(-\ell_I D_I - \ell_J D_J )) }
\end{equation}
which is a consequence of the commutativity of the diagram, the top horizontal map being injective and the top left term having the expected dimension.  The other faces of (\ref{eq:Ncart}) can be shown to be Cartesian in the same way.
\end{proof}

\begin{Remark}
Proposition 2.4.4 in \cite{KV04} implies that $G_{\O,X^I}$ is also an inverse limit of flat maps between smooth schemes which are flat over $X^I$.  
\end{Remark}

The global version of $\Gr_G$ is the Beilinson-Drinfeld (BD) Grassmannian $\Gr_{G,X^I}$ (or $\Gr_{X^I}$ for short). It can be obtained as the quotient $G_{\K,X^I}/G_{\O,X^I}$. Like $\Gr_G$ it has the structure of an ind-scheme which is ind-finitely presented and ind-proper over $X^I$. As in Section \ref{sec:Rdefs} we fix a presentation $\Gr_{X^I} \cong \colim \Gr_{X^I}^\al$. 

The global analogue of $\T_{G,N}$ is $\T_{X^I} := G_{\K,X^I} \times_{G_{\O,X^I}} N_{\O,X^I}$. The natural action of $G_{\K,X^I}$ on $N_{\K,X^I}$ gives us a map $\T_{X^I} \to N_{\K,X^I}$. 

    \begin{Definition}
        The factorization space of triples $\hR_{G,N,X^I}$ (or $\hR_{X^I}$ for short) is the fiber product
        \begin{equation}\label{eq:factorizationR}
        \begin{tikzpicture}[baseline=(current  bounding  box.center),thick,>=\arrtip]
        \newcommand*{\ha}{3}; \newcommand*{\va}{1.5};
        \node (aa) at (0,0) {$\hR_{X^I}$}; \node (ab) at (\ha,0) {$\Gr_{X^I} \times_{X^I} N_{\O,X^I}$};
        \node (ba) at (0,-\va) {$\T_{X^I}$}; \node (bb) at (\ha,-\va) {$\Gr_{X^I} \times_{X^I} N_{\K,X^I}$};
        \draw[->] (aa) to node[above] {$i_1$} (ab); \draw[->] (ba) to node[above] {$ j_2 $} (bb);
        \draw[->] (aa) to node[left] {$i_2 $} (ba); \draw[->] (ab) to node[right] {$ j_1$} (bb);
        \end{tikzpicture}
        \end{equation}
\end{Definition}

The map $\T_{X^I}^{\le \l^\vee} \to N_{\K,X^I}$ factors through some $N^{k}_{\K,X^I}$. This gives an explicit presentation of $\hR_{X^I}$ as a colimit of schemes. 

There is a natural action of $G_{\O,X^I}$ on $\hR_{X^I}$. If $X = \A^1$ then the natural $\C^\times$ action on $X$ induces a $\C^\times$ action on $\hR_{X^I}$ which generalizes the $\C^\times$ action by loop rotation on $\hR$. Moreover, the natural $\C^\times$ action on $N$ of weight one induces an action on $N_{\O,X^I}$ and subsequently one on $\hR_{X^I}$ as well. Putting this together we get an action of 
 $$\hGOX := (G_{\O,X^I} \rtimes \C^\times) \times \C^\times$$
 on $\hR_{X^I}$ where the inner  $\C^\times$ is the loop scaling. 
 
The description of $\hR/\hGO$ from (\ref{eq:Rsymmdiagram}) has the following natural generalization to $\hR_{X^I}$
\begin{equation}\label{eq:Rsymmdiagram-fact}
      \begin{tikzpicture}
        [baseline=(current  bounding  box.center),thick,>=\arrtip]
                \node (a) at (0,0) {$\hR_{X^I}/\hGOX$};
\node (b) at (3.5,0) {$N_{\O,X^{I \sqcup J}}/{\widehat{G}_{\O,X^{I}}}$};
\node (c) at (0,-1.5) {$N_{\O,X^{I}}/{\widehat{G}_{\O,X^{I}}}$};
\node (d) at (3.5,-1.5) {$N_{\K,X^{I}}/{\widehat{G}_{\K,X^{I}}}$};
\draw[->] (a) to node[above] {$\pi_1$} (b);
\draw[->] (b) to node[right] {$ $} (d);
\draw[->] (a) to node[left] {$\pi_2$}(c);
\draw[->] (c) to node[above] {$ $} (d);
\end{tikzpicture}
\end{equation}
This follows by the same argument used in the proof of Proposition \ref{prop:symmetricRdef}.

\subsection{Convolution structure}\label{sec:globalconv}

The convolution of coherent sheaves on $\hR_{X^I}$ uses the following global analogue of (\ref{eq:conv})
\begin{equation}\label{eq:globalconv}
	\begin{tikzpicture}[baseline=(current  bounding  box.center),thick,>=\arrtip]
		\newcommand*{\hp}{3}; \newcommand*{\hpp}{6}; 
		\newcommand*{\vb}{-1.5}; 
		\node (a) at (0,0) {$\hR_{X^I} \ttimes \hR_{X^I}$};
		\node (a') at (\hp,0) {$\convspace_{X^I}$};
		\node (a'') at (\hpp,0) {$\hR_{X^I}$};
		\node (b) at (0,\vb) {$\T_{X^I} \ttimes \hR_{X^I}$};
		\node (b') at (\hp,\vb) {$\Gr_{X^I} \ttimes \hR_{X^I}$};
		\node (b'') at (\hpp,\vb) {$\T_{X^I}$};
		
		\draw[<-] (a) to node[above] {$d $} (a');
		\draw[->] (a') to node[above] {$m $} (a'');
		\draw[<-] (b) to node[above] {$\rd $} (b');
		\draw[->] (b') to node[above] {$\tm $} (b'');
		
		\draw[->] (a) to node[left] {$i_2 \ttimes \id $} (b);
		\draw[->] (a') to node[right] {$i $} (b');
		\draw[->] (a'') to node[right] {$i_2 $} (b'');
	\end{tikzpicture}
\end{equation}
Here $\T_{X^I} \ttimes \hR_{X^I} = (G_{\K,X^I} \times_{X^I} N_{\O,X^I}) \times_{G_{\O,X^I}} \hR_{X^I}$ and similarly for $\Gr_{X^I} \ttimes \hR_{X^I}$. As in Proposition \ref{prop:oneconvtorulethemall} one may define the convolution product of $\cF, \cG \in \Coh^{\hGOX}(\hR_{X^I})$ as
$$\cF \conv \cG := m_* d^* (\cF \tbox \cG) \in \Coh^{\hGOX}(\hR_{X^I}).$$
The proof that this product preserves coherent sheaves is entirely parallel to that from Section~\ref{sec:convolution} for $\hR$. Instead of rewriting out all the details we highlight the main technical properties used. 

First, the map $\tm$ is defined as the composition 
$$\Gr_{X^I} \ttimes \hR_{X^I} \xrightarrow{id \ttimes i_2} \Gr_{X^I} \ttimes \T_{X^I} \to \T_{X^I}$$
where the second map is multiplication. It is ind-proper and almost of ind-finite presentation because both $i_2$ and the multiplication map $\Gr_{X^I} \ttimes \Gr_{X^I} \to \Gr_{X^I}$ have these properties. By base change $m$ also has these properties and hence $m_*$ preserves coherent sheaves. 

Second, the map $\pi_1: \hR_{X^I} \to N_{\O,X^I}$ is almost of ind-finite presentation and ind-proper using the same argument as in Proposition \ref{prop:R} with the key new input being that the map $N_{\O,X^I} \hookrightarrow N_{\K,X^I}$ has both these properties. Lastly, $\td$ has coherent pullback using the same argument as in Proposition \ref{prop:convspacedef} with the new input being that $\Delta: N_{\O,X^I} \to N_{\O,X^I} \times_{X^I} N_{\O,X^I}$ has \stable coherent pullback. 

There is also the natural generalization of Proposition \ref{prop:oneconvtorulethemall} which identifies this product on $\Coh^{\hGOX}(\hR_{X^I})$ with the natural convolution product obtained from the realization of $\hR_{X^I}/\hGOX$ as the fiber product from (\ref{eq:Rsymmdiagram-fact}). As in Section \ref{sec:convolution}, it follows that this convolution is associative. Moreover, applying Proposition \ref{prop:cohrigidity}, we find that $\Coh^{\hGOX}(\hR_{X^I})$ is also rigid. 

\subsection{Unital structure}\label{sec:globalunital}

To obtain a unital structure we will use the intermediate space $\hR_{X^{I/J}}$ defined as the fiber product
\begin{equation}\label{eq:factorizationR2}
        \begin{tikzpicture}[baseline=(current  bounding  box.center),thick,>=\arrtip]
        \newcommand*{\ha}{3}; \newcommand*{\va}{1.5};
        \node (aa) at (0,0) {$\hR_{X^{I/J}}$}; \node (ab) at (\ha,0) {$ N_{\O,X^{I \sqcup J}}$};
        \node (ba) at (0,-\va) {$\T_{X^{I/J}}$}; \node (bb) at (\ha,-\va) {$ N_{\K,X^{I/J}}$};
        \draw[->] (aa) to node[above] {$i_1$} (ab); \draw[->] (ba) to node[above] {$ j_1 $} (bb);
        \draw[->] (aa) to node[left] {$i_2 $} (ba); \draw[->] (ab) to node[left] {$ j_2$} (bb);
        \end{tikzpicture}
\end{equation}
Here $j_2$ is induced by the inclusion of $\hz_{I \sqcup J} \smallsetminus z_I$ into $\hz_{I \sqcup J}$. Moreover,
$$\T_{X^{I/J}} := G_{\K, X^{I/J}} \times_{G_{\O,X^{I \sqcup J}}} N_{\O,I \sqcup J}$$
where $G_{\O,X^{I \sqcup J}} \hookrightarrow G_{\K, X^{I/J}}$ again via the inclusion of $\hz_{I \sqcup J} \smallsetminus z_I$ into $\hz_{I \sqcup J}$. Notice that if $J = \varnothing$ this recovers $\hR_{X^I}$. 

As in (\ref{eq:Rsymmdiagram-fact}), we have an alternative description of $\hR_{X^{I/J}}/{\widehat{G}_{\O,X^{I \sqcup J}}}$ as the fiber product 
$$      
\begin{tikzpicture}[baseline=(current  bounding  box.center),thick,>=\arrtip]
\node (a) at (0,0) {$\hR_{X^{I/J}}/{\widehat{G}_{\O,X^{I \sqcup J}}}$};
\node (b) at (3.9,0) {$N_{\O,X^{I \sqcup J}}/{\widehat{G}_{\O,X^{I \sqcup J}}}$};
\node (c) at (0,-1.5) {$N_{\O,X^{I \sqcup J}}/{\widehat{G}_{\O,X^{I \sqcup J}}}$};
\node (d) at (3.9,-1.5) {$N_{\K,X^{I/J}}/{\widehat{G}_{\K,X^{I/J}}}$};
\draw[->] (a) to node[above] {$\pi_1$} (b);
\draw[->] (b) to node[right] {$ $} (d);
\draw[->] (a) to node[left] {$\pi_2$}(c);
\draw[->] (c) to node[above] {$ $} (d);
\end{tikzpicture}
$$
A convolution structure on $\Coh^{\widehat{G}_{\O,X^{I \sqcup J}}}(\hR_{X^{I/J}})$ can then be obtained as in Section \ref{sec:globalconv} or by applying the usual convolution construction on fiber products as in Section \ref{sec:convolution}. 

Next we will construct the following correspondence
\begin{equation}\label{eq:jetloopcorrs}
\begin{tikzpicture}[baseline=(current  bounding  box.center),thick,>=\arrtip]
\newcommand*{\ha}{3}; \newcommand*{\hb}{3}; \newcommand*{\va}{1.5};
        \node (aa) at (0,0) {$\hR_{X^{I/J}}$};
        \node (ab) at (\hb,0) {$\hR_{X^{I \sqcup J}}$};
        \node (ba) at (0,-\va) {$\hR_{X^I} \times X^J$};
        \draw[->] (aa) to node[above] {$i$} (ab);
        \draw[->] (aa) to node[left] {$q$} (ba);
\end{tikzpicture}
\end{equation}
To define $i$ consider the following commutative diagram 
\begin{equation}
	\begin{tikzpicture}[baseline=(current  bounding  box.center),thick,>=\arrtip]
		\newcommand*{\ha}{3}; 
		\newcommand*{\va}{-1.5};
		\node (a) at (0,0) {$\hR_{X^{I/J}}$};
		\node (b) at (\ha,0) {$N_{\O,X^{I \sqcup J}}$};
		\node (c) at (0,\va) {$\T_{X^{I/J}}$};
		\node (d) at (\ha,\va) {$N_{\K,X^{I/J}}$};
		\node (e) at (0,2*\va) {$\T_{X^{I \sqcup J}}$};
		\node (f) at (\ha,2*\va) {$N_{\K,X^{I \sqcup J}}$};
		\draw[->] (a) to node[above] {$i_1$} (b);
		\draw[->] (a) to node[left] {$i_2$} (c);
		\draw[->] (c) to node[above] {$j_1$} (d);
		\draw[->] (b) to node[right] {$j_2$} (d);
		\draw[->] (c) to node[left] {} (e);
		\draw[->] (d) to node[right] {} (f);
		\draw[->] (e) to node[above] {} (f);
	\end{tikzpicture}
\end{equation}
Commutativity of the outer rectangle induces a map $i: \hR_{X^{I/J}} \to \hR_{X^{I \sqcup J}}$. Moreover, in the commutative square 
\begin{equation*}
	\begin{tikzpicture}[baseline=(current  bounding  box.center),thick,>=\arrtip]
		\newcommand*{\ha}{3}; 
		\newcommand*{\va}{-1.5};
		\node (a) at (0,0) {$\hR_{X^{I/J}}$};
		\node (b) at (\ha,0) {$\hR_{X^{I \sqcup J}}$};
		\node (c) at (0,\va) {$\T_{X^{I/J}}$};
		\node (d) at (\ha,\va) {$\T_{X^{I \sqcup J}}$};
		\draw[->] (a) to node[above] {$i$} (b);
		\draw[->] (a) to node[left] {} (c);
		\draw[->] (c) to node[above] {} (d);
		\draw[->] (b) to node[right] {} (d);
	\end{tikzpicture}
\end{equation*}
the vertical maps and the bottom map are ind-proper and almost of ind-finite presentation since, by construction, they are all the base change of maps with these properties (the bottom map is the base change of $\Gr_{X^{I/J}} \to \Gr_{X^{I \sqcup J}}$). It follows that $i$ is also ind-proper and almost of ind-finite presentation.  

To define $q$ consider the following commutative cube
\begin{equation}\label{eq:cube}
	\begin{tikzpicture}[baseline=(current bounding box.center),thick,>=\arrtip]
		\newcommand*{\ha}{1.7}; \newcommand*{\hb}{1.7}; \newcommand*{\hc}{1.7};
		\newcommand*{\va}{-1.0}; \newcommand*{\vb}{-1.0}; \newcommand*{\vc}{-1.0}; 
		
		\node (ab) at (\ha,0) {$N_{\O,X^I} \times X^J$};
		\node (ad) at (\ha+\hb+\hc,0) {$N_{\O,X^{I \sqcup J}}$};
		\node (ba) at (0,\va) {$\hR_{X^I} \times X^J$};
		\node (bc) at (\ha+\hb,\va) {$\hR_{X^{I/J}}$};
		\node (cb) at (\ha,\va+\vb) {$N_{\K,X^I} \times X^J$};
		\node (cd) at (\ha+\hb+\hc,\va+\vb) {$N_{\K,X^{I/J}}$};
		\node (da) at (0,\va+\vb+\vc) {$\T_{X^I} \times X^J$};
		\node (dc) at (\ha+\hb,\va+\vb+\vc) {$\T_{X^{I/J}}$};
		
		\draw[->] (ad) -- (ab);
		\draw[->] (ba) -- (ab);
		\draw[<-] (cb) -- (ab);
		\draw[->] (bc) -- (ad);
		\draw[<-] (cd) -- (ad);
		\draw[<-] (da) -- (ba);
		\draw[->] (cd) -- (cb);
		\draw[->] (da) -- (cb);
		\draw[->] (dc) -- (cd);
		\draw[->] (dc) -- (da);
		
		\draw[-,line width=6pt,draw=white] (bc) -- (ba);
		\draw[dashed,->] (bc) -- (ba);
		\draw[-,line width=6pt,draw=white] (dc) -- (bc);
		\draw[<-] (dc) -- (bc);
	\end{tikzpicture}
\end{equation}
where the morphisms in the rear square are the obvious maps induced by inclusions such as $\hz_{I \sqcup J} \smallsetminus z_I \hookrightarrow \hz_{I \sqcup J}$. We claim that all the faces not involving the dotted map are Cartesian. 

The left and right squares are Cartesian by construction. To see that the rear square it is Cartesian it suffices to show that 
\begin{equation*}
	\begin{tikzpicture}[baseline=(current bounding box.center),thick,>=\arrtip]
		\newcommand*{\ha}{3}; \newcommand*{\va}{-1.5};
		
		\node (a) at (0,0) {$N_{\O,X^{I \sqcup J}}$};
		\node (b) at (\ha,0) {$N_{\O,X^I}$};
		\node (c) at (0,\va) {$N_{\K,X^{I/J}}^k$};
		\node (d) at (\ha,\va) {$N_{\K,X^I}^k$};
		
		\draw[->] (a) -- (b);
		\draw[->] (a) -- (c);
		\draw[->] (b) -- (d);
		\draw[->] (c) -- (d);
	\end{tikzpicture}
\end{equation*}
is Cartesian. In the notation of Lemma \ref{lem:flat} this square can be written as 
\begin{equation*}
	\begin{tikzpicture}[baseline=(current bounding box.center),thick,>=\arrtip]
		\newcommand*{\ha}{3}; \newcommand*{\va}{-1.5};
		
		\node (a) at (0,0) {$N_{\K,X^{I/J}}^{0,\infty/\infty}$};
		\node (b) at (\ha,0) {$N_{\K,X^{I/J}}^{0,\infty/0}$};
		\node (c) at (0,\va) {$N_{\K,X^{I/J}}^{k,\infty/\infty}$};
		\node (d) at (\ha,\va) {$N_{\K,X^{I/J}}^{k,\infty/0}$};
		
		\draw[->] (a) -- (b);
		\draw[->] (a) -- (c);
		\draw[->] (b) -- (d);
		\draw[->] (c) -- (d);
	\end{tikzpicture}
\end{equation*}
where the $\infty$s are understood to mean the obvious limits. The fact that this square is Cartesian then follows since the faces of (\ref{eq:Ncart}) are all Cartesian.

The bottom square is given by the following composition of smaller squares 
\begin{equation*}
	\begin{tikzpicture}[baseline=(current bounding box.center),thick,>=stealth]
		\newcommand*{\ha}{3.9}; \newcommand*{\hb}{5.2}; \newcommand*{\hc}{3.9}; \newcommand*{\va}{-1.5};
		
		\node (a) at (0,0) {$\T_{X^{I/J}}$};
		\node (b) at (\ha,0) {$G_{\K,X^{I/J}} \times_{G_{\O,X^{I \sqcup J}}} N_{\K,X^{I/J}}$};
		\node (c) at (\ha+\hb,0) {$\Gr_{X^I} \times_{X^I} N_{\K,X^{I/J}}$};
		\node (d) at (\ha+\hb+\hc,0) {$N_{\K,X^{I/J}}$};
		\node (e) at (0,\va) {$\T_{X^I} \times X^J$};
		\node (f) at (\ha,\va) {$G_{\K,X^I} \times_{G_{\O,X^I}} N_{\K,X^I} \times X^J$};
		\node (g) at (\ha+\hb,\va) {$\Gr_{X^I} \times_{X^I} N_{\K,X^I} \times X^J$};
		\node (h) at (\ha+\hb+\hc,\va) {$N_{\K,X^I} \times X^J$};
		
		\draw[->] (a) -- node[above] {} (b);
		\draw[->] (b) -- node[above] {$\sim$} (c);
		\draw[->] (c) -- node[above] {} (d);
		\draw[->] (a) -- node[left] {} (e);
		\draw[->] (b) -- node[left] {} (f);
		\draw[->] (c) -- node[left] {} (g);
		\draw[->] (d) -- node[left] {} (h);
		\draw[->] (e) -- node[above] {} (f);
		\draw[->] (f) -- node[above] {$\sim$} (g);
		\draw[->] (g) -- node[above] {} (h);
	\end{tikzpicture}
\end{equation*}

where the first top (resp. bottom) map is induced by the inclusion of $N_{\O,X^{I \sqcup J}}$ into $N_{\K,X^{I/J}}$ (resp. $N_{\O,X^I}$ into $N_{\K,X^I}$) and the last top and bottom maps are projections. Each small square is clearly Cartesian which means that the total rectangle is also Cartesian. The dotted map in (\ref{eq:cube}), which defines $q$, is then induced by the commutativity of the other other squares. 

Associated to $i$ and $q$ we have the following correspondence of quotient stacks 
\begin{equation}\label{eq:jetloopcorrs2}
\begin{tikzpicture}[baseline=(current  bounding  box.center),thick,>=\arrtip]
\newcommand*{\ha}{3}; \newcommand*{\hb}{4.0}; \newcommand*{\va}{1.6};
\node[matrix] at (0,0) {
        \node (aa) at (0,0) {$\hR_{X^{I/J}}/{\widehat{G}_{\O,X^{I \sqcup J}}}$};
        \node (ab) at (\hb,0) {$\hR_{X^{I \sqcup J}}/{\widehat{G}_{\O,X^{I \sqcup J}}}$};
        \node (ba) at (0,-\va) {$\hR_{X^I}/\hGOX \times X^J$};
        \draw[->] (aa) to node[above] {$i$} (ab);
        \draw[->] (aa) to node[left] {$q$} (ba);\\
};
\end{tikzpicture}
\end{equation}
It is this correspondence that we will use to define our functors. Moreover, as in previous cases, this correspondence has a more formal description in terms of fiber products. More precisely, given a composition $Y \to Y' \to Z$ of stacks one has an induced morphism $Y \times_Z Y \to Y' \times_Z Y'$. If we take the composition 
$$N_{\O,X^{I \sqcup J}}/\widehat{G}_{\O,X^{I \sqcup J}} \to N_{\K,X^{I/J}}/\widehat{G}_{\K,X^{I/J}} \to N_{\K,X^{I \sqcup J}}/\widehat{G}_{\K,X^{I \sqcup J}}$$
then one can show, as in the proof of Proposition \ref{prop:symmetricRdef}, that the induced map is $i$ from (\ref{eq:jetloopcorrs2}). 

On the other hand, given a Cartesian square of stacks 
\begin{equation*}
	\begin{tikzpicture}[baseline=(current bounding box.center),thick,>=stealth]
		\newcommand*{\ha}{3}; \newcommand*{\va}{-1.5};
		
		\node (a) at (0,0) {$Y$};
		\node (b) at (\ha,0) {$Y'$};
		\node (c) at (0,\va) {$Z$};
		\node (d) at (\ha,\va) {$Z'$};
		
		\draw[->] (a) -- (b);
		\draw[->] (a) -- (c);
		\draw[->] (b) -- (d);
		\draw[->] (c) -- (d);
	\end{tikzpicture}
\end{equation*}
one obtains a natural map $Y \times_Z Y \to Y' \times_{Z'} Y'$. If one applies this construction to the Cartesian square (\ref{eq:fact-cart}) below then this recovers the map $q$ from (\ref{eq:jetloopcorrs2}). 

\begin{Lemma}\label{lem:cartesian}
The following natural inclusions 
\begin{equation*}
	\begin{tikzpicture}[baseline=(current  bounding  box.center),thick,>=\arrtip]
		\newcommand*{\ha}{3}; \newcommand*{\va}{-1.5};
		\node (a) at (0,0) {$\hz_{I \sqcup J}$};
		\node (b) at (\ha,0) {$\hz_{I}$};
		\node (c) at (0,\va) {$\hz_{I \sqcup J} \smallsetminus z_I$};
		\node (d) at (\ha,\va) {$\hz_I \smallsetminus z_I$};
		\draw[->] (b) to node[above] {} (a);
		\draw[->] (d) to node[above] {} (b);
		\draw[->] (d) to node[left] {} (c);
		\draw[->] (c) to node[left] {} (a);
	\end{tikzpicture}
\end{equation*}
induce a Cartesian square 
\begin{equation}\label{eq:fact-cart}
	\begin{tikzpicture}[baseline=(current  bounding  box.center),thick,>=\arrtip]
		\newcommand*{\ha}{3.5}; \newcommand*{\va}{-1.5};
		\node (a) at (0,0) {$N_{\O,X^{I \sqcup J}}/\widehat{G}_{\O,X^{I \sqcup J}}$};
		\node (b) at (\ha,0) {$N_{\O,X^I}/\widehat{G}_{\O,X^I}$};
		\node (c) at (0,\va) {$N_{\K,X^{I/J}}/\widehat{G}_{\K,X^{I/J}}$};
		\node (d) at (\ha,\va) {$N_{\K,X^I}/\widehat{G}_{\K,X^I}$};
		\draw[->] (a) to node[above] {} (b);
		\draw[->] (a) to node[left] {} (c);
		\draw[->] (c) to node[above] {} (d);
		\draw[->] (b) to node[right] {} (d);
	\end{tikzpicture}
\end{equation}
\end{Lemma}
\begin{proof}
We can factor (\ref{eq:fact-cart}) as 
\begin{equation}\label{eq:fact-cart2}
	\begin{tikzpicture}[baseline=(current  bounding  box.center),thick,>=\arrtip]
		\newcommand*{\ha}{3.5}; \newcommand*{\va}{-1.5};
		\node (a) at (0,0) {$N_{\O,X^{I \sqcup J}}/\widehat{G}_{\O,X^{I \sqcup J}}$};
		\node (b) at (\ha,0) {$N_{\O,X^I}/\widehat{G}_{\O,X^I}$};
		\node (c) at (0,\va) {$N_{\K,X^{I/J}}/\widehat{G}_{\O,X^{I \sqcup J}}$};
		\node (d) at (\ha,\va) {$N_{\K,X^I}/\widehat{G}_{\O,X^I}$};
		\node (e) at (0,2*\va) {$N_{\K,X^{I/J}}/\widehat{G}_{\K,X^{I/J}}$};
		\node (f) at (\ha,2*\va) {$N_{\K,X^I}/\widehat{G}_{\K,X^I}$};
		\draw[->] (a) to node[above] {} (b);
		\draw[->] (a) to node[left] {} (c);
		\draw[->] (c) to node[above] {} (d);
		\draw[->] (b) to node[right] {} (d);
		\draw[->] (c) to node[left] {} (e);
		\draw[->] (d) to node[right] {} (f);
		\draw[->] (e) to node[above] {} (f);
	\end{tikzpicture}
\end{equation}
The top square is Cartesian because the rear square in (\ref{eq:cube}) is Cartesian. To see that the bottom square is Cartesian it suffices to show that 
\begin{equation}\label{eq:fact-cart3}
	\begin{tikzpicture}[baseline=(current  bounding  box.center),thick,>=\arrtip]
		\newcommand*{\ha}{3}; \newcommand*{\va}{-1.5};
		\node (a) at (0,0) {$pt/\widehat{G}_{\O,X^{I \sqcup J}}$};
		\node (b) at (\ha,0) {$pt/\widehat{G}_{\O,X^I}$};
		\node (c) at (0,\va) {$pt/\widehat{G}_{\K,X^{I/J}}$};
		\node (d) at (\ha,\va) {$pt/\widehat{G}_{\K,X^I}$};
		\draw[->] (a) to node[above] {} (b);
		\draw[->] (a) to node[left] {} (c);
		\draw[->] (c) to node[above] {} (d);
		\draw[->] (b) to node[right] {} (d);
	\end{tikzpicture}
\end{equation}
is Cartesian. Now consider the following commutative square
\begin{equation*}
\begin{tikzpicture}[baseline=(current  bounding  box.center),thick,>=\arrtip]
	\newcommand*{\ha}{3}; \newcommand*{\va}{-1.5};
	\node (a) at (0,0) {$\widehat{G}_{\O,X^{I \sqcup J}}$};
	\node (b) at (\ha,0) {$\widehat{G}_{\O,X^I}$};
	\node (c) at (0,\va) {$\widehat{G}_{\K,X^{I/J}}$};
	\node (d) at (\ha,\va) {$\widehat{G}_{\K,X^I}$};
	\draw[->] (a) to node[above] {$\alpha_1$} (b);
	\draw[->] (a) to node[left] {$\beta_1$} (c);
	\draw[->] (c) to node[above] {$\alpha_2$} (d);
	\draw[->] (b) to node[right] {$\beta_2$} (d);
\end{tikzpicture}
\end{equation*}
where we view everything as group schemes over $X^{I \sqcup J}$. Notice that the $\alpha_i$ are surjective and the $\beta_i$ injective. One can check that $\ker(\beta_1) \cong \ker(\beta_2)$ and that the quotients $\widehat{G}_{\K,X^I}/\widehat{G}_{\O,X^I}$ and $\widehat{G}_{\K,X^{I/J}}/\widehat{G}_{\O,X^{I \sqcup J}}$ are both isomorphic to $\Gr_{X^I}\times X^J$. From this one can conclude that 
$$\widehat{G}_{\K,X^I}/\widehat{G}_{\O,X^{I \sqcup J}} \cong \widehat{G}_{\K,X^I}/\widehat{G}_{\K,X^{I/J}} \times_{X^{I \sqcup J}} \widehat{G}_{\K,X^I}/\widehat{G}_{\O,X^I}.$$

In other words, one has a Cartesian square 
\begin{equation*}
	\begin{tikzpicture}[baseline=(current  bounding  box.center),thick,>=\arrtip]
		\newcommand*{\ha}{3.5}; \newcommand*{\va}{-1.5};
		\node (a) at (0,0) {$\widehat{G}_{\K,X^I}/\widehat{G}_{\O,X^{I \sqcup J}}$};
		\node (b) at (\ha,0) {$\widehat{G}_{\K,X^I}/\widehat{G}_{\O,X^I}$};
		\node (c) at (0,\va) {$\widehat{G}_{\K,X^I}/\widehat{G}_{\K,X^{I/J}}$};
		\node (d) at (\ha,\va) {$X^{I \sqcup J}$};
		\draw[->] (a) to node[above] {} (b);
		\draw[->] (a) to node[left] {} (c);
		\draw[->] (c) to node[above] {} (d);
		\draw[->] (b) to node[right] {} (d);
	\end{tikzpicture}
\end{equation*}
Taking the quotient by $\widehat{G}_{\K,X^I}$ gives us (\ref{eq:fact-cart3}). 
\end{proof}

We now arrive at the main result of this subsection. 

\begin{Proposition}\label{prop:unital}
In the notation of (\ref{eq:jetloopcorrs2}) the functors $i_*$ and $q^*$ preserve coherent sheaves. Moreover, their composition induces a monoidal functor 
\begin{align*}
\eta_I^{I \sqcup J}:  \Coh^{\hGOX}(\hR_{X^I} \times X^J) & \to \Coh^{{\widehat{G}_{\O,X^{I \sqcup J}}}}(\hR_{X^{I \sqcup J}}) \\
(-) & \mapsto i_* q^* (-)
\end{align*}
\end{Proposition}
\begin{Remark} 
There is an analogous functor
$$\eta_J^{I \sqcup J}: \Coh^{\widehat{G}_{\O,X^J}}(X^I \times \hR_{X^J}) \to \Coh^{\widehat{G}_{\O,X^{I \sqcup J}}}(\hR_{X^{I \sqcup J}})$$
which is also monoidal.
\end{Remark}
\begin{proof}
The functor $i_*$ (resp. $q^*$) preserves coherent sheaves and is monoidal as a consequence of the discussion above and Lemma \ref{lem:fact1} (resp. Lemma \ref{lem:fact2}). The result follows. 
\end{proof}

\begin{Lemma}\label{lem:fact1}
Consider a Cartesian square of stacks 
\begin{equation*}
	\begin{tikzpicture}[baseline=(current  bounding  box.center),thick,>=\arrtip]
		\newcommand*{\ha}{3}; \newcommand*{\va}{-1.5};
		\node (a) at (0,0) {$Y$};
		\node (b) at (\ha,0) {$Y'$};
		\node (c) at (0,\va) {$Z$};
		\node (d) at (\ha,\va) {$Z'$};
		\draw[->] (a) to node[above] {} (b);
		\draw[->] (a) to node[left] {} (c);
		\draw[->] (c) to node[above] {} (d);
		\draw[->] (b) to node[right] {} (d);
	\end{tikzpicture}
\end{equation*}
where $X := Y \times_Z Y$ and $X' := Y' \times_{Z'} Y'$ satisfy the hypotheses (\ref{eq:convhypos}). If  the map $Y \to Y'$ has \stable coherent pullback then so does the induced map $q: X \to X'$. Moreover, the map  $q^*: \Coh(X') \to \Coh(X)$ is monoidal. 
\end{Lemma}
\begin{proof}
Consider the following cubic diagram.
\begin{equation}\label{eq:cubeyz}
	\begin{tikzpicture}[baseline=(current  bounding  box.center),thick,>=\arrtip]
		\newcommand*{\ha}{1.5}; \newcommand*{\hb}{1.5}; \newcommand*{\hc}{1.5};
		\newcommand*{\va}{-.9}; \newcommand*{\vb}{-.9}; \newcommand*{\vc}{-.9}; 
		\node (ab) at (\ha,0) {$Y'$};
		\node (ad) at (\ha+\hb+\hc,0) {$X'$};
		\node (ba) at (0,\va) {$Y$};
		\node (bc) at (\ha+\hb,\va) {$X$};
		\node (cb) at (\ha,\va+\vb) {$Z'$};
		\node (cd) at (\ha+\hb+\hc,\va+\vb) {$Y'$};
		\node (da) at (0,\va+\vb+\vc) {$Z$};
		\node (dc) at (\ha+\hb,\va+\vb+\vc) {$Y$};
		\draw[<-] (ab) to node[above] {} (ad);
		\draw[<-] (ab) to node[above] {} (ba);
		\draw[->] (ab) to node[left,pos=.8] {} (cb);
		\draw[<-] (ad) to node[above] {} (bc);
		\draw[->] (ad) to node[right] {} (cd);
		\draw[->] (ba) to node[above] {} (da);
		\draw[<-] (cb) to node[above,pos=.2] {} (cd);
		\draw[<-] (cb) to node[above] {} (da);
		\draw[<-] (cd) to node[above] {} (dc);
		\draw[<-] (da) to node[above,pos=.6] {} (dc);
		
		\draw[-,line width=6pt,draw=white] (ba) to  (bc);
		\draw[<-] (ba) to node[above,pos=.75] {} (bc);
		\draw[-,line width=6pt,draw=white] (bc) to  (dc);
		\draw[->] (bc) to node[right,pos=.2] {} (dc);
	\end{tikzpicture}
\end{equation}
The front and back faces are Cartesian by construction and the bottom face is Cartesian by assumption. It follows that the top face is also Cartesian. This gives us the left Cartesian squares below where $i \in \{1,2\}$. Repeating the argument also gives us the right Cartesian squares below. 
\begin{equation*}
	\begin{tikzpicture}[baseline=(current  bounding  box.center),thick,>=\arrtip]
		\newcommand*{\ha}{3}; \newcommand*{\hb}{3}; \newcommand*{\va}{-1.5};
		\node[matrix] at (0,0) {
			\node (aa) at (0,0) {$X$};
			\node (ab) at (\ha,0) {$X'$};
			\node (ba) at (0,\va) {$Y$};
			\node (bb) at (\ha,\va) {$Y'$};
			\draw[->] (aa) to node[above] {$q$} (ab); 
			\draw[->] (aa) to node[left] {$\pi_i$} (ba);
			\draw[->] (ba) to node[below] {} (bb);
			\draw[->] (ab) to node[right] {$\pi_i$} (bb); \\
		};
		\node[matrix] at (5.5,0) {
			\node (aa) at (0,0) {$X \times_Y X$};
			\node (ab) at (\hb,0) {$X' \times_{Y'} X'$};
			\node (ba) at (0,\va) {$X$};
			\node (bb) at (\hb,\va) {$X'$};
			\draw[->] (aa) to node[above] {$q'$} (ab); 
			\draw[->] (aa) to node[left] {$\pi_i$} (ba);
			\draw[->] (ba) to node[above] {$q$} (bb);
			\draw[->] (ab) to node[right] {$\pi_i$} (bb); \\
		};
	\end{tikzpicture}
\end{equation*}

Since the maps $\pi_i$ are almost ind-finitely presented the left squares above imply that $q$ has \stable coherent pullback . 

Now consider the following commutative diagram.
\begin{equation*}\label{eq:xy}
	\begin{tikzpicture}[baseline=(current bounding box.center),thick,>=\arrtip]
		\newcommand*{\ha}{3}; \newcommand*{\hb}{3}; \newcommand*{\va}{-1.5};
		
		\node (a) at (0,0) {$X \times X$};
		\node (a') at (\ha,0) {$X \times_Y X$};
		\node (a'') at (\ha+\hb,0) {$X$};
		\node (a''') at (2*\ha+\hb,0) {$Y$};
		\node (b) at (0,\va) {$X' \times X'$};
		\node (b') at (\ha,\va) {$X' \times_{Y'} X'$};
		\node (b'') at (\ha+\hb,\va) {$X'$};
		\node (b''') at (2*\ha+\hb,\va) {$Y'$};
		
		\draw[->] (a') to node[above] {$\Delta$} (a);
		\draw[->] (a') to node[above] {$m$} (a'');
		\draw[->] (a'') to node[above] {$\pi_1$} (a''');
		\draw[->] (b') to node[above] {$\Delta'$} (b);
		\draw[->] (b') to node[above] {$m'$} (b'');
		\draw[->] (b'') to node[above] {$\pi_1$} (b''');
		
		\draw[->] (a) to node[left] {$q \times q$} (b);
		\draw[->] (a') to node[right] {$q'$} (b');
		\draw[->] (a'') to node[right] {$q$} (b'');
		\draw[->] (a''') to node[right] {} (b''');
	\end{tikzpicture}
\end{equation*}
The composition of the two rightmost squares can be rewritten as 
\begin{equation*}
	\begin{tikzpicture}[baseline=(current bounding box.center),thick,>=\arrtip]
		\newcommand*{\ha}{3}; \newcommand*{\hb}{3}; \newcommand*{\va}{-1.5};
		
		\node (a) at (0,0) {$X \times_Y X$};
		\node (a') at (\ha,0) {$X$};
		\node (a'') at (\ha+\hb,0) {$Y$};
		\node (b) at (0,\va) {$X' \times_{Y'} X'$};
		\node (b') at (\ha,\va) {$X'$};
		\node (b'') at (\ha+\hb,\va) {$Y'$,};
		
		\draw[->] (a) to node[above] {$\pi_1$} (a');
		\draw[->] (a') to node[above] {$\pi_1$} (a'');
		\draw[->] (b) to node[above] {$\pi_1$} (b');
		\draw[->] (b') to node[above] {$\pi_1$} (b'');
		
		\draw[->] (a) to node[left] {$q'$} (b);
		\draw[->] (a') to node[right] {$q$} (b');
		\draw[->] (a'') to node[right] {} (b'');
	\end{tikzpicture}
\end{equation*}
where both small squares above are Cartesian. It follows that the total square is Cartesian and subsequently so is the middle square in (\ref{eq:xy}). Since $m$ and $m'$ are ind-proper and almost ind-finitely presented it follows that this middle square satisfies base change. This means that for $\cF,\cG \in \Coh(X')$ we have 
\begin{align*}
q^*(\cF * \cG) 
&\cong q^* m'_* \Delta'^* (\cF \boxtimes \cG) \\
&\cong m_* q'^* \Delta'^* (\cF \boxtimes \cG) \\
&\cong m_* \Delta^* (q \times q)^* (\cF \boxtimes \cG) \cong q^*(\cF) * q^* (\cG).\qedhere
\end{align*}
\end{proof}

\begin{Lemma}\label{lem:fact2}
Consider a composition $Y \to Z \to Z'$ where $X := Y \times_Z Y$ and $X' := Y \times_{Z'} Y$ satisfy the hypotheses (\ref{eq:convhypos}). If the induced morphism $i: X \to X'$ is ind-proper, almost ind-finitely presented, then $i_*: \Coh(X) \to \Coh(X')$ is monoidal. 
\end{Lemma}
\begin{proof}
Consider the following commutative diagram
\begin{equation*}
	\begin{tikzpicture}[baseline=(current bounding box.center),thick,>=\arrtip]
		\newcommand*{\ha}{3}; \newcommand*{\hb}{3}; \newcommand*{\va}{-1.5};
		
		\node (a) at (0,0) {$X \times X$};
		\node (a') at (\ha,0) {$X \times_Y X$};
		\node (a'') at (\ha+\hb,0) {$X$};
		\node (b) at (0,\va) {$X' \times X'$};
		\node (b') at (\ha,\va) {$X' \times_{Y} X'$};
		\node (b'') at (\ha+\hb,\va) {$X'$};
		
		\draw[->] (a') to node[above] {$\Delta$} (a);
		\draw[->] (a') to node[above] {$m$} (a'');
		\draw[->] (b') to node[above] {$\Delta'$} (b);
		\draw[->] (b') to node[above] {$m'$} (b'');
		
		\draw[->] (a) to node[left] {$i \times i$} (b);
		\draw[->] (a') to node[right] {$i'$} (b');
		\draw[->] (a'') to node[right] {$i$} (b'');
	\end{tikzpicture}
\end{equation*}
Note that all stacks in this diagram are ind-tamely presented and coherent. The left square is Cartesian because both $\Delta$ and $\Delta'$ are obtained by base change from $\Delta_Y: Y \to Y \times Y$. This base change also implies that $\Delta$ and $\Delta'$ have \stable coherent pullback. Then for $\cF,\cG \in \Coh(X)$ we have
\begin{align*}
i_* (\cF * \cG) 
&\cong i_* m_* \Delta^* (\cF \boxtimes \cG) \cong m'_* i'_* \Delta^* (\cF \boxtimes \cG) \\
&\cong m'_* \Delta'^* (i \times i)_* (\cF \boxtimes \cG) \cong m'_* \Delta'^* (i_* \cF \boxtimes i_* \cG) \\
&\cong (i_* \cF) * (i_* \cG).
\end{align*}
where the third isomorphism uses the fact that $i'$ and $i \times i$ are ind-proper and almost ind-finitely presented and that $\Delta$ and $\Delta'$ have \stable coherent pullback. 
\end{proof}

Finally, putting everything together we obtain the main result of this section. 

\begin{proof}[Proof of Theorem \ref{thm:rmatrix}]
From \cite[Thm. 5.10]{CW1} we obtain a system of renormalized $r$-matrices assuming we have a factorization structure satisfying conditions (1)-(4) from \cite[Sec. 5.2]{CW1}. Condition (1) follows from the evident compatibility between the global convolution diagram in Section \ref{sec:globalconv} with the factorization structure. Condition (2) is a consequence of the fact that $\hR_{X} \cong \hR \times X$. Condition (3) is a special case of Proposition ~\ref{prop:unital}. Finally, condition (4) is a consequence of the discussion in Section \ref{sec:globalconv}. 
\end{proof} 

\section{Examples}\label{sec:GLn}

We conlude by studying the case $(G,N) = (GL_n, \C^n)$. The space $\hR_{GL_n, \C^n}$ admits a lattice model extending that of $\Gr_{GL_n}$, so this example lets us to illustrate more concretely the abstract theory developed in earlier sections. Moreover, we can directly establish the existence of the anticipated cluster structure when $n=2$. In general, the category $\KPGLn := \KP{GL_n}{\C^n}^\eta$ is expected to categorify a quantum cluster algebra associated to the following quiver:
\begin{equation}\label{fig:cluster}
\begin{tikzpicture}[baseline=(current  bounding  box.center),thick,>=\arrtip]
\newcommand*{\Ddotsdist}{2}
\newcommand*{\shft}{1}
\newcommand*{\DrawDots}[1]{
  \fill ($(#1) + .25*(\Ddotsdist,0)$) circle (.03);
  \fill ($(#1) + .5*(\Ddotsdist,0)$) circle (.03);
  \fill ($(#1) + .75*(\Ddotsdist,0)$) circle (.03);
}
\node [matrix] (Q) at (0,0)
{
\coordinate [label={below:$\cP_{1,0}$}] (2) at (0,0);
\coordinate [label={below:$\cP_{2,0}$}] (4) at (2,0);
\coordinate (6) at (4,0);
\coordinate (2n-4) at (6-\shft,0);
\coordinate [label={below:$\cP_{n-1,0}$}] (2n-2) at (8-\shft,0);
\coordinate [label={below:$\cP_{n,0}$}] (2n) at (10-\shft,0);

\coordinate [label={$\cP_{1,1}$}] (1) at (0,2);
\coordinate [label={$\cP_{2,1}$}] (3) at (2,2);
\coordinate (5) at (4,2);
\coordinate (2n-5) at (6-\shft,2);
\coordinate [label={$\cP_{n-1,1}$}] (2n-3) at (8-\shft,2);
\coordinate [label={$\cP_{0,\det}$}] (2n-1) at (10-\shft,2);

\foreach \v in {1,2,3,4,2n-3,2n-2,2n} {\fill (\v) circle (.06);};
\foreach \s/\t in {1/2,3/4,2n-3/2n-2} {
  \draw [->,shorten <=1.7mm,shorten >=1.7mm] ($(\s)+(0.06,0)$) to ($(\t)+(0.06,0)$);
  \draw [->,shorten <=1.7mm,shorten >=1.7mm] ($(\s)-(0.06,0)$) to ($(\t)-(0.06,0)$);
};

\newlength\squareEdgeLengthb
\setlength\squareEdgeLengthb{0.12cm} 

\draw (2n-1) ++(-0.5\squareEdgeLengthb,-0.5\squareEdgeLengthb) rectangle ++(\squareEdgeLengthb,\squareEdgeLengthb);

\draw [->,shorten <=1.7mm,shorten >=1.7mm] ($(2n-1)$) to ($(2n)$);
\draw [->,shorten <=1.7mm,shorten >=1.7mm] ($(2n-2)$) to ($(2n)$);

\foreach \s/\t in {2/3,4/5,2n-4/2n-3,2n-2/2n-1} {
  \draw [->,shorten <=1.7mm,shorten >=1.7mm] ($(\s)$) to ($(\t)$);
};
\foreach \s/\t in {4/1,6/3,2n-2/2n-5,2n/2n-3} {
  \draw [shorten <=1.7mm,shorten >=1.7mm] ($(\s)$) to ($(\s)!.5!(\t)$);
  \draw [->,shorten <=1.7mm,shorten >=1.7mm] ($(\s)!.5!(\t)$) to ($(\t)$);
};
\DrawDots{4-\shft*.5,1};\\
};
\end{tikzpicture}
\end{equation}
Here the square vertex is frozen, the labels indicate corresponding simple objects in $\KPGLn$, and we note that when $n=1$ this specializes to the quiver (\ref{eq:introquiver}) from the introduction. 

Below we construct some of the expected exchange relations, and in the $n=2$ case these are enough to prove that we indeed have a monoidal cluster categorification (Theorem \ref{thm:qcatthminthesec}). We also connect duals in $\KPGLtwo$ to the twist automorphism (alias DT transformation, monodromy operator, or spectrum generator). Finally, we emphasize that these results are consistent with the calculation of the BPS quiver of the associated gauge theory \cite{ACCERV14}, providing a key consistency check that the main construction of this paper is physically correct.

\subsection{The lattice model}

Throughout Section \ref{sec:GLn} we adopt the standing convention that $\Gr_{GL_n}$ refers to the reduced locus of the affine Grassmannian, likewise for symbols denoting Schubert and convolution varieties. As we only perform computations involving sheaves supported on the reduced locus, this will be sufficient for our purposes while reducing notational overhead.

First recall the lattice model for the affine Grassmannian of $GL_n$. We have an identification
\begin{align*}
\Gr &\congto \{ \cO\text{-lattices}\: L \subset \N_\cK \}\\
[g] &\:\mapsto\: gL_0
\end{align*}
where $N := \C^n$, and where $L_0$ denotes the standard lattice $N_\cO \subset N_\cK$ (we will write $L_0$ and $N_\O$ interchangeably). 

Setting $\Gr^k := \Gr_{\omega^\vee_k}$ for $1 \leq k \leq n$, we have an induced identification
$$\Gr^k  = \{L \overset{k}\subset L_0: tL_0 \subset L\},$$
where $L \overset{k}\subset L_0$ indicates that $\dim(L_0/L)=k$.
This space is isomorphic to the finite-dimensional Grassmannian of $k$-dimensional subspaces of $\C^n \cong L_0/tL_0$. 

The convolution spaces of such varieties can similarly be described as
\begin{align*}
\Gr^{k_1} \ttimes \dots \ttimes \Gr^{k_m} &\congto \{L_m \overset{k_m}\subset \cdots \overset{k_2}\subset L_1 \overset{k_1}\subset L_0: tL_{i-1} \subset L_i \}\\
[g_1,\dotsc,g_m] &\:\mapsto\: g_1 \cdots g_m L_0 \subset \cdots \subset g_1 L_0 \subset L_0.
\end{align*}
The multiplication map $m$ to $\Gr$ is then given by the forgetful map
\begin{align*}
m: \Gr^{k_1} \ttimes \dots \ttimes \Gr^{k_m} & \rightarrow \Gr \\
(L_m \subset \dots \subset L_1 \subset L_0) &\mapsto (L_m \subset L_0).
\end{align*}
We also have a lattice description of the anti-fundamental Schubert varieties as
$$\Gr_{-w_0 \omega^\vee_k} = \{L_0 \overset{k}\subset L: tL \subset L_0\}.$$
We denote this variety by $\Gr^{-k}$ for short.

The space $\hR := \hR_{GL_n,\C^n}$, or more precisely its reduced classical locus $\hR^{\cl}$, inherits a lattice description from $\Gr$. Specifically, its defining diagram as an intersection of bundles translates directly into the lattice model as follows.
\begin{equation}\label{eq:latticeintersectiondiagram}
\begin{tikzpicture}
[baseline=(current  bounding  box.center),thick,>=\arrtip]
\node (a) at (0,0) {$\hR^{\cl} \cong \{L \subset N_\cK, s \in L \cap L_0\}$};
\node (b) at (6.5,0) {$\Gr \times \N_\cO \cong \{L \subset N_\cK, s \in L_0\}$};
\node (c) at (0,-1.5) {$\T \cong \{L \subset N_\cK, s \in L\}$};
\node (d) at (6.5,-1.5) {$\Gr \times \N_\cK \cong \{L \subset N_\cK, s \in N_\cK\}$};
\draw[right hook->] (a) to node[above] {$$} (b);
\draw[right hook->] (b) to node[right] {$$} (d);
\draw[right hook->] (a) to node[left] {$$}(c);
\draw[right hook->] (c) to node[above] {$$} (d);
\end{tikzpicture}
\end{equation}
The derived structure on $\hR$ is determined by the Koszul resolution of $\N_\O \subset N_\cK$. In analogy with $\Gr^k$ we will write $\hR_k := \hR_{\omega_k^\vee}$ and $\T_k := \T_{\omega_k^\vee}$, similarly for $\hR_{-k}$ and $\T_{-k}$. From (\ref{eq:latticeintersectiondiagram}) we then obtain  identifications
\begin{equation*}
\hR_k^{\cl} = \{s \in L \overset{k}\subset L_0: tL_0 \subset L\} \quad \hR_{-k}^{\cl} = \{s \in L_0 \overset{k} \subset L: tL \subset L_0 \}.
\end{equation*}

\subsection{Koszul-perverse coherent sheaves} We now describe the most basic examples of simple Koszul-perverse sheaves on $\hR_{GL_n, \C^n}$. From now on we write $\hGO$ for $GL_{n,\O} \rtimes \C^\times$ and $\eta: \C^\times \to GL_n$ for the inclusion of the center. Since $\eta$ acts with weight one on $\C^n$, we may consider the category $\KPGLn := \cK \cP^\eta_{GL_n, \C^n}$. In order to simplify expressions, we abuse notation by writing $\la 1 \ra := [1]\{-1\}$ (this convention agrees with \cite{CW1}). Since there is no one-dimensional $\hGO$-representation on which $\eta$ acts with weight one, there is no Koszul shift in $\KPGLn$ and hence no ambiguity. 

Recall the following notation for tautological bundles on $\Gr$. We write 
$$ \cO_{\Gr^k} \otimes (L/tL_0) \in \Coh^\hGO(\Gr^k) $$
for the sheaf of sections of the bundle over $\Gr^k$ whose fiber over $[L] \in \Gr^k$ is $L/tL_0$, notating other tautological bundles on $\Gr$ and $\hR$ similarly.

Given $k \in [1,n]$ we write $i_k: \hR_k^{\cl} \into \hR$ for the natural embedding, similarly for $i_{-k}: \hR_{-k}^{\cl} \into \hR$. Given $\ell \in \Z$ we then define
\begin{equation}\label{eq:Pkldef}
\begin{gathered}
\cP_{k,\ell} := i_{k*} (\O_{\hR_k^{\cl}} \otimes \det(L_0/L)^\ell) \la \tfrac12 \dim \Gr^k \ra [-k\ell] \\
\cP_{-k,\ell} := i_{-k*} (\O_{\hR_{-k}^{\cl}} \otimes \det(L/L_0)^\ell) \la \tfrac12 \dim \Gr^{-k} \ra [-k\ell]\\
\cP_{0,\det} := i_{0*} (\O_{\hR^{\cl}_0} \otimes \det(L_0/tL_0)) \la -n \ra.
\end{gathered}
\end{equation}
For $\ell=0$ we also abbreviate $\cP_k := \cP_{k,\ell}$ and $\cP_{-k,\ell} := \cP_{-k,\ell}$. 

By abuse we also use the notation (\ref{eq:Pkldef}) for the corresponding sheaves on $\Gr$, i.e. the sheaves defined by the same formula but writing $i_{\pm k}$ for the embedding $\Gr^{\pm k} \into \Gr$. This notation was used in \cite[Section 2.2]{CW1} except for the presence of the shift $\la -k \ell \ra$. Without the shift these are simple objects in the perverse heart of $\Coh^\hGO(\Gr)$, hence after applying $\la -k \ell \ra$ they are simple objects in its Koszul-perverse heart. It then follows from Theorem \ref{thm:koszul-perverse} that the sheaves (\ref{eq:Pkldef}) are simple objects of $\KPGLn$. 

The t-exactness of the functors $\sigma_1^* i_{1*}$ and $\sigma_2^* i_{2*}$ is now illustrated concretely by the following calculation, noting that by definition tensoring with $\Lambda^i(L_0/L)^\vee [i]$ is t-exact for the Koszul-perverse t-structure on $\Coh^\hGO(\Gr)$. 

\begin{Proposition}\label{prop:extheta}
We have isomorphisms 
\begin{align*}
& \sigma_1^* i_{1*} (\cP_{-k,\ell}) \cong \cP_{-k,\ell}, \ \ \text{ } \ \   \sigma_1^* i_{1*}(\cP_{k,\ell}) \cong \cP_{k,\ell} \otimes \bigoplus_{i \ge 0} \Lambda^i(L_0/L)^\vee [i], \\
& \sigma_2^* i_{2*} (\cP_{k,\ell}) \cong \cP_{k,\ell}, \ \ \text{ } \ \  \sigma_2^* i_{2*}(\cP_{-k,\ell}) \cong \cP_{-k,\ell} \otimes \bigoplus_{i \ge 0} \Lambda^i(L/L_0)^\vee [i].
\end{align*}
\end{Proposition}
\begin{proof}
We compute $\sigma_2^* i_{2*}$ explicitly, the computation of $\sigma_1^* i_{1*}$ being similar. In the first case we have  $\hR_k^{\cl} \cong \T_k$ and so 
$$\sigma_2^* i_{2*}(\cP_{k,\ell}) \cong \sigma_2^*(\O_{\T_k} \otimes \pi_{\Gr^k}^* \det(L_0/L)^\ell) \la \tfrac12 \dim \Gr^k \ra [-k\ell] \cong \cP_{k,\ell}.$$
In the second case $\hR_{-k}^{\cl} \cong \{s \in L_0 \subset L\} \subset \{L_0 \subset L \ni s\} \cong \T_{-k}$., hence
\begin{align*}
\sigma_2^* i_{2*}(\cP_{-k,\ell}) 
&\cong \sigma_2^*(\O_{\hR_{-k}^{\cl}} \otimes \pi_{\Gr^{-k}}^* \det(L/L_0)^\ell) \la \tfrac12 \dim \Gr^{-k} \ra [-k\ell] \\
&\cong \bigoplus_{i \ge 0} \O_{\Gr^{-k}} \otimes \Lambda^i(L/L_0)^\vee [i] \otimes \det(L/L_0)^\ell \la \tfrac12 \dim \Gr^{-k} \ra [-k\ell] \\
&\cong \cP_{-k,\ell} \otimes \bigoplus_{i \ge 0} \Lambda^i(L/L_0)^\vee [i].
\end{align*}
Here to obtain the second isomorphism we use the Koszul resolution arising from the fact that $\{s \in L_0 \subset L\} \subset \{L_0 \subset L \ni s\} = \T_{-k}$ is the vanishing locus of the map $\O_{\T_{-k}} \to \O_{\T_{-k}} \otimes (L/L_0)$ given by $1 \mapsto s$. 
\end{proof}

\subsection{Exchange relations for $GL_n$}\label{sec:GLnmut} 
We now establish the existence of certain categorified exchange relations in $\KPGLn$. Specifically, we consider the exchange relations associated to the mutations of the purported cluster of (\ref{fig:cluster}). The most interesting of these is the mutation at $\cP_{n,0}$, which generalizes the computation performed in the $n=1$ case in Section~\ref{sec:abex}. 

\begin{Proposition}\label{prop:mutation1}
In $\KPGLn$ we have exact sequences
\begin{gather}
\label{eq:mut1}  0 \to \cP_{0,\det} * \cP_{n-1,\ell} \{-n\} \to \cP_{n,\ell} * (\cP_{0,\det} * \cP_{-1,-\ell}) \{n\} \to \cP_{n-1,\ell+1} \to 0 \\
\label{eq:mut2}  0 \to \cP_{n-1,\ell+1} \to (\cP_{-1,-\ell} * \cP_{0,\det}) * \cP_{n,\ell} \{-n\} \to \cP_{n-1,\ell} * \cP_{0,\det} \{n\} \to 0
\end{gather}
for any $\ell \in \Z$. 
\end{Proposition}
\begin{proof}
We consider the case $\ell=0$ as the general case is very similar. To compute $\cP_n \conv \cP_{-1}$ we consider the following diagram of Cartesian squares. 
\begin{equation*}
	\begin{tikzpicture}[baseline=(current  bounding  box.center),thick,>=\arrtip]
		\newcommand*{\ha}{3}; \newcommand*{\va}{-1.5};
		
		\node (aa) at (0,0) {$\hR_n^{\cl} \ttimes \hR_{-1}^{\cl}$};
		\node (ab) at (\ha,0) {$\convspace'_{n,-1}$};
		\node (ba) at (-\ha,\va) {$\hR_n$};
		\node (bb) at (0,\va) {$\hR_n \ttimes \hR_{-1}$};
		\node (bc) at (\ha,\va) {$\convspace_{n,-1}$};
		\node (bd) at (3*\ha, \va) {$\hR_{n-1}$};
		\node (ca) at (-\ha,2*\va) {$\T_n$};
		\node (cb) at (0,2*\va) {$\T_n \ttimes \hR_{-1}$};
		\node (cc) at (\ha,2*\va) {$\Gr_n \ttimes \hR_{-1}$};
		\node (cd) at (2*\ha, 2*\va) {$\Gr_n \ttimes \T_{-1}$};
		\node (ce) at (3*\ha,2*\va) {$\T_{n-1}$};
		
		\draw[<-] (aa) to node[above] {$\hat{d}$} (ab);
		\draw[->] (aa) to node[left] {$\cl$} (bb);
		\draw[->] (ab) to node[right] {$\cl'$} (bc);
		\draw[<-] (ba) to node[above] {$p$} (bb);
		\draw[<-] (bb) to node[above] {$\td$} (bc);
		\draw[->] (bc) to node[above] {$\tm$} (bd);
		\draw[->] (ba) to node[left] {$i_2$} (ca);
		\draw[->] (bb) to node[left] {$i_2 \ttimes \id$} (cb);
		\draw[->] (bc) to node[right] {$i$} (cc);
		\draw[->] (bd) to node[right] {$i_2$} (ce);
		\draw[->] (cb) to node[above] {$p$} (ca);
		\draw[<-] (cb) to node[above] {$d$} (cc);
		\draw[->] (cc) to node[above] {$\id \ttimes i_2$} (cd);
		\draw[->] (cd) to node[above] {$m$} (ce);
		
	\end{tikzpicture}
\end{equation*}
Here the bottom two rows are restrictions of (\ref{eq:conv}). We have
$$\O_{\hR_n^{\cl}} * \O_{\hR_{-1}^{\cl}} \cong \tm_* \td^* \cl_* (\O_{\hR_n^{\cl} \ttimes \hR_{-1}^{\cl}}) \cong \tm_* \cl'_* \hat{d}^*_{\cl} (\O_{\hR_n^{\cl} \ttimes \hR_{-1}^{\cl}}) \cong \tm_* \cl'_* (\O_{\convspace_{n,-1}^{\cl}}).$$
We also have
$$\convspace'_{n,-1} \cong \hR_n^{\cl} \times_{\T_n} (\Gr_n \ttimes \hR_{-1}^{\cl}) \cong \Gr_n \ttimes \hR_{-1}^{\cl} \cong \{tL_0=L_1 \overset{1}\subset L_2 \overset{n-1} \subset L_0, v \in tL_0 \}$$
since $\hR_n^{\cl} = \T_n$. Then the composition $\tm \circ \cl': \convspace'_{n,-1} \to \hR_{n-1}$ factors as
$$\convspace'_{n,-1} \hookrightarrow \hR_{n-1}^{\cl} \xrightarrow{\cl} \hR_{n-1}$$
where the first map is the obvious inclusion. It follows that $\O_{\hR_n^{\cl}} * \O_{\hR_{-1}^{\cl}} \cong \cl_*(\O_D)$ where
$$D \subset \hR_{n-1}^{\cl} = \{tL_0 \overset{1}\subset L \overset{n-1} \subset L_0, v \in L\}$$
is the divisor defined by requiring that $v \in tL_0$.

To describe this divisor consider the canonical section
$$\O_{\hR_{n-1}^{\cl}} \to \O_{\hR_{n-1}^{\cl}} \otimes (L/tL_0)$$
on $\hR_{n-1}^{\cl}$ induced by $v$. This section clearly vanishes along $D$ and thus we obtain the standard exact triangle
$$\O_{\hR_{n-1}^{\cl}} \otimes (L/tL_0)^\vee \to \O_{\hR_{n-1}^{\cl}} \to \O_D,$$
or equivalently
$$\O_{\hR_{n-1}^{\cl}} \to \O_{\hR_n^{\cl}} * \O_{\hR_{-1}^{\cl}} \to \O_{\hR_{n-1}^{\cl}} \otimes (L/tL_0)^\vee [1].$$
Twisting by $\det(L_0/tL_0)[-n]$ we obtain (\ref{eq:mut1}).

To compute $\cP_{-1} \conv \cP_n$ we consider the similar diagram
\begin{equation*}
	\begin{tikzpicture}[baseline=(current  bounding  box.center),thick,>=\arrtip]
		\newcommand*{\ha}{3}; \newcommand*{\va}{-1.5};
		
		\node (aa) at (0,0) {$\hR_{-1}^{\cl} \ttimes \hR_n^{\cl}$};
		\node (ab) at (\ha,0) {$\convspace'_{-1,n}$};
		\node (ba) at (-\ha,\va) {$\hR_{-1}$};
		\node (bb) at (0,\va) {$\hR_{-1} \ttimes \hR_n$};
		\node (bc) at (\ha,\va) {$\convspace_{-1,n}$};
		\node (bd) at (3*\ha, \va) {$\hR_{n-1}$};
		\node (ca) at (-\ha,2*\va) {$\T_{-1}$};
		\node (cb) at (0,2*\va) {$\T_{-1} \ttimes \hR_{n}$};
		\node (cc) at (\ha,2*\va) {$\Gr_{-1} \ttimes \hR_n$};
		\node (cd) at (2*\ha, 2*\va) {$\Gr_{-1} \ttimes \T_n$};
		\node (ce) at (3*\ha,2*\va) {$\T_{n-1}$};
		
		\draw[<-] (aa) to node[above] {$\hat{d}$} (ab);
		\draw[->] (aa) to node[left] {$\cl$} (bb);
		\draw[->] (ab) to node[right] {$\cl'$} (bc);
		\draw[<-] (ba) to node[above] {$p$} (bb);
		\draw[<-] (bb) to node[above] {$\td$} (bc);
		\draw[->] (bc) to node[above] {$\tm$} (bd);
		\draw[->] (ba) to node[left] {$i_2$} (ca);
		\draw[->] (bb) to node[left] {$i_2 \ttimes \id$} (cb);
		\draw[->] (bc) to node[right] {$i$} (cc);
		\draw[->] (bd) to node[right] {$i_2$} (ce);
		\draw[->] (cb) to node[above] {$p$} (ca);
		\draw[<-] (cb) to node[above] {$d$} (cc);
		\draw[->] (cc) to node[above] {$\id \ttimes i_2$} (cd);
		\draw[->] (cd) to node[above] {$m$} (ce);
		
	\end{tikzpicture}
\end{equation*}
of Cartesian squares. As above we have
$$\O_{\hR_{-1}^{\cl}} * \O_{\hR_n^{\cl}} \cong \tm_* \cl'_* (\O_{\convspace'_{-1,n}})$$
as well as
$$\convspace'_{-1,n} \cong \hR_{-1}^\cl \times_{\T_{-1}} (\Gr_{-1} \ttimes \hR_n^\cl).$$
This time, however, $\hR_{-1}^\cl \hookrightarrow \T_{-1}$ is a codimension one embedding of bundles. To understand this fiber product we use the following diagram of fiber products 
\begin{equation*}
	\begin{tikzpicture}[baseline=(current  bounding  box.center),thick,>=\arrtip]
		\newcommand*{\ha}{3}; \newcommand*{\va}{-1.5};
		
		\node (aa) at (0,0) {$\convspace'_{-1,n}$};
		\node (ab) at (\ha,0) {$\hR_{-1}^\cl \times_{\T_{-1}} \hR_{-1}^\cl$};
		\node (ac) at (2*\ha,0) {$\hR_{-1}^\cl$};
		\node (ba) at (0,\va) {$\Gr_{-1} \ttimes \hR_n^\cl$};
		\node (bb) at (\ha,\va) {$\hR_{-1}^\cl$};
		\node (bc) at (2*\ha,\va) {$\T_{-1}$};
		
		\draw[->] (aa) to node[above] {} (ab);
		\draw[->] (ab) to node[above] {} (ac);
		\draw[->] (aa) to node[left] {} (ba);
		\draw[->] (ab) to node[left] {} (bb);
		\draw[->] (ac) to node[right] {} (bc);
		\draw[->] (ba) to node[below] {} (bb);
		\draw[->] (bb) to node[below] {} (bc);
	\end{tikzpicture}
\end{equation*}
On $\hR_{-1}^\cl \times_{\T_{-1}} \hR_{-1}^\cl$ we have the triangle
$$\O_{\hR_{-1}^\cl} \otimes (L_1/L_0)^\vee [1] \to \O_{\hR_{-1}^\cl \times_{\T_{-1}} \hR_{-1}^\cl} \to \O_{\hR_{-1}^\cl},$$
where we identify $(\hR_{-1}^\cl \times_{\T_{-1}} \hR_{-1}^\cl)^\cl$ with $\hR_{-1}^\cl$. This follows since the normal bundle of $\hR_{-1}^\cl \subset \T_{-1}$ is the line bundle $L_1/L_0$. It further follows that on $\convspace'_{-1,n}$ we have the triangle
$$\O_{\Gr_{-1} \ttimes \hR_n^\cl} \otimes (L_1/L_0)^\vee [1] \to \O_{\convspace'_{-1,n}} \to \O_{\Gr_{-1} \ttimes \hR_n^\cl},$$
where we identify $(\convspace'_{-1,n})^\cl$ with $\Gr_{-1} \ttimes \hR_n^\cl$. Since
$$\Gr_{-1} \ttimes \hR_n^\cl \cong \hR^\cl_{n-1} = \{tL_0 \overset{1}\subset L \overset{n-1}\subset L_0, v \in L\},$$
applying $\tm_* \cl'_*$ leaves us with the triangle
$$\O_{\hR_{n-1}^\cl} \otimes (t^{-1}L/L_0)^\vee [1] \to \O_{\hR_{-1}^{\cl}} * \O_{\hR_n^{\cl}} \to \O_{\hR_{n-1}^\cl}.$$
Twisting by $\det(L_0/tL_0) [-n]\{2\}$ and noting that $t^{-1}L/L_0 \cong L/tL_0 \{2\}$, we recover (\ref{eq:mut2}).
\end{proof}

We also have the following exact sequences in $\KPGLn$, corresponding to mutations at the other vertices of (\ref{fig:cluster}). We omit the proof as it is essentially identical to that of \cite[Prop. 2.17]{CW1}, and only nontrivially involves the geometry of $\Gr$ rather than that of $\hR$. 

\begin{Proposition}\label{prop:oldmutations}
	In $\KPGLn$ we have exact sequences
	\begin{gather*}
		0 \to \cP_{k-1,\ell} * \cP_{k+1,\ell} \{-1\} \to \cP_{k,\ell+1} * \cP_{k,\ell-1} \{-2k\} \to \cP_{k,\ell} * \cP_{k,\ell} \to 0 \\
		0 \to \cP_{k,\ell} * \cP_{k,\ell} \to \cP_{k,\ell-1} * \cP_{k,\ell+1} \{2k\} \to \cP_{k+1,\ell} * \cP_{k-1,\ell} \{1\} \to 0
	\end{gather*}
	for any $k \in [1,n-1]$ and $\ell \in \Z$. 
\end{Proposition} 

Likewise, the following follows from a direct generalization of the proof of (the easy case~of) \cite[Prop. 2.10]{CW1}. 

\begin{Proposition}\label{prop:realsimples}
In $\KPGLn$ the objects $\cP_{k,\ell}$ and $\cP_{-k,\ell}$ are real
	for any $k \in [1,n]$ and $\ell \in \Z$. That is, $\cP_{k,\ell} \conv \cP_{k,\ell}$ and $\cP_{-k,\ell} \conv \cP_{-k,\ell}$ are also simple. 
\end{Proposition} 

\subsection{Cluster structure for $GL_2$} We now further specialize to the case $n = 2$, where we have the following result. 
\begin{Theorem}\label{thm:qcatthminthesec}
The category $\KPGLtwo$ is a quantum monoidal cluster categorification in which 
$(\{\cP_{1,1}, \cP_{1,0}, \cP_{2,0}, \cP_{0,\det}\}, L, \wt{B})$ is a quantum monoidal seed whose coefficient and exchange matrices are
\begin{equation*}
	L = \begin{pmatrix}
		0 & -2 & -2 & 2 \\
		2 & 0 & 0 & 2 \\
		2 & 0 & 0 & 4 \\
		-2 & -2 & -4 & 0
	\end{pmatrix}, \quad \wt{B} = \begin{pmatrix}
	0 & 2 & -1 \\
	-2 & 0 & 1 \\
	1 & -1 & 0 \\
	0 & -1 & 1 \\
\end{pmatrix}.
\end{equation*}
\end{Theorem}
The terminology used here is reviewed in \cite[Sec. 4.4]{CW1}, following \cite{KKKO18, BZ05}. The exchange matrix $\wt{B}$ is the adjacency matrix of the quiver from (\ref{fig:cluster}), and is of affine type $A_2^{(1)}$. The exchange graph of the $A_2^{(1)}$ cluster algebra is known to be the following: 
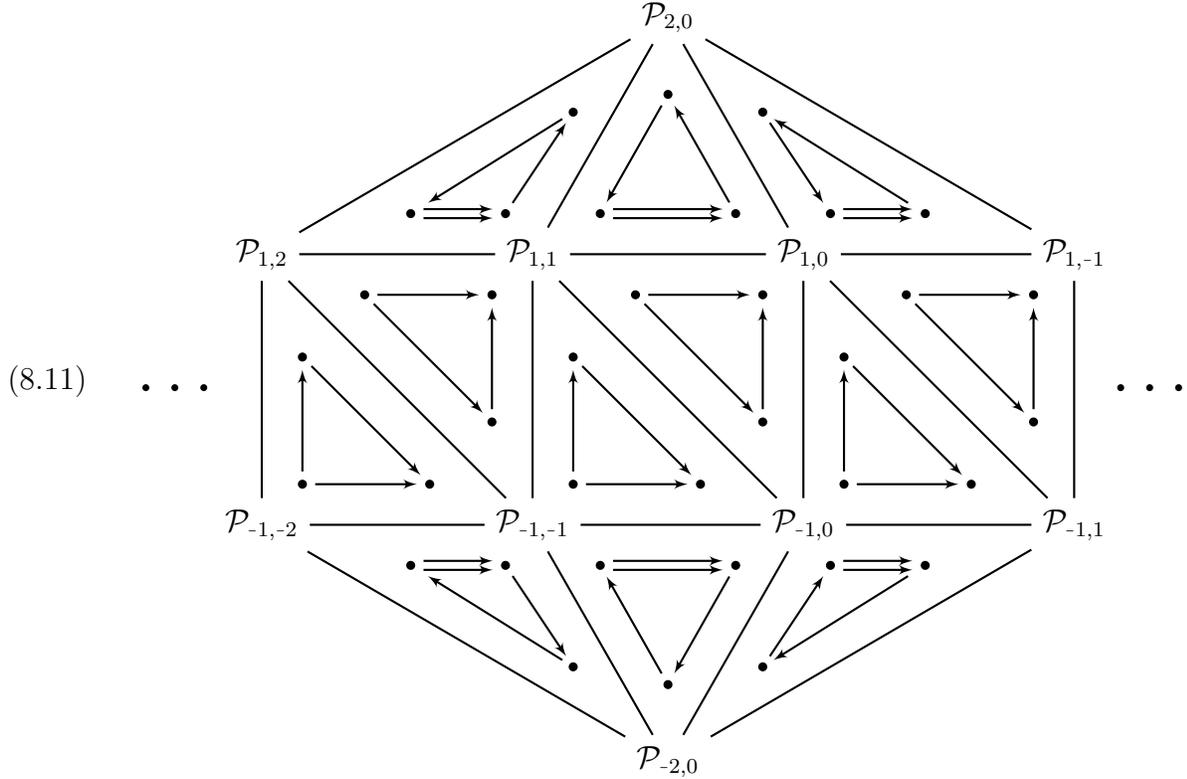
\begin{equation}\label{fig:GL2C2quiver}
\begin{tikzpicture}
[baseline=(current  bounding  box.center),thick,>=\arrtip]
\newcommand*{\ah}{4*.9}
\newcommand*{\av}{2*.9}
\newcommand*{\bv}{3.5*.9}
\node (aa) at (0,-\av-\bv) {$\cP_{-2,0}$};
\node (ba) at (-1.5*\ah,-\av) {$\cP_{-1,-2}$};
\node (bb) at (-.5*\ah,-\av) {$\cP_{-1,-1}$};
\node (bc) at (.5*\ah,-\av) {$\cP_{-1,0}$};
\node (bd) at (1.5*\ah,-\av) {$\cP_{-1,1}$};
\node (ca) at (-1.5*\ah,\av) {$\cP_{1,2}$};
\node (cb) at (-.5*\ah,\av) {$\cP_{1,1}$};
\node (cc) at (.5*\ah,\av) {$\cP_{1,0}$};
\node (cd) at (1.5*\ah,\av) {$\cP_{1,-1}$};
\node (da) at (0,\av+\bv) {$\cP_{2,0}$};

\draw[-] (aa) to (ba); \draw[-] (aa) to (bb); \draw[-] (aa) to (bc); \draw[-] (aa) to (bd);
\draw[-] (da) to (ca); \draw[-] (da) to (cb); \draw[-] (da) to (cc); \draw[-] (da) to (cd);
\draw[-] (ba) to (bb); \draw[-] (bb) to (bc); \draw[-] (bc) to (bd);
\draw[-] (ca) to (cb); \draw[-] (cb) to (cc); \draw[-] (cc) to (cd);
\draw[-] (ca) to (bb); \draw[-] (cb) to (bc); \draw[-] (cc) to (bd);
\draw[-] (ca) to (ba); \draw[-] (cb) to (bb); \draw[-] (cc) to (bc); \draw[-] (cd) to (bd);

\foreach \c in {ba, bb, bc} {
        \coordinate (bl) at ($(\c)+(.15*\ah,.3*\av)$);
        \coordinate (br) at ($(\c)+(.62*\ah,.3*\av)$);
        \coordinate (tl) at ($(\c)+(.15*\ah,1.24*\av)$);
        \fill (bl) circle (.06); \fill (br) circle (.06); \fill (tl) circle (.06);
        \draw[->,shorten <=1.7mm,shorten >=1.7mm] (bl) to (tl);
        \draw[->,shorten <=1.7mm,shorten >=1.7mm] (bl) to (br);
        \draw[->,shorten <=1.7mm,shorten >=1.7mm] (tl) to (br);
};

\foreach \c in {cb, cc, cd} {
\coordinate (bl) at ($(\c)+(-.15*\ah,-.3*\av)$);
\coordinate (br) at ($(\c)+(-.62*\ah,-.3*\av)$);
\coordinate (tl) at ($(\c)+(-.15*\ah,-1.24*\av)$);
\fill (bl) circle (.06); \fill (br) circle (.06); \fill (tl) circle (.06);
\draw[<-,shorten <=1.7mm,shorten >=1.7mm] (bl) to (tl);
\draw[<-,shorten <=1.7mm,shorten >=1.7mm] (bl) to (br);
\draw[<-,shorten <=1.7mm,shorten >=1.7mm] (tl) to (br);
};

\coordinate (vaa) at ($(ba)+(1.15*\ah,-.6*\bv)$);
\coordinate (vab) at ($(ba)+(.55*\ah,-.3*\av)$);
\coordinate (vac) at ($(ba)+(.9*\ah,-.3*\av)$);
\coordinate (vba) at ($(bb)+(.5*\ah,-.675*\bv)$);
\coordinate (vbb) at ($(bb)+(.25*\ah,-.3*\av)$);
\coordinate (vbc) at ($(bb)+(.75*\ah,-.3*\av)$);
\coordinate (vca) at ($(bd)+(-1.15*\ah,-.6*\bv)$);
\coordinate (vcb) at ($(bd)+(-.55*\ah,-.3*\av)$);
\coordinate (vcc) at ($(bd)+(-.9*\ah,-.3*\av)$);
\coordinate (vda) at ($(ca)+(1.15*\ah,.6*\bv)$);
\coordinate (vdb) at ($(ca)+(.55*\ah,.3*\av)$);
\coordinate (vdc) at ($(ca)+(.9*\ah,.3*\av)$);
\coordinate (vea) at ($(cb)+(.5*\ah,.675*\bv)$);
\coordinate (veb) at ($(cb)+(.25*\ah,.3*\av)$);
\coordinate (vec) at ($(cb)+(.75*\ah,.3*\av)$);
\coordinate (vfa) at ($(cd)+(-1.15*\ah,.6*\bv)$);
\coordinate (vfb) at ($(cd)+(-.55*\ah,.3*\av)$);
\coordinate (vfc) at ($(cd)+(-.9*\ah,.3*\av)$);

\foreach \va/\vb/\vc in {vaa/vab/vac, vda/vdb/vdc} {
\draw[->,shorten <=1.7mm,shorten >=1.7mm] (\vc) to (\va);
\draw[->,shorten <=1.7mm,shorten >=2.7mm] (\va) to (\vb);
\draw [->,shorten <=1.7mm,shorten >=1.7mm] ($(\vb)+(0,0.06)$) to ($(\vc)+(0,0.06)$);
\draw [->,shorten <=1.7mm,shorten >=1.7mm] ($(\vb)-(0,0.06)$) to ($(\vc)-(0,0.06)$);
};

\foreach \va/\vb/\vc in {vca/vcb/vcc, vfa/vfb/vfc} {
        \draw[<-,shorten <=1.7mm,shorten >=1.7mm] (\vc) to (\va);
        \draw[<-,shorten <=1.7mm,shorten >=2.7mm] (\va) to (\vb);
        \draw [<-,shorten <=1.7mm,shorten >=1.7mm] ($(\vb)+(0,0.06)$) to ($(\vc)+(0,0.06)$);
        \draw [<-,shorten <=1.7mm,shorten >=1.7mm] ($(\vb)-(0,0.06)$) to ($(\vc)-(0,0.06)$);
};

\foreach \va/\vb/\vc in {vba/vbb/vbc, vea/veb/vec} {
        \draw[->,shorten <=1.7mm,shorten >=1.7mm] (\vc) to (\va);
        \draw[->,shorten <=1.7mm,shorten >=1.7mm] (\va) to (\vb);
        \draw [->,shorten <=1.7mm,shorten >=1.7mm] ($(\vb)+(0,0.06)$) to ($(\vc)+(0,0.06)$);
        \draw [->,shorten <=1.7mm,shorten >=1.7mm] ($(\vb)-(0,0.06)$) to ($(\vc)-(0,0.06)$);
};

\foreach \c in {vaa, vab, vac, vba, vbb, vbc, vca, vcb, vcc, vda, vdb, vec, vea, veb, vdc, vfa, vfb, vfc} {\fill (\c) circle (.06);};

\node at (-1.8*\ah,0) {{\huge $\cdots$}};
\node at (1.8*\ah,0) {{\huge $\cdots$}};

\end{tikzpicture}
\end{equation}
Here we have inscribed the quivers associated to each cluster, and have labeled the vertices by the simple objects associated to them by Theorem \ref{thm:qcatthminthesec}. Given this exchange graph, the claim follows immediately from Propositions~\ref{prop:mutation1} and~\ref{prop:oldmutations} together with Propositions~\ref{prop:commute} and~\ref{prop:mutation3}. The above statement is somewhat abusive in that we do not prove $K_0(\KPGLtwo)$ is the relevant quantum cluster algebra (as opposed to a possibly larger algebra containing~it), but this question is not about the t-structure per se, and is addressed from a different perspective in \cite[Sec. 1.5]{FKRD18}. 

\begin{Proposition}\label{prop:commute}
In $\KPGLtwo$ we have isomorphisms
\begin{equation}\label{eq:GLtwocommute}
\begin{gathered}
 \cP_2 * \cP_{1,\ell} \cong \cP_{1,\ell} * \cP_2 \{-2\ell\},  \ \  \text{ and } \ \  \cP_{-2} * \cP_{-1,\ell} \cong \cP_{-1,\ell} * \cP_{-2} \{2\ell\}, \\
 \cP_{1,\ell} * \cP_{1,\ell+1} \cong \cP_{1,\ell+1} * \cP_{1,\ell} \{-2\}  \ \  \text{ and } \ \   \cP_{-1,\ell} * \cP_{-1,\ell+1} \cong \cP_{-1,\ell+1} * \cP_{-1,\ell} \{2\}, \\
  \cP_{1,\ell} * \cP_{-1,-\ell} \cong \cP_{-1,-\ell} * \cP_{1,\ell}  \ \  \text{ and } \ \   \cP_{1,\ell} * \cP_{-1,-\ell+1} \cong \cP_{-1,-\ell+1} * \cP_{1,\ell}.
\end{gathered}
\end{equation}
\end{Proposition}
\begin{proof}
The first two isomorphisms follow as in \cite[Prop. 2.10, Rem. 2.14]{CW1}, and as before we omit them. It remains to establish the last isomorphisms, which we prove when $\ell=0$ as the general case is similar (and also follows by tensoring with a suitable global line bundle). 

Consider the following diagram of Cartesian squares, where $\l^\vee := 2 \omega_1 - \omega_2$. 
\begin{equation*}
	\begin{tikzpicture}[baseline=(current  bounding  box.center),thick,>=\arrtip]
		\newcommand*{\ha}{3}; \newcommand*{\va}{-1.5};
		
		\node (aa) at (0,0) {$\hR_1^{\cl} \ttimes \hR_{-1}^{\cl}$};
		\node (ab) at (\ha,0) {$\convspace'_{1,-1}$};
		\node (ba) at (-\ha,\va) {$\hR_1$};
		\node (bb) at (0,\va) {$\hR_1 \ttimes \hR_{-1}$};
		\node (bc) at (\ha,\va) {$\convspace_{1,-1}$};
		\node (bd) at (3*\ha, \va) {$\hR_{\le \l^\vee}$};
		\node (ca) at (-\ha,2*\va) {$\T_1$};
		\node (cb) at (0,2*\va) {$\T_1 \ttimes \hR_{-1}$};
		\node (cc) at (\ha,2*\va) {$\Gr_1 \ttimes \hR_{-1}$};
		\node (cd) at (2*\ha, 2*\va) {$\Gr_1 \ttimes \T_{-1}$};
		\node (ce) at (3*\ha,2*\va) {$\T_{\le \l^\vee}$};
		
		\draw[<-] (aa) to node[above] {$\hat{d}$} (ab);
		\draw[->] (aa) to node[left] {$\cl$} (bb);
		\draw[->] (ab) to node[right] {$\cl'$} (bc);
		\draw[<-] (ba) to node[above] {$p$} (bb);
		\draw[<-] (bb) to node[above] {$\td$} (bc);
		\draw[->] (bc) to node[above] {$\tm$} (bd);
		\draw[->] (ba) to node[left] {$i_2$} (ca);
		\draw[->] (bb) to node[left] {$i_2 \ttimes \id$} (cb);
		\draw[->] (bc) to node[right] {$i$} (cc);
		\draw[->] (bd) to node[right] {$i_2$} (ce);
		\draw[->] (cb) to node[above] {$p$} (ca);
		\draw[<-] (cb) to node[above] {$d$} (cc);
		\draw[->] (cc) to node[above] {$\id \ttimes i_2$} (cd);
		\draw[->] (cd) to node[above] {$m$} (ce);
		
	\end{tikzpicture}
\end{equation*}
Since $\hR_1 = \T_1$ it follows that
$$\convspace'_{1,-1} \cong \Gr_1 \ttimes \hR_{-1}^\cl = \{(L_1,L_2,v): tL_0 \overset{1}\subset L_1 \overset{1}\subset L_0, L_1 \overset{1}\subset L_2 \overset{1}\subset t^{-1}L_1, v \in L_1\}.$$
It further follows that
$$\cP_1 * \cP_{-1} \cong \pi_*(\O_{\convspace'_{1,-1}})[1]\{-1\}$$
where $\pi := \tm \circ \cl'$ forgets $L_1$ and lands in $\hR^\cl_{\le \l^\vee}$. 

A similar argument now shows that 
$$\cP_{-1} * \cP_1 \cong \pi_*(\O_{\convspace'_{-1,1}}) [1]\{-1\}$$
where 
$$\convspace'_{-1,1} = \{(L_1,L_2,v): L_0 \overset{1}\subset L_1 \overset{1}\subset t^{-1}L_0, tL_1 \overset{1}\subset L_2 \overset{1}\subset L_1, v \in L_0 \cap L_2\}.$$
and $\pi$ forgets $L_1$. To relate these two note that $\convspace'_{-1,1}$ has two components $\convspace_1$ and $\convspace_2$, consisting of the locus where $L_0=L_2$ and $v \in tL_1$ respectively. The intersection $\convspace_1 \cap \convspace_2 \subset \convspace_1$ is carved out by the zero section of the tautological map 
$$v: \O_{\convspace_1} \to \O_{\convspace_1} \otimes (L_0/tL_1)$$
and subsequently gives us the standard exact triangle
$$\O_{\convspace_1} \otimes (L_0/tL_1)^\vee \to  \O_{\convspace'_{-1,1}} \to \O_{\convspace_2}.$$
Applying $\pi_*$ to this triangle we find that the left hand term vanishes since $\pi$ restricted to $\convspace_1$ is a $\P^1$ bundle and $L_0/tL_1$ restricts to $\O_{\P^1}(-1)$ on the fibers. Thus 
$$\cP_{-1} * \cP_1 \cong \pi_*( \O_{\convspace'_{-1,1}}) [1]\{-1\} \cong \pi_*(\O_{\convspace_2}) [1]\{-1\} \cong \pi_*(\O_{\convspace'_{1,-1}})[1]\{-1\} \cong \cP_{-1} * \cP_1,$$
where the second to last isomorphism is because we can identify $\convspace_2$ with $\convspace'_{1,-1}$. This established the bottom left isomorphism in (\ref{eq:GLtwocommute}).

Changing the line bundles in the argument above, we first find that 
\begin{align*}
\cP_1 * \cP_{-1,1} &\cong \pi_*(\O_{\convspace'_{1,-1}} \otimes (L_2/L_1))\{-1\} \\
\cP_{-1,1} * \cP_1 &\cong \pi_*(\O_{\convspace'_{-1,1}} \otimes (L_1/L_0)) \{-1\}.
\end{align*}
This time, to compute the latter pushforward, we note that the intersection $\convspace_1 \cap \convspace_2 \subset \convspace_2$ is the divisor consisting of the locus where $L_0=L_2$. This locus is carved out by the zero section of the map 
$$\O_{\convspace_2} \otimes (L_0/tL_1) \to \O_{\convspace_2} \otimes (L_1/L_2),$$
which subsequently gives us the standard exact triangle
\begin{equation}\label{eq:section}
\O_{\convspace_2} \otimes (L_0/tL_1) \otimes (L_1/L_2)^\vee \to \O_{\convspace'_{-1,1}} \to \O_{\convspace_1}.
\end{equation}
Tensoring with $(L_1/L_0) \{-1\}$ and applying $\pi_*$ it is the right hand term which vanishes this time. Thus we obtain
\begin{align*}
\cP_{-1,1} * \cP_1 
&\cong \pi_*(\O_{\convspace_2} \otimes (L_1/tL_1) \otimes (L_1/L_2)^\vee) \{-1\} \\
&\cong \pi_*(\O_{\convspace_2} \otimes (L_2/tL_1) \{-1\}) \\
&\cong \pi_*(\O_{\convspace'_{1,-1}} \otimes (L_2/L_1))\{-1\} \\&\cong \cP_1 * \cP_{-1,1}.\qedhere
\end{align*}
\end{proof}

\begin{Proposition}\label{prop:mutation3}
In $\KPGLtwo$ we have exact sequences
\begin{equation}\label{eq:mutGLtwo}
\begin{gathered}
0 \to \cP_{0,0} \{-1\} \to \cP_{1,1+\ell} * \cP_{-1,1-\ell} * \cP_{0,\det}^{-1} \to \cP_{1,\ell} * \cP_{-1,-\ell} \to 0 \\
0 \to \cP_{-1,-\ell} * \cP_{1,\ell} \to \cP_{0,\det}^{-1} * \cP_{-1,1-\ell} * \cP_{1,1+\ell} \to \cP_{0,0} \{1\} \to 0 \\
0 \to \cP_{1,1+\ell} * \cP_{-1,-\ell} \to \cP_{1,\ell} * \cP_{-1,-1-\ell} * \cP_{0,\det} \to \cP_{0,0} \{1\} \to 0 \\
0 \to \cP_{0,0} \{-1\} \to \cP_{0,\det} * \cP_{-1,-1-\ell} * \cP_{1,\ell} \to \cP_{-1,-\ell} * \cP_{1,1+\ell} \to 0 
\end{gathered}
\end{equation}
for any $\ell \in \Z$.
\end{Proposition}
\begin{proof}
We work out the case $\ell=0$, the general one again being similar. We also omit the derivation of the last two sequences, as they directly parallel those of the first two. Using the same argument and notation as in the proof of Proposition \ref{prop:commute}, we have 
\begin{align*}
\cP_1 * \cP_{-1} &\cong \pi_*(\O_{\convspace'_{1,-1}})[1]\{-1\} \\
\cP_{1,1} * \cP_{-1,1} * \cP_{0,\det}^{-1} &\cong \pi_*(\O_{\convspace'_{1,-1}} \otimes (L_0/L_1) \otimes (L_2/L_1) \otimes \det(L_2/tL_2)^\vee) [1]\{-3\}.
\end{align*}
To relate these two consider the divisor $D \subset \convspace'_{1,-1}$ consisting of the locus where $L_0=L_2$. We have the natural map of line bundles
$$t: \O_{\convspace'_{1,-1}} \otimes (L_0/L_1) \to \O_{\convspace'_{1,-1}} \otimes (L_1/tL_2) \{2\},$$
which vanishes precisely along $D$. It follows that on $\convspace'_{1,-1}$ we have the triangle
$$\O_{\convspace'_{1,-1}} \otimes (L_0/L_1) \otimes (L_2/L_1) \otimes \det(L_2/tL_2)^\vee \{-2\} \to \O_{\convspace'_{1,-1}} \to \O_D.$$
Shifting by $[1]\{-1\}$ and applying $\pi_*$ we arrive at the triangle
$$\cP_{1,1} * \cP_{-1,1} * \cP_{0,\det}^{-1} \to \cP_1 * \cP_{-1} \to \cP_{0,0} [1]\{-1\},$$
which after rotation gives the first sequence in (\ref{eq:mutGLtwo}). 

Similarly, we have 
\begin{align*}
\cP_{-1} * \cP_1 &\cong \pi_*(\O_{\convspace'_{-1,1}}) [1]\{-1\} \\
\cP_{-1,1} * \cP_{0,\det}^{-1} * \cP_{1,1} &\cong \pi_*(\O_{\convspace'_{-1,1}} \otimes (L_0/tL_1)^\vee \otimes (L_1/L_2)) [1]\{-3\}.
\end{align*}
This time we rewrite (\ref{eq:section}) as
$$\O_{\convspace_2} \to \O_{\convspace'_{-1,1}} \otimes (L_0/tL_1)^\vee \otimes (L_1/L_2)  \to \O_{\convspace_1} \otimes (L_1/L_0) \otimes (t^{-1}L_0/L_1)^\vee \{2\}.$$
Finally, we shift by $[1]\{-1\}$ and apply $\pi_*$. Using the observations that $\pi_*(\O_{\convspace_2})[1]\{-1\} \cong \cP_{-1} * \cP_1$, that $\pi_*(\O_{\convspace_1} \otimes (L_1/L_0) \otimes (t^{-1}L_0/L_1)^\vee) \cong \cP_{0,0} [-1]$, and that $\cP_{-1,1} * \cP_{0,\det}^{-1} \cong \cP_{0,\det}^{-1} * \cP_{-1,1} \{-2\}$, we obtain the second sequence in (\ref{eq:mutGLtwo}). 
\end{proof}

\subsection{Duality and the twist}\label{sec:adjoints} 
Given Theorem \ref{thm:adjoints}, the following result is an immediate consequence of Lemmas~\ref{lem:adjoint1} and \ref{lem:adjoint2}, which compute the action of $(-)^*$ and $\D_1$. 
\begin{Proposition}\label{prop:dualsforGLn}
In $\KPGLn$ we have
	\begin{align*}
		\cP_{k,\ell}^L \cong (\cP_{0,\det} \{k\} )^{*(-k)} * \cP_{-k,-\ell+n} , \ \ & \ \ 
		\cP_{k,\ell}^R \cong (\cP_{0,\det} \{k\} )^{*k} * \cP_{-k,-\ell-n+1} \\
		\cP_{-k,\ell}^L \cong (\cP_{0,\det} \{-k\})^{*k} * \cP_{k,-\ell-n+1} , \ \ & \ \ 
		\cP_{-k,\ell}^R \cong (\cP_{0,\det} \{-k\})^{*(-k)} * \cP_{k,-\ell+n} .
	\end{align*}
for all $k \in [1,n]$ and $\ell \in \Z$. 
\end{Proposition}

The most interesting feature of Proposition \ref{prop:dualsforGLn} is that, in the $n = 2$ case, it confirms the expectation that duality functors categorify the twist automorphism of the associated cluster algebra (see \cite[Sec. 6.3]{CW1} for a detailed discussion of this). This can be seen from (\ref{fig:GL2C2quiver}) as follows. The clusters in the middle rows have acyclic quivers, hence the action of the twist is given by a sequence of source mutations. Thus, for example, the cluster containing $\{\cP_{-1,-1}, \cP_{1,1}, \cP_{-1,0}\}$ is taken after three source mutations to the cluster containing $\{\cP_{1,0}, \cP_{-1,1}, , \cP_{1,-1}\}$. But by Proposition \ref{prop:dualsforGLn} the latter are exactly the left duals of the former, up to powers of the frozen variable $\cP_{0,\det}$. Moreover, since sheaves supported on the identity are manifestly closed under duals, it follows that for arbitrary $n$ the duality functors define a discrete integrable system on the K-theoretic Coulomb branch. 

\begin{Lemma}\label{lem:adjoint1}
In $\KPGLn$ we have $\cP_{k,\ell}^* \cong \cP_{-k,\ell}$ for all $k \in [1,n]$ and $\ell \in \Z$.
\end{Lemma}
\begin{proof}
	The argument (and result) closely parallels \cite[Lem. 3.8]{CW1}. By Proposition~\ref{prop:*-atlas} we have $\cP_{k,\ell}^* \cong s^*(\O_{\Gr^{-k}} \tbox \cP_{k,\ell})$, where $s: \hR_{-k} \to \Gr^{-k} \ttimes \hR_k$. We have a Cartesian square 
	\begin{equation*}
		\begin{tikzpicture}[baseline=(current  bounding  box.center),thick,>=\arrtip]
			\newcommand*{\ha}{3}; \newcommand*{\va}{-1.5};
			
			\node (aa) at (0,0) {$\hR_{-k}^{\cl}$};
			\node (ab) at (\ha,0) {$\Gr^{-k} \ttimes \hR_k^{\cl}$};
			\node (ba) at (0,\va) {$\hR_{-k}$};
			\node (bb) at (\ha,\va) {$\Gr^{-k} \ttimes \hR_k$};
			
			\draw[->] (aa) to node[above] {$s^\cl$} (ab);
			\draw[->] (aa) to node[left] {$\cl$} (ba);
			\draw[->] (ab) to node[right] {$\id \ttimes \cl$} (bb);
			\draw[->] (ba) to node[below] {$s$} (bb);
			
		\end{tikzpicture}
	\end{equation*}
	where the map $s^{\cl}$ is given by 
	$$(s \in L_0 \overset{k} \subset L) \mapsto ((L_0 \overset{k} \subset L), (s \in L_0 \overset{k} \subset L)).$$
It follows that
	\begin{align*}
		\cP_{k,\ell}^* 
		& \cong s^* (\id \ttimes \cl)_* (\O_{\Gr^{-k}} \tbox \cP_{k,\ell})\\ &\cong \cl_* (s^{\cl})^* (\O_{\Gr^{-k} \ttimes \hR_k^\cl} \otimes \det(L/L_0)^\ell) \la \tfrac12 \dim \Gr^k \ra [-k \ell] \\
		& \cong \cl_* (\O_{\hR^\cl_{-k}} \otimes \det(L/L_0)^\ell) \la \tfrac12 \dim \Gr^{-k} \ra [-k \ell],
	\end{align*}
	which by definition s $\cP_{-k,\ell}$. 
\end{proof}

\begin{Lemma}\label{lem:adjoint2}
In $\KPGLn$ we have 
	\begin{gather*}
		\D_1(\cP_{k,\ell})  \cong (\cP_{0,\det} \{-k\})^{*k} * \cP_{k,-\ell-n+1} \\
		\D_1(\cP_{-k,\ell})  \cong (\cP_{0,\det} \{k\})^{*(-k) } * \cP_{-k,-\ell+n} 
	\end{gather*}
for all $k \in [1,n]$ and $\ell \in \Z$.
\end{Lemma}
\begin{proof}
	Given the closed immersions $\hR_k^\cl \xrightarrow{\cl} \hR_k \xrightarrow{i_1} \Gr^k \times N_\O$ we have
	\begin{align*}
		\D_1(\cP_{k,\ell}) 
		& \cong \cHom(\cl_*(\O_{\hR_k^\cl} \otimes \det(L_0/L)^\ell) \la \tfrac12 k(n-k) \ra [-k\ell], i_1^! (\omega_{\Gr^k})) \\
		& \cong \cl_* \cHom(\O_{\hR_k^\cl} \otimes \det(L_0/L)^\ell, \cl^! i_1^! (\omega_{\Gr^k})) \la - \tfrac12 k(n-k) \ra [k\ell].
	\end{align*}
	The composition $i_1 \circ \cl$ is the inclusion 
	$$(s \in L \overset{k} \subset L_0) \mapsto (L \overset{k} \subset L_0, s \in L_0)$$
	which is of codimension $k$ with normal bundle $L_0/L$. Thus 
	$$\cl^! i_1^!(\omega_{\Gr^k}) \cong \cl^* i_1^* (\omega_{\Gr^k}) \otimes \det(L_0/L) [-k].$$
	We also have
	$$\omega_{\Gr^k} \cong \O_{\Gr^k} \otimes \det(L_0/L)^{-n} \otimes \det(L_0/tL_0)^k [k(n-k)],$$
	see \cite[Lem. 3.9]{CW1}. Putting these together we find
	\begin{align*}
		\D_1(\cP_{k,\ell}) 
		&\cong \cl_* (\O_{\hR_k^\cl} \otimes \det(L_0/L)^{-\ell-n+1} \otimes \det(L_0/tL_0)^k) \la \tfrac12 k(n-k) \ra \{k(n-k) \} [k\ell-k] \\
		&\cong \cP_{0,\det}^{*k} * \cP_{k,-\ell-n+1} \{-k^2\}.
	\end{align*}
	Computing $\D_1(\cP_{-k,\ell})$ is similar (in fact easier, since the normal bundle is trivial in this case) using the fact that 
	\begin{equation*}
\omega_{\Gr^{-k}} \cong \O_{\Gr^{-k}} \otimes \det(L/L_0)^n \otimes \det(t^{-1}L_0/L_0)^{-k} [k(n-k)].\qedhere
	\end{equation*}
\end{proof}

\bibliographystyle{amsalpha}
\bibliography{bibKoszul}

\end{document}